\documentclass[a4paper,pdftex]{article}
\usepackage{latexsym}
\usepackage{amsmath}
\usepackage{amssymb}
\usepackage{graphicx}
\usepackage{hhline}
\usepackage{ascmac}
\usepackage{mathrsfs}
\usepackage{cases}
\usepackage{bm}
\usepackage{amsthm}
\usepackage{comment}
\usepackage{authblk}
\usepackage{here}
\usepackage{cite}
\usepackage{mathtools}
\usepackage{amsbsy,amsfonts,amsmath,amssymb,amsthm}
\usepackage{array}
\usepackage{color}
\usepackage{enumerate}
\usepackage{mathrsfs}
\usepackage[hidelinks]{hyperref}
\usepackage{here}
\usepackage{cases}
\definecolor{wineRed}{rgb}{0.7,0,0.3}
\definecolor{grandBleu}{rgb}{0,0,0.8}
\definecolor{darkGreen}{rgb}{0,0.4,0}
\definecolor{blueViolet}{rgb}{0.4,0,1.0}
\definecolor{bloodOrange}{rgb}{0.85,0.05,0}
\definecolor{mycolor}{rgb}{0.8,0,0.2}
\definecolor{}{rgb}{0.8,0,0.2}
\usepackage[textsize=small]{todonotes}
\usepackage{cite}

\newcommand{\KS}[1]{{#1}}

\newcommand{\MO}[1]{{#1}}

\DeclareMathAlphabet{\mathpzc}{OT1}{pzc}{m}{it}

\numberwithin{equation}{section}

\theoremstyle{plain}
 
\newtheorem{lem}{Lemma}%[section]

\newtheorem{mainthm}{Main Theorem}%[section]

\theoremstyle{definition}
\newtheorem{defn}{Definition}%[section]
\newtheorem{rem}{Remark}

\def\N{\mathbb{N}}
\def\R{\mathbb{R}}

\def\ds{\displaystyle}

\providecommand{\keywords}[1]
{
  \small	
  \textbf{Key words:} #1
}
\providecommand{\subclass}[1]
{
  \small	
  \textbf{AMS Subject Classification:} #1
}

\title{Structure-preserving scheme for 1D KWC system}

\author[1]{Makoto Okumura\thanks{okumura@konan-u.ac.jp}}
\author[2]{Shodai Kubota\thanks{skubota@cc.miyakonojo-nct.ac.jp}}
\author[3]{Ken Shirakawa\thanks{sirakawa@faculty.chiba-u.jp}}

\affil[1]{Department of Intelligence and Informatics, Konan University, 6--1 Nishiokamoto, Higashinada-Ku, Kobe-Shi, Hyogo, 658-0073, Japan}
\affil[2]{Department of General Education, National Institute of Technology, Miyakonojo College, 473--1, Yoshio-cho, Miyakonojo City, Miyazaki Prefecture, 885-8567, Japan}
\affil[3]{Department of Mathematics, Faculty of Education, Chiba University, 1--33, Yayoi-cho, Inage-ku, 263--8522, Chiba, Japan}

\date{}

\DeclareMathOperator{\diag}{\mathrm{diag}}

\begin{document}

\maketitle

% \vspace{-6ex}
% \noindent
% {\bf Abstract.}
% In this paper, we consider a system of one-dimensional parabolic PDEs, known as the KWC system, as a phase-field model for grain boundary motion. 
% A key feature of this system is that the equation for the crystalline orientation angle is described as a quasilinear diffusion equation with variable mobility. 
% The goal of this paper is to establish a structure-preserving numerical scheme for the system, focusing on two main structural properties: $\sharp\,1)$ range preservation; and $\sharp\,2)$ energy dissipation. 
% Under suitable assumptions, we construct a structure-preserving numerical scheme and address the following in the main theorems: (O) verification of the structural properties; (I) clarification of the convergence conditions; and (II) error estimate for the scheme.
% \newpage

\begin{abstract}
	In this paper, we consider a system of one-dimensional parabolic PDEs, known as the KWC system, as a phase-field model for grain boundary motion. 
A key feature of this system is that the equation for the crystalline orientation angle is described as a quasilinear diffusion equation with variable mobility. 
The goal of this paper is to establish a structure-preserving numerical scheme for the system, focusing on two main structural properties: $\sharp\,1)$ range preservation; and $\sharp\,2)$ energy dissipation. 
Under suitable assumptions, we construct a structure-preserving numerical scheme and address the following in the main theorems: (O) verification of the structural properties; (I) clarification of the convergence conditions; and (II) error estimate for the scheme.
\end{abstract}
\\
\subclass{
	35A15, %Variational methods applied to PDEs
	35K20, %Initial-boundary value problems for second-order parabolic equations
	65M12, %Stability and convergence of numerical methods for initial value and initial-boundary value problems involving PDEs
	65M06, %Finite difference methods for initial value and initial-boundary value problems involving PDEs
	74N20. %Dynamics of phase boundaries in solids
}

\section*{Introduction}
Let $ (0, T) $ be a time-interval with a constant $ 0 < T < \infty $, and let $\Omega := (0, 1) \subset \R $ be a one-dimensional bounded domain with the boundary $ \Gamma := \partial \Omega =  \{ 0, 1 \} $. Additionally, we set $ Q := (0, T) \times \Omega $ and $ \Sigma := (0, T) \times \Gamma $. 
\medskip

Let $ \varepsilon \in (0, 1) $ be a fixed positive constant. In this paper, we consider the following system of parabolic PDEs, denoted by (S)$_{\varepsilon}$:
\medskip

%\noindent
~~~~\hypertarget{(S)}{(S)$_\varepsilon$}
%\vspace{-0ex}
\begin{equation}\label{1}
\left\{ \parbox{11.5cm}{
    $ \partial_{t} \eta - \kappa_0^2 \partial_x^2 \eta +c (\eta -1) +\kappa \eta \sqrt{\varepsilon^2 +|\partial_x \theta|^2} = 0 $ \quad in $ Q $,
\\[1ex]
    $ (-1)^{\ell -1} \partial_x \eta(t, \ell) = 0, \quad (t,\ell) \in \Sigma$,
\\[1ex]
$ \eta(0, x) = \eta_{0}(x),$ \quad $ x \in \Omega $;
}\right. 
\end{equation}
\begin{equation}\label{2}
\left\{\parbox{11.5cm}{
    $ \displaystyle \alpha_{0}(t, x) \partial_t \theta -\partial_x \! \left( \! \kappa \!\left( \frac{\eta^2}{2} +\delta_0 \right) \frac{\partial_x \theta}{\sqrt{\varepsilon^2 +|\partial_x \theta|^2}} +\nu^2 \partial_x \theta  \! \right) = 0 $ 
    \quad in $Q$,
\\[1ex]
    $ (-1)^{\ell -1} \partial_x \theta(t, \ell) = 0, \quad (t,\ell) \in \Sigma$,
    \\[1ex]
    $ \theta(0, x) = \theta_{0}(x),$ \quad $x \in \Omega $.
}\right. 
\end{equation}
This system is based on the Kobayashi--Warren--Carter system, which was originally proposed by Kobayashi--Warren--Carter \cite{MR1752970,MR1794359}, as a possible phase-field model of planar grain boundary motion. From these works, we can see that the system (S)$_\varepsilon$ is governed by the following \emph{free-energy,} denoted by $ \mathcal{F}_\varepsilon $:
\begin{align}\label{freeEnergyOrg}
    \mathcal{F}_\varepsilon : [{\eta}, {\theta}] \in & \KS{[H^1 (\Omega)]^2} \mapsto \mathcal{F}_\varepsilon({\eta}, {\theta}) := \frac{\kappa_0^2}{2} \int_\Omega |\partial_x {\eta}|^2 \, dx +\frac{c}{2}\int_\Omega (\eta -1)^2 \, dx
    \nonumber
    \\
    & +\kappa \int_\Omega \left( \frac{\eta^2}{2} +\delta_0 \right) \sqrt{\varepsilon^2 +|\partial_x \theta|^2} \, dx +\frac{\nu^2}{2} \int_\Omega |\partial_x {\theta}|^2 \, dx  \in [0, \infty].
\end{align}
The polycrystalline micro-structure, including the grain boundary, is described by a vector field $ \eta \bigl[ \cos \theta, \sin \theta \bigr] $ in $ Q $, consisting of two unknown variables $ \eta $ and $ \theta $. 
In this context, $ \eta = \eta(t, x) $ and $ \theta = \theta(t, x) $ are  supposed to be the order parameters of the orientation order and orientation angle of crystallization, respectively. 
The constant $ \kappa_0 > 0 $ represents the spatial diffusion of $ \eta $, while the constant $ c > 0 $ is associated with a perturbation to control the range of $ \eta $. 
Additionally, $ \alpha_0 = \alpha_0(t, x) $ is a fixed positive-valued Lipschitz function, $ \delta_0 \in (0, 1) $, $ \kappa > 0 $, and $ \nu > 0 $ are fixed constants to characterize the phase-change of crystalline orientation.
%\medskip

    From the practical viewpoint, the mathematical model (S)$_\varepsilon$ is required to satisfy the following structural properties:
    \begin{description}
        \item[{\boldmath (a) (the range-preserving)}]$ 0 \leq \eta \leq 1 $ and $ \inf \theta_0(\Omega) \leq \theta \leq \sup \theta_0(\Omega) $ over $ Q $;
        \item[{\boldmath (b) (the energy-dissipation)}]$ \mathcal{F}_\varepsilon(\eta(t), \theta(t)) \leq \mathcal{F}_\varepsilon(\eta(s), \theta(s)) $ for all $ 0 \leq s \leq t \leq T $.
    \end{description}
    These properties have been principal issues in mathematical analysis (cf. \cite{MR4218112,MR4395725,MR3082861,MR3203495,MR3268865,MR3670006,MR4177183}), and indeed, some previous works (cf. \cite{MR4218112,MR4395725,MR3268865,MR3670006,MR4177183}) have established the time-discrete approximation scheme fulfilling (a) and (b). However, from scientific viewpoint, the line of works have a lot of remaining issues in numerics,  such as numerical experiment, convergence rate, error estimate, and so on.  
\medskip

In view of these, we set the goal of this paper to establish an effective structure-preserving scheme for the KWC system (S)$_\varepsilon$. Consequently, the Main Theorems concerned with:
\begin{description}
    \item[\textmd{(O)}]verification of the structural properties  (a) and (b);
        % \vspace{-2ex}
    \item[\textmd{(\,I\,)}]clarification of the convergence condition of the scheme;
        % \vspace{-2ex}
    \item[\textmd{(II)}]error estimate for the scheme;
\end{description}
will be demonstrated as evidence backing the effectiveness of our scheme. 
\medskip

The content of this paper is as follows. 
The Main Theorems will be stated in Section 2 on the basis of the preliminaries prepared in Section 1. 
In Section 3, we will give the proofs of Main Theorems.

\section{Preliminaries} 
We begin by prescribing the notations used throughout this paper. 
\medskip

\noindent
\underline{\textbf{\textit{Basic notations.}}} 
For arbitrary $ r_0 $, $ s_0 \in [-\infty, \infty]$, we define:
\begin{equation*}
r_0 \vee s_0 := \max\{r_0, s_0 \}, 
\end{equation*}
and in particular, we set:
\begin{equation*}
    [r]^+ := r \vee 0. 
\end{equation*}    
For any dimension $ k \in \N $ and any set $ A $, we write:
\begin{equation*}
    [A]^k := \overbrace{A \times \cdots \times A}^{\mbox{$k$ times}}.
\end{equation*}

\noindent
\underline{\textbf{\textit{Abstract notations.}}}
For an abstract Banach space $ E $, we denote by $ |\cdot|_{E} $ the norm of $ E $, and denote by $ \langle \cdot, \cdot \rangle_E $ the duality pairing between $ E $ and its dual $ E^* $. Let $I_d : E \longrightarrow E $ be the identity map from $E$ onto $E$. In particular, when $ H $ is a Hilbert space, we denote by $ (\cdot,\cdot)_{H} $ the inner product of $ H $. 

For a proper functional $ \Psi : E \longrightarrow (-\infty, \infty] $ on a Banach space $ E $, we denote by $ D(\Psi) $ the domain of $ \Psi $, i.e., 
    $ D(\Psi) := \left\{ \begin{array}{l|l}
        z \in E & \Psi(z) < \infty
    \end{array} \right\} $.

For Banach spaces $ E_1, \dots, E_m $, with $ 1 < m \in \mathbb{N} $, let $ E_1 \times \dots \times E_m $ be the product Banach space endowed with the norm $ |\cdot|_{E_1 \times \cdots \times E_m} := |\cdot|_{E_1} + \cdots +|\cdot|_{E_m} $. However, when all $ H_1, \dots, H_m $ are Hilbert spaces, $ H_1 \times \dots \times H_m $ denotes the product Hilbert space endowed with the inner product $ (\cdot, \cdot)_{H_1 \times \cdots \times H_m} := (\cdot, \cdot)_{H_1} + \cdots +(\cdot, \cdot)_{H_m} $ and the norm $ |\cdot|_{H_1 \times \cdots \times H_m} := \bigl( |\cdot|_{H_1}^2 + \cdots +|\cdot|_{H_m}^2 \bigr)^{\frac{1}{2}} $. 
% \pagebreak
\medskip

\noindent 
\underline{\textbf{\textit{Basic notations of discrete operators.}}} 
Let $\Delta x$ be a space mesh size, i.e., $\Delta x := 1/K$, where $K$ is a positive integer. 
\begin{defn}[Average operators]
We define average operators $\mu_{k}^{+}$, $\mu_{k}^{-}$ concerning subscript $k$ by
\begin{equation*}
\mu_{k}^{+}f_{k} := \frac{f_{k} + f_{k+1}}{2}, \quad \mu_{k}^{-}f_{k} := \frac{f_{k} + f_{k-1}}{2}, 
\end{equation*}
for $\{f_{k}\}_{k=-s}^{K+s} \in \mathbb{R}^{K+1+2s}$, where $s \in \mathbb{N} \cup \{0\}$. 
\end{defn}

\begin{defn}[Difference operators] 
Let us define the difference operators $\delta_{k}^{+}$, $\delta_{k}^{-}$, $\delta_{k}^{\langle 1 \rangle}$, and $\delta_{k}^{\langle 2 \rangle}$ concerning subscript $k$ by 
\begin{align*}
\delta_{k}^{+}f_{k} := \frac{f_{k+1} - f_{k}}{\Delta x}, & & & \delta_{k}^{-}f_{k} := \frac{f_{k} - f_{k-1}}{\Delta x}, \\[-1pt]
\delta_{k}^{\langle 1 \rangle}f_{k} := \frac{f_{k+1} - f_{k-1}}{2\Delta x}, & & & \delta_{k}^{\langle 2 \rangle}f_{k} := \frac{f_{k+1} - 2f_{k} + f_{k-1}}{(\Delta x)^2},  
\end{align*}
for $\{f_{k}\}_{k=-s}^{K+s} \in \mathbb{R}^{K+1+2s}$, where $s \in \mathbb{N} \cup \{0\}$. 
\end{defn}

\begin{defn}[Summation operator] 
As a discretization of the integral, we define the summation operator $\sum _ {k=0}^{K}{}^{\prime \prime}$: $\mathbb{R}^{K+1+2s} \to \mathbb{R}$ by 
\begin{equation*}
\sum _ {k=0}^{K}{}^{\prime \prime} f_{k} := \frac{1}{2}f_{0} +\sum_{k=1}^{K-1}f_{k} +\frac{1}{2}f_{K} \quad \mbox{for\ } \{f_{k}\}_{k=-s}^{K+s} \in \mathbb{R}^{K+1+2s},\ \mbox{where\ } s \in \mathbb{N} \cup \{0\}. 
\end{equation*}
\end{defn}

\begin{defn}[Discrete norms]
We define the discrete $L^{\infty}$-norm $\|\cdot\|_{L_{\rm d}^{\infty}}$ and the discrete $L^{2}$-norm $\|\cdot\|_{L_{\rm d}^{2}}$ by 
\begin{equation*}
	\|\bm{f}\|_{L_{\rm d}^{\infty}} := \max_{k=0,1,\ldots,K}|f_{k}|, \quad 
    \|\bm{f}\|_{L_{\rm d}^{2}} := \sqrt{ \sum _ {k=0}^{K}{}^{\prime \prime}|f_{k}|^{2}\Delta x}
\end{equation*}
for $\bm{f} = \{f_{k}\}_{k=0}^{K} \in \mathbb{R}^{K+1}$. 
For $\bm{f} = \{f_{k}\}_{k=0}^{K} \in \mathbb{R}^{K+1}$, we define the discrete Dirichlet semi-norm $\|D\bm{f}\|$ of $\bm{f}$ by 
\begin{equation*}
\| D\bm{f}\| := \sqrt{ \sum _ {k=0}^{K-1}|\delta_{k}^{+}f_{k}|^{2}\Delta x}.  
\end{equation*}
\end{defn}

\begin{defn}[Difference quotient]
    For a function $F \in C^{1}(\Omega)$ and arbitrary $u, v \in \Omega$, the difference quotient $d F/d (u, v)$ of $F$ at $(u, v)$ is defined as follows: 
    \begin{numcases}
        {\frac{dF}{d (u, v)} :=}
        \frac{F(u) - F(v)}{u - v} & ($u \neq v$), \nonumber\\
        F'(v) & ($u = v$). \nonumber
    \end{numcases}
\end{defn}

\medskip

\section{Main Theorems}

We begin by setting up the assumptions and notations needed in our Main Theorems.
\begin{enumerate}
    \item[(A1)]
        \begin{itemize}
            \item $\varepsilon \in (0, 1)$, $ \delta_0 \in (0, 1) $, $ c > 0 $, $ \kappa_0 > 0 $, $ \kappa > 0 $, and $ \nu > 0 $ are fixed constants. 
                %Also, $ \alpha_0 \in W^{1, \infty}(Q) $ is a fixed function, such that $ \inf \alpha_0(Q) \geq \delta_0 $.
                Also, $ \alpha_0 : Q \longrightarrow \mathbb{R} $ is a fixed Lipschitz continuous function with a Lipschitz constant $ L_{\alpha_0} \geq 0 $, such that $ \delta_0 := \inf \alpha_0(Q) > 0 $. 
            \item $ \gamma_\varepsilon \in C^\infty(\R) $ is a function defined by $ \gamma_\varepsilon(\vartheta) := \sqrt{\varepsilon^2 +|\vartheta|^2} $ with its derivative $ \gamma_\varepsilon'(\vartheta) := \vartheta/\sqrt{\varepsilon^2 +|\vartheta|^2} $, for all $ \vartheta \in \R $.
            \item $ \alpha : \mathbb{R} \longrightarrow (0, \infty) $ is a function defined by $\alpha(\eta) := (\eta^{2}/2) + \delta_{0}$ with its derivative $ \alpha'(\eta) = \eta $, for all $\eta \in \R$. 
            \item $ [\eta_0, \theta_0] \in \KS{[H^1(\Omega)]^2} $ is a pair of initial data, such that $ 0 \leq \eta_0(x) \leq 1  $ for a.e. $ x \in \Omega $. 
        \end{itemize}
\end{enumerate}
With these assumptions, we define the solution to the system (S)$_\varepsilon$ as follows. 
\begin{defn}\label{DefOfSolToSys}
    A pair of functions $ [\eta, \theta] \in L^2(0, T; [L^2(\Omega)]^2) $ is called a solution to the system (S)$_\varepsilon$, iff. $ [\eta, \theta] $ fulfills the following items. 
    \begin{enumerate}[({S}1)]
        \item $ [\eta, \theta] \in W^{1, 2}(0, T; [L^2(\Omega)]^2) \cap L^\infty(0, T; \KS{[H^1(\Omega)]^2})$, and $ [\eta(0), \theta(0)] = [\eta_0, \theta_0] $ in $ [L^2(\Omega)]^2 $.
        \item $ \ds \int_\Omega \bigl( \partial_t \eta(t) +c(\eta(t) -1) +\kappa \alpha'(\eta(t)) \gamma_\varepsilon(\partial_x \theta(t)) \bigr) \varphi  \, dx +\kappa_0^2 \int_\Omega \partial_x \eta(t) \partial_x \varphi \, dx = 0 $, 
            \\ 
            \hspace*{4ex}for all $ \varphi \in H^1(\Omega) $.
        \item $ \ds \int_\Omega  \alpha_0(t) \partial_t \theta(t)  \psi \, dx +\int_\Omega \bigl( \kappa \alpha(\eta(t)) \gamma_\varepsilon'(\partial_x \theta(t)) +\nu^2 \partial_x \theta(t) \bigr) \partial_x \psi(t) \, dx = 0 $, 
            \\
            \hspace*{4ex}for all $ \psi \in \MO{H^1}(\Omega)$.
    \end{enumerate}
\end{defn}

Next, we present the numerical scheme designed to approximate the solution of the governing equations in system (S)$_\varepsilon$. 
The scheme is derived using a finite difference approach to discretize both space and time. 
In this scheme, the time step size is denoted by $ \Delta t $, and the unknowns, $H_k^{(j)}$ and $\Theta_k^{(j)}$, represent the discretized state variables $ \eta $ and $ \theta $ at time step $j \in \{0, 1, \dots \}$.

\subsection{Proposed scheme}
The concrete form of our scheme is as follows: for $j=0,1,\ldots$,
\begin{gather}
\begin{split}
\frac{ H_{k}^{(j+1)} - H_{k}^{(j)} }{\Delta t} 
	= & \kappa_{0}^{2}\delta_{k}^{\langle 2 \rangle}H_{k}^{(j+1)} - c\left(H_{k}^{(j+1)} - 1 \right) \\
	& - \kappa H_{k}^{(j+1)}\frac{ \gamma_{\varepsilon}\bigl( \delta_{k}^{+}\Theta_{k}^{(j)}\bigr) + \gamma_{\varepsilon}\bigl( \delta_{k}^{-}\Theta_{k}^{(j)}\bigr)  }{2} \quad (k=0, \ldots, K), 
\end{split} \label{sc1}\\[-2pt]
\begin{split}
\alpha_{0, k}^{(j+1)} \frac{ \Theta_{k}^{(j+1)} \! - \! \Theta_{k}^{(j)} }{\Delta t} 
	= & \frac{\kappa}{2}\!\left\{ \delta_{k}^{+}\!\left(\! \alpha\!\left(\! H_{k}^{(j+1)} \!\right)\!\gamma_{\varepsilon}' (\delta_{k}^{-}\Theta_{k}^{(j+1)}) \!\right) \! + \! \delta_{k}^{-}\!\left(\! \alpha\!\left(\! H_{k}^{(j+1)} \!\right)\! \gamma_{\varepsilon}' (\delta_{k}^{+}\Theta_{k}^{(j+1)}) \!\right) \!\right\} \\
	& + \nu^{2}\delta_{k}^{\langle 2 \rangle}\Theta_{k}^{(j+1)} \quad (k=0, \ldots, K), 
\end{split} \label{sc2}%\\[-2pt]
% \delta_{k}^{\langle 1 \rangle}H_{k}^{(j)} = 0 \quad (k=0, K), \label{bc1}\\[-2pt]
% \delta_{k}^{\langle 1 \rangle}\Theta_{k}^{(j)} = 0 \quad (k=0, K). \label{bc2} 
\end{gather}
\begin{gather}
	% \begin{split}
	% \frac{ H_{k}^{(j+1)} - H_{k}^{(j)} }{\Delta t} 
	% 	= & \kappa_{0}^{2}\delta_{k}^{\langle 2 \rangle}H_{k}^{(j+1)} - c\left(H_{k}^{(j+1)} - 1 \right) \\
	% 	& - \kappa H_{k}^{(j+1)}\frac{ \gamma_{\varepsilon}\bigl( \delta_{k}^{+}\Theta_{k}^{(j)}\bigr) + \gamma_{\varepsilon}\bigl( \delta_{k}^{-}\Theta_{k}^{(j)}\bigr)  }{2} \quad (k=0, \ldots, K), 
	% \end{split} \label{sc1}\\[-2pt]
	% \begin{split}
	% \alpha_{0, k}^{(j+1)} \frac{ \Theta_{k}^{(j+1)} \! - \! \Theta_{k}^{(j)} }{\Delta t} 
	% 	= & \frac{\kappa}{2}\!\left\{ \delta_{k}^{+}\!\left(\! \alpha\!\left(\! H_{k}^{(j+1)} \!\right)\!\gamma_{\varepsilon}' (\delta_{k}^{-}\Theta_{k}^{(j+1)}) \!\right) \! + \! \delta_{k}^{-}\!\left(\! \alpha\!\left(\! H_{k}^{(j+1)} \!\right)\! \gamma_{\varepsilon}' (\delta_{k}^{+}\Theta_{k}^{(j+1)}) \!\right) \!\right\} \\
	% 	& + \nu^{2}\delta_{k}^{\langle 2 \rangle}\Theta_{k}^{(j+1)} \quad (k=0, \ldots, K), 
	% \end{split} \label{sc2}\\[-2pt]
	\delta_{k}^{\langle 1 \rangle}H_{k}^{(j)} = 0 \quad (k=0, K), \label{bc1}\\[-2pt]
	\delta_{k}^{\langle 1 \rangle}\Theta_{k}^{(j)} = 0 \quad (k=0, K). \label{bc2} 
\end{gather}
The notations used in the above scheme are defined below. 

\begin{defn}
    \label{Def.TBA}
    Let $K$ be a natural number.
    We define $H_{k}^{(j)}$ and $\Theta_{k}^{(j)} \ (k \! = \! -1, 0, \ldots, K, K+1,\ j = 0,1,\ldots)$ to be the approximation to $\eta(t,x)$ and $\theta(t,x)$ at time $t = j\Delta t$ and location $x=k\Delta x$, respectively. 
    As a side note, 
    $\Delta t$ is a time mesh size. 
    They are also written in vector as 
    \begin{gather*}
        \bm{H}^{(j)} := (H_{0}^{(j)},\ldots,H_{K}^{(j)})^{\top} \quad \mbox{or} \quad \bm{H}^{(j)} := (H_{-1}^{(j)}, H_{0}^{(j)},\ldots,H_{K}^{(j)},H_{K+1}^{(j)})^{\top}, \\
        \bm{\Theta}^{(j)} := (\Theta_{0}^{(j)},\ldots,\Theta_{K}^{(j)})^{\top} \quad \mbox{or} \quad \bm{\Theta}^{(j)} := (\Theta_{-1}^{(j)}, \Theta_{0}^{(j)},\ldots,\Theta_{K}^{(j)},\Theta_{K+1}^{(j)})^{\top}. 
    \end{gather*}
    Guess the meanings of $\bm{H}^{(j)}$ and $\bm{\Theta}^{(j)}$ from the context. 
    Hereafter, in this paper, $\bm{H}^{(j)}$ and $\bm{\Theta}^{(j)}$ are obtained by the proposed scheme \eqref{sc1}--\eqref{bc2}. 
    $H_{-1}^{(j)}$, $H_{K+1}^{(j)}$, $\Theta_{-1}^{(j)}$, and $\Theta_{K+1}^{(j)}$ are artificial quantities and are determined by the imposed discrete boundary conditions \eqref{bc1}--\eqref{bc2}. 
    Namely, 
    \begin{gather*}
        H_{-1}^{(j)} = H_{1}^{(j)}, \quad H_{K+1}^{(j)} = H_{K-1}^{(j)} \quad (j = 0,1, \ldots), \\
        \Theta_{-1}^{(j)} = \Theta_{1}^{(j)}, \quad \Theta_{K+1}^{(j)} = \Theta_{K-1}^{(j)} \quad (j = 0,1, \ldots). 
    \end{gather*}
    Furthermore, we define $\alpha_{0, k}^{(j)}$ by $\alpha_{0, k}^{(j)} := \alpha_{0}(j\Delta t,k\Delta x) \ (k = 0, 1, \ldots, K,\ j = 0,1,\ldots)$. 
\end{defn}

Fix a natural number $N \in \mathbb{N}$. 
We compute $(\bm{H}^{(j)},\bm{\Theta}^{(j)})$ up to $j=N$ by our proposed scheme \eqref{sc1}--\eqref{bc2} and estimate the error between it and the solution to the problem \eqref{1}--\eqref{2} up to $T=N\Delta t$.
Moreover, 
    we extend the solution $\eta$ in $\overline{Q}$ to $\tilde{\eta}$ in $[0,T] \times [-\Delta x, 1 + \Delta x]$ as follows: 
\begin{equation}
    \tilde{\eta}(t,x) := \begin{cases}
    \displaystyle \eta(t,-x) 
    & (-\Delta x \leq x < 0), \\
    \eta(t,x) & (0 \leq x \leq 1), \\
    \displaystyle \eta(t,2 - x) 
    & (1 < x \leq 1 + \Delta x), 
    \end{cases} \label{def_ext}
\end{equation}
for $t \in [0,T]$. 
Similarly, an extension $\tilde{\theta}$ of $\theta$ is also defined.
From the direct calculation, we can check that the following properties hold: 
if $\partial_{t}^{l}\partial_{x}^{m}\eta, \partial_{t}^{l}\partial_{x}^{m}\theta \in C(\overline{Q}) \ (0 \leq l \leq 1,\ 0 \leq m \leq 2,\ l + m \leq 2)$, then 
\begin{gather*}
    \partial_{t}^{l}\partial_{x}^{m}\tilde{\eta}, \partial_{t}^{l}\partial_{x}^{m}\tilde{\theta} \in C([0,T] \times [-\Delta x,1 + \Delta x]) \quad (0 \leq l \leq 1,\ 0 \leq m \leq 2,\ l + m \leq 2), \\
    \tilde{\eta}(t,x) = \eta(t,x) \quad \mbox{for\ all\ } x \in [0,1], \\
    \tilde{\eta}(t,-\Delta x) = \tilde{\eta}(t,\Delta x), \quad \tilde{\eta}(t,1 + \Delta x) = \tilde{\eta}(t,1 - \Delta x), \\
    \tilde{\theta}(t,x) = \theta(t,x) \quad \mbox{for\ all\ } x \in [0,1], \\
    \tilde{\theta}(t,-\Delta x) = \tilde{\theta}(t,\Delta x), \quad \tilde{\theta}(t, 1 + \Delta x) = \tilde{\theta}(t, 1 - \Delta x). 
\end{gather*}

Let $H_{k}^{(0)} = \tilde{\eta}(0,k\Delta x), \Theta_{k}^{(0)} = \tilde{\theta}(0,k\Delta x) \ (k = -1,0,\ldots, K,K+1)$. 
In addition, we define the errors $e_{\eta,k}^{(j)}$ and $e_{\theta,k}^{(j)}$ by 
\begin{gather*}
e_{\eta,k}^{(j)} := H_{k}^{(j)} - \tilde{\eta}(j\Delta t, k\Delta x) \quad (k=-1,0,\ldots, K,K+1,\ j=0,1,\ldots, N), \\
e_{\theta,k}^{(j)} := \Theta_{k}^{(j)} - \tilde{\theta}(j\Delta t, k\Delta x) \quad (k = -1,0,\ldots, K,K + 1,\ j = 0,1,\ldots, N).
\end{gather*}
Additionally, we use the following vector notations: 
\begin{gather*}
    \bm{e}_{\eta}^{(j)} := (e_{\eta,-1}^{(j)}, e_{\eta,0}^{(j)},\ldots,e_{\eta,K}^{(j)},e_{\eta,K+1}^{(j)})^{\top} \in \mathbb{R}^{K+3}, \\ 
    \bm{e}_{\theta}^{(j)} := (e_{\theta,-1}^{(j)}, e_{\theta,0}^{(j)},\ldots,e_{\theta,K}^{(j)},e_{\theta,K+1}^{(j)})^{\top} \in \mathbb{R}^{K+3},
\end{gather*}
for $ j = 0, 1, \dots, N$.
For simplicity, we use the expression $\tilde{\eta}_{k}^{(j)} := \tilde{\eta}(j\Delta t, k\Delta x)$ and $\tilde{\theta}_{k}^{(j)} := \tilde{\theta}(j\Delta t, k\Delta x)$ from now on.  
Also, the expressions $\delta_{k}^{\ast}f_{l}$ means $\left. \delta_{k}^{\ast}f_{k} \right|_{k=l}$, where the symbol ``$\ast$" denotes $+$, $\langle 1 \rangle$, or $\langle 2 \rangle$. 

\begin{rem}
    We remark on the existence of a solution to the scheme \eqref{sc1}--\eqref{bc2}. 
    For any given $(\bm{H}^{(j)},\bm{\Theta}^{(j)})$, if $\Delta t$ satisfies 
    \begin{equation}
        \Delta t < \frac{\nu^{2}\varepsilon^{2}}{4\kappa^{2}\delta_{0}\left(\delta_{0} + \dfrac{1}{2}\right)^{2}}(\Delta x)^{2}, \label{exist_cond}
    \end{equation}
    then there exists a unique solution $(\bm{H}^{(j + 1)},\bm{\Theta}^{(j + 1)})$ satisfying \eqref{sc1}--\eqref{sc2} with \eqref{bc1}--\eqref{bc2}.
    
    More precisely, the existence and uniqueness of $\bm{H}^{(j + 1)}$ are unconditionally obtained from the regularity of the coefficient matrix of $\bm{H}^{(j + 1)}$ 
    since scheme \eqref{sc1} is linear and implicit (cf. Lemma \ref{exist1}). 
    On the other hand, the existence and uniqueness of $\bm{\Theta}^{(j + 1)}$ can be obtained by the method of Banach's fixed point theorem, e.g., as in the first author's results \cite{MO2018JIAM, MO2020DCDS}, 
    since scheme \eqref{sc2} is nonlinear and implicit (cf. Lemma \ref{exist2}). 
    Note that the condition \eqref{exist_cond} above is a sufficient condition for the existence and uniqueness of $\bm{\Theta}^{(j + 1)}$.
\end{rem}

\begin{rem}
    \eqref{exist_cond} is only a sufficient condition, and a weaker assumption, i.e., a coarser $\Delta t$, could still yield a solvability result.
\end{rem}

Henceforth, we discuss under the assumption of condition \eqref{exist_cond} above. 
Now, the goal of this paper is to prove the following Main Theorems.
\begin{mainthm}
    Let us assume (A1) and assume that {$ \bigl\{ (\bm{H}^{(j)},\bm{\Theta}^{(j)}) \bigr\}_{j = 0}^{N} $ is a sequence consisting of the solutions $ (\bm{H}^{(j)},\bm{\Theta}^{(j)}) $, $ j = 0, 1, \ldots, N $, to the proposed scheme \eqref{sc1}--\eqref{bc2}.} 
    Then, the followings hold: 
    \begin{enumerate}
        \item[(O)] 
            \begin{enumerate}
                \item[(a)]if the initial data satisfies that
        \begin{gather}
            0 \leq H_{k}^{(0)} \leq 1 \quad (k = 0,1,\ldots,K), \quad \left|\Theta_{k}^{(0)}\right| \leq \xi_{0} \quad (k = 0,1,\ldots,K) %~\mbox{ for $ k = 0,1,\ldots,K $,} 
            \notag\\
            \mbox{with a constant  $\xi_{0} > 0$ independent of $k$ and $j$,} \label{ass_bd}
        \end{gather}
        then it holds that:
        \begin{equation}
            0 \leq H_{k}^{(j)} \leq 1 \quad (k = 0,1,\ldots,K), \quad \left|\Theta_{k}^{(j)}\right| \leq \xi_{0} \quad (k = 0,1,\ldots,K). 
            \label{bound}
        \end{equation}
    \item[(b)] 
        we define discrete global energy $\mathscr{F}_{1, \rm d}$: $\mathbb{R}^{K+3} \times \mathbb{R}^{K+3} \to \mathbb{R}$ as follows: 
        \begin{align*}
        \mathscr{F}_{1, \rm d}(\bm{H}, \bm{\Theta}) \! 
            := & \frac{\kappa_{0}^{2}}{2}\sum_{k=0}^{K}{}^{\prime\prime}\frac{ \left| \delta_{k}^{+}H_{k} \right|^{2} + \left| \delta_{k}^{-}H_{k} \right|^{2} }{2}\Delta x + \frac{c}{2}\sum_{k=0}^{K}{}^{\prime\prime}\left| H_{k} - 1 \right|^{2} \Delta x  \\[-1pt]
            & + \! \frac{ \nu^{2} }{2}\!\sum_{k=0}^{K}\!{}^{\prime\prime}\frac{ \left| \delta_{k}^{+}\Theta_{k} \right|^{2} \!\! + \! \left| \delta_{k}^{-}\Theta_{k} \right|^{2} }{2}\Delta x 
            \\[-1pt]
            & + \! \kappa\!\sum_{k=0}^{K}\!{}^{\prime\prime}\!\alpha\!\left( H_{k} \right)\!\frac{ \gamma_{\varepsilon}( \delta_{k}^{+}\Theta_{k}) \! + \! \gamma_{\varepsilon}( \delta_{k}^{-}\Theta_{k}) }{2} \Delta x, \\
            & \qquad\qquad \mbox{for}\ (\bm{H}, \bm{\Theta}) = (\{H_{k}\}_{k = -1}^{K+1}, \{\Theta_{k}\}_{k = -1}^{K+1}) \in \mathbb{R}^{K+3} \times \mathbb{R}^{K+3},
        \end{align*}
        then, the following dissipative property holds: 
        \begin{align}
            & \frac{ \mathscr{F}_{1, \rm d}(\bm{H}^{(j+1)}, \bm{\Theta}^{(j+1)}) - \mathscr{F}_{1, \rm d}(\bm{H}^{(j)}, \bm{\Theta}^{(j)}) }{\Delta t} \nonumber\\
            \leq & - \sum_{k=0}^{K}{}^{\prime\prime}\left| \frac{ H_{k}^{(j+1)} - H_{k}^{(j)} }{\Delta t} \right|^{2}\Delta x - \sum_{k=0}^{K}{}^{\prime\prime}\alpha_{0,k}^{(j+1)} \left| \frac{ \Theta_{k}^{(j+1)} - \Theta_{k}^{(j)} }{\Delta t} \right|^{2}\Delta x \leq 0, 
            \label{d_en_ineq}
        \end{align}
        for $ j = 0, 1, \dots, N -1 $. 
            \end{enumerate}
        \item[(\,I\,)]In addition,  let us assume the following condition.
            \begin{enumerate}
                \item[(A2)] $\partial_{t}^{l}\partial_{x}^{m}\eta, \partial_{t}^{l}\partial_{x}^{m}\theta \in C(\overline{Q}) \ (0 \leq l \leq 1,\ 0 \leq m \leq 2,\ l + m \leq 2)$. 
                
            \end{enumerate}
            Under (A1) and (A2), if $\Delta t$ satisfies
        \begin{equation}
            \Delta t < \frac{1}{3a}, \label{dt_ass} 
        \end{equation}
        where $a := \max\{2(\kappa C_{1}/\nu )^{2} - (2\kappa\varepsilon + c), (L_{\alpha_{0}} + 1)/\delta_{0}\}$ 
            and also $L_{\alpha_{0}} \geq 0$ is the Lipschitz constant of $\alpha_{0}$ as defined in (A1), 
            and $C_{1}$ is the constant as defined in \eqref{eta_bound},  
        then it holds that: 
        \begin{gather}
            \sup_{0 \leq j \leq N} \bigl\| \bm{e}_\eta^{(j)} \bigr\|_{L_d^2} +\sup_{0 \leq j \leq N} \bigl\| \bm{e}_\theta^{(j)} \bigr\|_{L_d^2} \to 0, \mbox{ as $ \Delta x, \Delta t \to 0 $.}
            \label{conv}
        \end{gather}
    \end{enumerate}
\end{mainthm}
\begin{mainthm}
    In addition to (A1), (A2), and \eqref{dt_ass}, let us assume the following condition. 
    \begin{enumerate}
        \item[(A3)] $\partial_{t}\eta$, $\partial_{x}^{2}\eta$, $\partial_{t}\theta$, and $\partial_{x}^{2}\theta$ are H\"{o}lder continuous functions of order $\sigma \in (0,1)$.
    \end{enumerate}
    Then, the following item holds. 
    \begin{enumerate}
        \item[(II)]There exists a constant $C$ independent of $k$ and $j$ such that 
    \begin{gather}
        \left\|\bm{e}_{\eta}^{(j)}\right\|_{L_{\rm d}^{2}} 
        \leq C\bigl( (\Delta x)^{\sigma} + (\Delta t)^{\sigma} + \Delta t + \Delta x + (\Delta x)^{2} \bigr) \quad (j = 1, \ldots, N), 
        \label{e_err_est}\\
        \left\|\bm{e}_{\theta}^{(j)}\right\|_{L_{\rm d}^{2}}
        \leq C\bigl( (\Delta x)^{\sigma} + (\Delta t)^{\sigma} + \Delta t + \Delta x + (\Delta x)^{2} \bigr) \quad (j = 1, \ldots, N). 
        \label{t_err_est}
    \end{gather}
    \end{enumerate}
\end{mainthm}

\begin{rem}
    In other words, from \eqref{e_err_est} and \eqref{t_err_est}, the convergence rates of the solutions to our proposed scheme \eqref{sc1}--\eqref{bc2} are $O((\Delta x)^{\sigma} + (\Delta t)^{\sigma})$.
\end{rem}

\section{Proofs of Main Theorems}

\subsection{Proof of Main Theorem 1}

\subsubsection{Proof of (O)(a)}
First, we estimate $\bm{H}^{(j+1)}$. 
It is seen that 
\begin{equation}
\frac{1}{\Delta t}(0-0)
	\leq \kappa_{0}^{2}\delta_{k}^{\langle 2 \rangle}0 - c\left(0 - 1 \right) - \kappa \cdot 0 \cdot \frac{ \gamma_{\varepsilon}\bigl( \delta_{k}^{+}\Theta_{k}^{(j)}\bigr) + \gamma_{\varepsilon}\bigl( \delta_{k}^{-}\Theta_{k}^{(j)}\bigr)  }{2} \quad (k=0, \ldots, K). \label{eta_tool_0}
\end{equation}
Substracing \eqref{sc1} from \eqref{eta_tool_0}, we obtain 
\begin{align}
\frac{1}{\Delta t}\left\{\! \left(-H_{k}^{(j+1)}\right) \! - \! \left(-H_{k}^{(j)}\right) \!\right\}
	\leq & \kappa_{0}^{2}\delta_{k}^{\langle 2 \rangle}\left(-H_{k}^{(j+1)}\right)
	-c\left(-H_{k}^{(j+1)}\right) \nonumber\\
	& - \kappa \!\left(\! -H_{k}^{(j+1)} \!\right) \frac{ \gamma_{\varepsilon}\bigl( \delta_{k}^{+}\Theta_{k}^{(j)}\bigr) \! + \! \gamma_{\varepsilon}\bigl( \delta_{k}^{-}\Theta_{k}^{(j)}\bigr)  }{2} \quad (k=0, \ldots, K). \label{eta_tool_0-sc1}
\end{align}
Let us multiply both sides of \eqref{eta_tool_0-sc1} by $[-H_{k}^{(j+1)}]^{+}$ and sum each of the resulting equality over $k = 0,1, \ldots, K$. 
Then, by using the summation by parts formula (Lemma \ref{sbp5}) and \eqref{bc1}, we have  
\begin{align*}
& \frac{1}{\Delta t}\sum_{k=0}^{K}{}^{\prime\prime}\left| [-H_{k}^{(j+1)}]^{+}\right|^{2}\Delta x 
+ \frac{1}{\Delta t}\sum_{k=0}^{K}{}^{\prime\prime}H_{k}^{(j)}[-H_{k}^{(j+1)}]^{+}\Delta x \\
	\leq & \kappa_{0}^{2}\sum_{k=0}^{K}{}^{\prime\prime}\left\{\delta_{k}^{\langle 2 \rangle}\left(-H_{k}^{(j+1)}\right)\right\}[-H_{k}^{(j+1)}]^{+}\Delta x
	- c\sum_{k=0}^{K}{}^{\prime\prime}\left| [-H_{k}^{(j+1)}]^{+}\right|^{2}\Delta x \\
	& - \kappa\sum_{k=0}^{K}{}^{\prime\prime}\left| [-H_{k}^{(j+1)}]^{+}\right|^{2}\frac{ \gamma_{\varepsilon}\bigl( \delta_{k}^{+}\Theta_{k}^{(j)}\bigr) + \gamma_{\varepsilon}\bigl( \delta_{k}^{-}\Theta_{k}^{(j)}\bigr)  }{2}\Delta x \\
	\leq & -\kappa_{0}^{2}\sum_{k=0}^{K-1}\left\{\delta_{k}^{+}\left(-H_{k}^{(j+1)}\right)\right\}\left\{\delta_{k}^{+}[-H_{k}^{(j+1)}]^{+}\right\}\Delta x 
	\leq -\kappa_{0}^{2}\sum_{k=0}^{K-1}\left|\delta_{k}^{+}[-H_{k}^{(j+1)}]^{+}\right|^{2}\Delta x
	\leq 0. 
\end{align*}
From the above inequality and the assumptions \eqref{ass_bd}, it is seen that 
\begin{equation*}
	\frac{1}{\Delta t}\sum_{k=0}^{K}{}^{\prime\prime}\left| [-H_{k}^{(j+1)}]^{+}\right|^{2}\Delta x \leq 0. 
\end{equation*}
Namely, we have $0 \leq H_{k}^{(j+1)} \ (k=0, \ldots, K)$. 
It follows from this inequality and \eqref{sc1} that 
\begin{equation}
	\frac{ H_{k}^{(j+1)} - H_{k}^{(j)} }{\Delta t} 
	\leq \kappa_{0}^{2}\delta_{k}^{\langle 2 \rangle}H_{k}^{(j+1)} - c\left(H_{k}^{(j+1)} - 1 \right) \quad (k=0, \ldots, K). \label{eta_tool_sc1}
\end{equation}
Furthermore, it holds that 
\begin{equation}
\frac{1}{\Delta t}(1-1)
	= \kappa_{0}^{2}\delta_{k}^{\langle 2 \rangle}1 - c\left(1 - 1 \right). \label{eta_tool_1}
\end{equation}
Substracing \eqref{eta_tool_1} from \eqref{eta_tool_sc1}, we obtain 
\begin{equation}
\frac{1}{\Delta t}\left\{\! \left(H_{k}^{(j+1)} - 1\right) \! - \! \left(H_{k}^{(j)} - 1\right) \!\right\}
	\leq \kappa_{0}^{2}\delta_{k}^{\langle 2 \rangle}\!\left(H_{k}^{(j+1)} - 1\right)
	-c\!\left(H_{k}^{(j+1)} - 1\right) \quad (k=0, \ldots, K). \label{eta_tool_sc1-1}
\end{equation}
Let us multiply both sides of \eqref{eta_tool_sc1-1} by $[H_{k}^{(j+1)} - 1]^{+}$ and sum each of the resulting equality over $k = 0,1, \ldots, K$. 
Then, by using the summation by parts formula (Lemma \ref{sbp5}) and \eqref{bc1}, we have  
\begin{align*}
& \frac{1}{\Delta t}\sum_{k=0}^{K}{}^{\prime\prime}\left| [H_{k}^{(j+1)} - 1]^{+}\right|^{2}\Delta x 
+ \frac{1}{\Delta t}\sum_{k=0}^{K}{}^{\prime\prime}\left( 1 - H_{k}^{(j)}\right)[H_{k}^{(j+1)} - 1]^{+}\Delta x \\
	\leq & \kappa_{0}^{2}\sum_{k=0}^{K}{}^{\prime\prime}\left\{\delta_{k}^{\langle 2 \rangle}\left(H_{k}^{(j+1)} - 1\right)\right\}[H_{k}^{(j+1)} - 1]^{+}\Delta x
	- c\sum_{k=0}^{K}{}^{\prime\prime}\left| [H_{k}^{(j+1)} - 1]^{+}\right|^{2}\Delta x \\
	\leq & - \! \kappa_{0}^{2}\sum_{k=0}^{K-1}\!\left\{\!\delta_{k}^{+}\!\left(\! H_{k}^{(j+1)} - 1 \!\right)\!\right\}\!\left\{\delta_{k}^{+}[H_{k}^{(j+1)} - 1]^{+}\right\}\Delta x 
	\leq -\kappa_{0}^{2}\sum_{k=0}^{K-1}\left|\delta_{k}^{+}[H_{k}^{(j+1)} - 1]^{+}\right|^{2}\Delta x
	\leq 0. 
\end{align*}
% Thus, we get 
% \begin{equation*}
% 	\frac{1}{\Delta t}\sum_{k=0}^{K}{}^{\prime\prime}\left| [H_{k}^{(j+1)} - 1]^{+}\right|^{2}\Delta x \leq 0.
% \end{equation*}
In other words, we see from the above inequality and the assumptions \eqref{ass_bd} that $H_{k}^{(j+1)} \leq 1 \ (k=0, \ldots, K)$. 
As a result, we obtain $0 \leq H_{k}^{(j+1)} \leq 1 \ (k=0, \ldots, K)$. 

Next, we estimate $\bm{\Theta}^{(j+1)}$. 
It follows from \eqref{sc2} that  
\begin{align}
& \alpha_{0,k}^{(j+1)}\frac{ \Theta_{k}^{(j+1)} - \Theta_{k}^{(j)} }{\Delta t} - \nu^{2}\delta_{k}^{\langle 2 \rangle}\Theta_{k}^{(j+1)} \nonumber\\
	= & \frac{\kappa}{2}\left\{ \delta_{k}^{+}\left( \alpha\left( H_{k}^{(j+1)} \right)\gamma_{\varepsilon}' (\delta_{k}^{-}\Theta_{k}^{(j+1)}) \right) + \delta_{k}^{-}\left( \alpha\left( H_{k}^{(j+1)} \right)\gamma_{\varepsilon}'(\delta_{k}^{+}\Theta_{k}^{(j+1)})\right) \right\} \quad (k=0, \ldots, K). \label{theta_tool_sc2}
\end{align}
Besides, it is seen that 
\begin{align}
\alpha_{0,k}^{(j+1)}\frac{ \xi_{0} - \xi_{0} }{\Delta t} - \nu^{2}\delta_{k}^{\langle 2 \rangle}\xi_{0} 
	= & \frac{\kappa}{2}\left\{ \delta_{k}^{+}\left( \alpha\left( H_{k}^{(j+1)} \right)\gamma_{\varepsilon}' (\delta_{k}^{-}\xi_{0}) \right) + \delta_{k}^{-}\left( \alpha\left( H_{k}^{(j+1)} \right)\gamma_{\varepsilon}'(\delta_{k}^{+}\xi_{0})\right) \right\} \nonumber\\
	= & \frac{\kappa}{2}\left\{ \delta_{k}^{+}\left( \alpha\left( H_{k}^{(j+1)} \right)\gamma_{\varepsilon}' (0) \right) + \delta_{k}^{-}\left( \alpha\left( H_{k}^{(j+1)} \right)\gamma_{\varepsilon}'(0)\right) \right\} \nonumber\\
	(= & 0) \quad (k=0, \ldots, K). \label{theta_tool_xi0}
\end{align}
Substracing \eqref{theta_tool_xi0} from \eqref{theta_tool_sc2}, we obtain 
\begin{align}
	& \alpha_{0,k}^{(j+1)}\frac{ (\Theta_{k}^{(j+1)} - \xi_{0}) - (\Theta_{k}^{(j)} - \xi_{0}) }{\Delta t} - \nu^{2}\delta_{k}^{\langle 2 \rangle}(\Theta_{k}^{(j+1)} - \xi_{0}) \nonumber\\
	= & \frac{\kappa}{2}\left\{ \delta_{k}^{+}\left( \alpha\left( H_{k}^{(j+1)} \right)\gamma_{\varepsilon}' (\delta_{k}^{-}\Theta_{k}^{(j+1)}) \right) + \delta_{k}^{-}\left( \alpha\left( H_{k}^{(j+1)} \right)\gamma_{\varepsilon}'(\delta_{k}^{+}\Theta_{k}^{(j+1)})\right) \right\} \quad (k=0, \ldots, K). \label{theta_tool_xi0-sc2}
\end{align}
Let us multiply both sides of \eqref{theta_tool_xi0-sc2} by $[\Theta_{k}^{(j+1)} - \xi_{0}]^{+}$ and sum each of the resulting equality over $k = 0,1, \ldots, K$. 
Then, by using the summation by parts formula (Lemma \ref{sbp2}), we can transform the right-hand side of the resulting equality as follows:
\begin{align}
& \frac{\kappa}{2}\sum_{k=0}^{K}{}^{\prime\prime}\left\{ \delta_{k}^{+}\left( \alpha\left( H_{k}^{(j+1)} \right)\gamma_{\varepsilon}' (\delta_{k}^{-}\Theta_{k}^{(j+1)}) \right)\right\}[\Theta_{k}^{(j+1)} - \xi_{0}]^{+}\Delta x \notag\\
& + \frac{\kappa}{2}\sum_{k=0}^{K}{}^{\prime\prime}\left\{ \delta_{k}^{-}\left( \alpha\left( H_{k}^{(j+1)} \right)\gamma_{\varepsilon}' (\delta_{k}^{+}\Theta_{k}^{(j+1)}) \right)\right\}[\Theta_{k}^{(j+1)} - \xi_{0}]^{+}\Delta x \notag\\
	= & \frac{\kappa}{2}\left\{ -\sum_{k=1}^{K} \alpha\left( H_{k}^{(j+1)} \right)\gamma_{\varepsilon}' (\delta_{k}^{-}\Theta_{k}^{(j+1)}) \delta_{k}^{-}[\Theta_{k}^{(j+1)} - \xi_{0}]^{+}\Delta x \right. \notag\\
	& \left. + \left[ \alpha\left( H_{k}^{(j+1)} \right)\gamma_{\varepsilon}' (\delta_{k}^{-}\Theta_{k}^{(j+1)}) \mu_{k}^{-}[\Theta_{k}^{(j+1)} - \xi_{0}]^{+} \right]_{0}^{K}\right\} \notag\\
	& + \frac{\kappa}{2}\left\{ -\sum_{k=0}^{K-1} \alpha\left( H_{k}^{(j+1)} \right)\gamma_{\varepsilon}' (\delta_{k}^{+}\Theta_{k}^{(j+1)}) \delta_{k}^{+}[\Theta_{k}^{(j+1)} - \xi_{0}]^{+}\Delta x \right. \notag\\
	& \left. + \left[ \alpha\left( H_{k}^{(j+1)} \right)\gamma_{\varepsilon}' (\delta_{k}^{+}\Theta_{k}^{(j+1)}) \mu_{k}^{+}[\Theta_{k}^{(j+1)} - \xi_{0}]^{+} \right]_{0}^{K}\right\}. \label{theta_tool_sum}
\end{align}
Here, it holds from \eqref{bc2} that
\begin{gather*}
\left. \delta_{k}^{+}\Theta_{k}^{(j+1)} \right|_{k=K} = \left. - \delta_{k}^{-}\Theta_{k}^{(j+1)} \right|_{k=K}, \quad 
\left. \delta_{k}^{-}\Theta_{k}^{(j+1)} \right|_{k=0} = \left. - \delta_{k}^{+}\Theta_{k}^{(j+1)} \right|_{k=0}, \\
\left. \mu_{k}^{+}[\Theta_{k}^{(j+1)} \! - \! \xi_{0}]^{+}\right|_{k=K} \! = \! \left. \mu_{k}^{-}[\Theta_{k}^{(j+1)} \! - \! \xi_{0}]^{+}\right|_{k=K} \! , \quad 
\left. \mu_{k}^{-}[\Theta_{k}^{(j+1)} \! - \! \xi_{0}]^{+}\right|_{k=0} \! = \! \left. \mu_{k}^{+}[\Theta_{k}^{(j+1)} \! - \! \xi_{0}]^{+}\right|_{k=0}. 
\end{gather*}
From the above and the fact that $\gamma_{\varepsilon}'$ is an odd function, we check that 
\begin{align*}
	& \mbox{(The right-hand side of \eqref{theta_tool_sum})} \\
% & \frac{\kappa}{2}\sum_{k=0}^{K}{}^{\prime\prime}\left\{ \delta_{k}^{+}\left( \alpha\left( H_{k}^{(j+1)} \right)\mo{\gamma_{\varepsilon}'} (\delta_{k}^{-}\Theta_{k}^{(j+1)}) \right)\right\}[\Theta_{k}^{(j+1)} - \xi_{0}]^{+}\Delta x \\
% & + \frac{\kappa}{2}\sum_{k=0}^{K}{}^{\prime\prime}\left\{ \delta_{k}^{-}\left( \alpha\left( H_{k}^{(j+1)} \right)\mo{\gamma_{\varepsilon}'} (\delta_{k}^{+}\Theta_{k}^{(j+1)}) \right)\right\}[\Theta_{k}^{(j+1)} - \xi_{0}]^{+}\Delta x \\
	= & -\frac{\kappa}{2}\sum_{k=1}^{K} \alpha\left( H_{k}^{(j+1)} \right)\gamma_{\varepsilon}' (\delta_{k}^{-}\Theta_{k}^{(j+1)}) \delta_{k}^{-}[\Theta_{k}^{(j+1)} - \xi_{0}]^{+}\Delta x \\
	& -\frac{\kappa}{2}\sum_{k=0}^{K-1} \alpha\left( H_{k}^{(j+1)} \right)\gamma_{\varepsilon}' (\delta_{k}^{+}\Theta_{k}^{(j+1)}) \delta_{k}^{+}[\Theta_{k}^{(j+1)} - \xi_{0}]^{+}\Delta x \\
	= & -\frac{\kappa}{2}\sum_{k=1}^{K} \alpha\left( H_{k}^{(j+1)} \right)\gamma_{\varepsilon}' (\delta_{k}^{+}\Theta_{k-1}^{(j+1)}) \delta_{k}^{+}[\Theta_{k-1}^{(j+1)} - \xi_{0}]^{+}\Delta x \\
	& -\frac{\kappa}{2}\sum_{k=0}^{K-1} \alpha\left( H_{k}^{(j+1)} \right)\gamma_{\varepsilon}' (\delta_{k}^{+}\Theta_{k}^{(j+1)}) \delta_{k}^{+}[\Theta_{k}^{(j+1)} - \xi_{0}]^{+}\Delta x \\
	= & -\frac{\kappa}{2}\sum_{k=0}^{K-1} \left\{ \alpha\left( H_{k+1}^{(j+1)} \right) + \alpha\left( H_{k}^{(j+1)} \right) \right\} \gamma_{\varepsilon}' (\delta_{k}^{+}\Theta_{k}^{(j+1)}) \delta_{k}^{+}[\Theta_{k}^{(j+1)} - \xi_{0}]^{+}\Delta x \\
	% = & -\frac{\kappa}{2}\sum_{k=0}^{K-1} \frac{ \alpha\left( H_{k+1}^{(j+1)} \right) + \alpha\left( H_{k}^{(j+1)} \right) }{\mo{\gamma_{\varepsilon}} (\delta_{k}^{+}\Theta_{k}^{(j+1)})}\left(\delta_{k}^{+}\Theta_{k}^{(j+1)} \right)\delta_{k}^{+}[\Theta_{k}^{(j+1)} - \xi_{0}]^{+}\Delta x \\
	= & -\frac{\kappa}{2}\sum_{k=0}^{K-1} \frac{ \alpha\left( H_{k+1}^{(j+1)} \right) + \alpha\left( H_{k}^{(j+1)} \right) }{\gamma_{\varepsilon} (\delta_{k}^{+}\Theta_{k}^{(j+1)})}\left\{\delta_{k}^{+}\left(\Theta_{k}^{(j+1)} - \xi_{0}\right) \right\}\delta_{k}^{+}[\Theta_{k}^{(j+1)} - \xi_{0}]^{+}\Delta x \\
	\leq & -\frac{\kappa}{2}\sum_{k=0}^{K-1} \frac{ \alpha\left( H_{k+1}^{(j+1)} \right) + \alpha\left( H_{k}^{(j+1)} \right) }{\gamma_{\varepsilon} (\delta_{k}^{+}\Theta_{k}^{(j+1)})}\left| \delta_{k}^{+}[\Theta_{k}^{(j+1)} - \xi_{0}]^{+} \right|^{2}\Delta x 
	\leq 0.
\end{align*}
Meanwhile, it follows from (A1), the summation by parts formula (Lemma \ref{sbp5}), and \eqref{bc2} that 
\begin{align*}
& \frac{1}{\Delta t}\sum_{k=0}^{K}{}^{\prime\prime}\alpha_{0,k}^{(j+1)}\left| [\Theta_{k}^{(j+1)} - \xi_{0}]^{+}\right|^{2}\Delta x 
+ \frac{1}{\Delta t}\sum_{k=0}^{K}{}^{\prime\prime}\alpha_{0,k}^{(j+1)}\left( \xi_{0} - \Theta_{k}^{(j)} \right)[\Theta_{k}^{(j+1)} - \xi_{0}]^{+}\Delta x \\
& - \nu^{2}\sum_{k=0}^{K}{}^{\prime\prime}\left\{ \delta_{k}^{\langle 2 \rangle}\left(\Theta_{k}^{(j+1)} - \xi_{0}\right) \right\}[\Theta_{k}^{(j+1)} - \xi_{0}]^{+}\Delta x \\
	\geq & \frac{\delta_{0}}{\Delta t}\sum_{k=0}^{K}{}^{\prime\prime}\!\left| [\Theta_{k}^{(j+1)} \! - \xi_{0}]^{+}\right|^{2}\!\Delta x 
	+ \frac{\delta_{0}}{\Delta t}\sum_{k=0}^{K}{}^{\prime\prime}\!\left( \xi_{0} - \Theta_{k}^{(j)} \right)\![\Theta_{k}^{(j+1)} - \xi_{0}]^{+}\Delta x \\ 
	& + \nu^{2}\sum_{k=0}^{K-1}\left| \delta_{k}^{+}[\Theta_{k}^{(j+1)} \! - \xi_{0}]^{+} \right|^{2}\!\Delta x. 
\end{align*}
From the above inequalities and the assumptions \eqref{ass_bd}, we obtain  
\begin{equation*}
	\frac{\delta_{0}}{\Delta t}\sum_{k=0}^{K}{}^{\prime\prime}\left| [\Theta_{k}^{(j+1)} - \xi_{0}]^{+}\right|^{2}\Delta x 
	\leq -\nu^{2}\sum_{k=0}^{K-1}\left| \delta_{k}^{+}[\Theta_{k}^{(j+1)} - \xi_{0}]^{+} \right|^{2}\Delta x
	\leq 0.
\end{equation*}
Namely, it holds that $\Theta_{k}^{(j+1)} \leq \xi_{0} \ (k=0, \ldots, K)$. 
Similarly, we have $-\xi_{0} \leq \Theta_{k}^{(j+1)} \ (k=0, \ldots, K)$. 
Thus, we conclude that $|\Theta_{k}^{(j+1)}| \leq \xi_{0} \ (k=0, \ldots, K)$. 

\subsubsection{Proof of (O)(b)}
Multiplying both sides of \eqref{sc1} by $H_{k}^{(j + 1)} - H_{k}^{(j)}$ and summing each of the resulting equality over $k = 0,1, \ldots, K$, we have 
\begin{align*}
	0 = & \frac{1}{\Delta t}\sum_{k=0}^{K}{}^{\prime\prime}\left|H_{k}^{(j + 1)} - H_{k}^{(j)}\right|^{2}\Delta x 
		- \kappa_{0}^{2}\sum_{k=0}^{K}{}^{\prime\prime}\left(\delta_{k}^{\langle 2 \rangle}H_{k}^{(j + 1)}\right)\left(H_{k}^{(j + 1)} - H_{k}^{(j)}\right)\Delta x \\
	    &   + c\sum_{k=0}^{K}{}^{\prime\prime}\left(H_{k}^{(j + 1)} - 1\right)\left(H_{k}^{(j + 1)} - H_{k}^{(j)}\right)\Delta x \\
	    &	+ \kappa\sum_{k=0}^{K}{}^{\prime\prime}H_{k}^{(j + 1)}\left(H_{k}^{(j + 1)} - H_{k}^{(j)}\right)\frac{\gamma_{\varepsilon}\bigl( \delta_{k}^{+}\Theta_{k}^{(j)}\bigr) + \gamma_{\varepsilon}\bigl( \delta_{k}^{-}\Theta_{k}^{(j)}\bigr)}{2}\Delta x.
\end{align*}
By using the summation by parts formula (Lemma \ref{sbp5}), \eqref{bc1}, and the following inequality: 
\begin{equation}
	a(a - b) \geq \frac{1}{2}a^{2} - \frac{1}{2}b^{2} \quad \mbox{for all} \ a, b \in \mathbb{R}, \label{tool_ineq}
\end{equation}
we obtain 
\begin{align*}
	- \kappa_{0}^{2}\sum_{k=0}^{K}{}^{\prime\prime}\!\left(\delta_{k}^{\langle 2 \rangle}H_{k}^{(j + 1)}\right)\!\!\left(\! H_{k}^{(j + 1)} \! - \! H_{k}^{(j)}\!\right)\Delta x 
	= & \kappa_{0}^{2}\sum_{k=0}^{K-1}\left(\delta_{k}^{+}H_{k}^{(j + 1)}\right)\!\left\{\delta_{k}^{+}\!\left(H_{k}^{(j + 1)} - H_{k}^{(j)}\right)\right\}\Delta x \\
	& - \kappa_{0}^{2}\left[\left(\delta_{k}^{\langle 1 \rangle}H_{k}^{(j + 1)}\right)\left(H_{k}^{(j + 1)} - H_{k}^{(j)}\right)\right]_{0}^{K} \\
	= & \kappa_{0}^{2}\sum_{k=0}^{K-1}\left(\delta_{k}^{+}H_{k}^{(j + 1)}\right)\left(\delta_{k}^{+}H_{k}^{(j + 1)} - \delta_{k}^{+}H_{k}^{(j)}\right)\Delta x \\
	\geq & \frac{\kappa_{0}^{2}}{2}\sum_{k=0}^{K-1}\left|\delta_{k}^{+}H_{k}^{(j + 1)}\right|^{2}\Delta x - \frac{\kappa_{0}^{2}}{2}\sum_{k=0}^{K-1}\left|\delta_{k}^{+}H_{k}^{(j)}\right|^{2}\Delta x. 
\end{align*}
Similarly, it follows from \eqref{tool_ineq} that 
\begin{gather*}
	c\sum_{k=0}^{K}{}^{\prime\prime}\!\left(\! H_{k}^{(j + 1)} \! - \! 1 \!\right)\!\!\left(\! H_{k}^{(j + 1)} \! - \! H_{k}^{(j)}\!\right)\Delta x 
	\geq \frac{c}{2}\sum_{k=0}^{K}{}^{\prime\prime}\left|H_{k}^{(j + 1)} - 1\right|^{2}\Delta x 
	     - \frac{c}{2}\sum_{k=0}^{K}{}^{\prime\prime}\left|H_{k}^{(j)} - 1\right|^{2}\Delta x, \\
	\begin{split}
		& \kappa\sum_{k=0}^{K}{}^{\prime\prime}H_{k}^{(j + 1)}\left(H_{k}^{(j + 1)} - H_{k}^{(j)}\right)\frac{\gamma_{\varepsilon}\bigl( \delta_{k}^{+}\Theta_{k}^{(j)}\bigr) + \gamma_{\varepsilon}\bigl( \delta_{k}^{-}\Theta_{k}^{(j)}\bigr)}{2}\Delta x \\
		\geq & \kappa\!\sum_{k=0}^{K}\!{}^{\prime\prime}\alpha\!\left(\!H_{k}^{(j + 1)}\!\right)\frac{\gamma_{\varepsilon}\bigl( \delta_{k}^{+}\Theta_{k}^{(j)}\bigr) \! + \! \gamma_{\varepsilon}\bigl( \delta_{k}^{-}\Theta_{k}^{(j)}\bigr)}{2}\Delta x
		\! - \! \kappa\!\sum_{k=0}^{K}\!{}^{\prime\prime}\alpha\!\left(\!H_{k}^{(j)}\!\right)\frac{\gamma_{\varepsilon}\bigl( \delta_{k}^{+}\Theta_{k}^{(j)}\bigr) \! + \! \gamma_{\varepsilon}\bigl( \delta_{k}^{-}\Theta_{k}^{(j)}\bigr)}{2}\Delta x.
	\end{split}
\end{gather*}
From the above, we obtain 
\begin{align*}
	0 \! \geq & \frac{1}{\Delta t}\sum_{k=0}^{K}{}^{\prime\prime}\left|H_{k}^{(j + 1)} - H_{k}^{(j)}\right|^{2}\Delta x 
		+ \frac{\kappa_{0}^{2}}{2}\sum_{k=0}^{K-1}\left|\delta_{k}^{+}H_{k}^{(j + 1)}\right|^{2}\Delta x - \frac{\kappa_{0}^{2}}{2}\sum_{k=0}^{K-1}\left|\delta_{k}^{+}H_{k}^{(j)}\right|^{2}\Delta x \\
	    &   + \frac{c}{2}\sum_{k=0}^{K}{}^{\prime\prime}\left|H_{k}^{(j + 1)} - 1\right|^{2}\Delta x 
	    	- \frac{c}{2}\sum_{k=0}^{K}{}^{\prime\prime}\left|H_{k}^{(j)} - 1\right|^{2}\Delta x \\
	    &	+ \!\kappa\!\sum_{k=0}^{K}\!{}^{\prime\prime}\!\alpha\!\left(\!H_{k}^{(j + 1\!)}\!\right)\!\frac{\gamma_{\varepsilon}\!\bigl( \delta_{k}^{+}\Theta_{k}^{(j)}\bigr) \!\! + \! \gamma_{\varepsilon}\!\bigl( \delta_{k}^{-}\Theta_{k}^{(j)}\bigr)}{2}\Delta x
	    \! - \! \kappa\!\sum_{k=0}^{K}\!{}^{\prime\prime}\!\alpha\!\left(\!H_{k}^{(j)}\!\right)\!\frac{\gamma_{\varepsilon}\!\bigl( \delta_{k}^{+}\Theta_{k}^{(j)}\bigr) \!\! + \! \gamma_{\varepsilon}\!\bigl( \delta_{k}^{-}\Theta_{k}^{(j)}\bigr)}{2}\Delta x.
\end{align*}
Moreover, multiplying both sides of \eqref{sc2} by $\Theta_{k}^{(j + 1)} - \Theta_{k}^{(j)}$ and summing each of the resulting equality over $k = 0,1, \ldots, K$, we have 
\begin{align*}
	0 = & \frac{1}{\Delta t}\sum_{k=0}^{K}{}^{\prime\prime}\alpha_{0,k}^{(j + 1)}\left|\Theta_{k}^{(j + 1)} - \Theta_{k}^{(j)}\right|^{2}\Delta x 
		- \nu^{2}\sum_{k=0}^{K}{}^{\prime\prime}\left(\delta_{k}^{\langle 2 \rangle}\Theta_{k}^{(j + 1)}\right)\left(\Theta_{k}^{(j + 1)} - \Theta_{k}^{(j)}\right)\Delta x \\
	    &	- \! \frac{\kappa}{2}\sum_{k=0}^{K}\!{}^{\prime\prime}\!\left\{ \delta_{k}^{+}\!\left(\! \alpha\!\left(\! H_{k}^{(j+1\! )} \!\right)\! \gamma_{\varepsilon}' (\delta_{k}^{-}\Theta_{k}^{(j+1)}) \!\right) \!+ \! \delta_{k}^{-}\!\left(\! \alpha\!\left(\! H_{k}^{(j+1\! )} \!\right)\! \gamma_{\varepsilon}' (\delta_{k}^{+}\Theta_{k}^{(j+1)}) \!\right) \!\right\} \! \left(\! \Theta_{k}^{(j + 1)} \! - \! \Theta_{k}^{(j)} \!\right)\Delta x.
\end{align*}
In this part, we consider the third term in the right hand side of the above equality. 
We obtain the following equality by the summation by parts formula (Lemma \ref{sbp1}): 
\begin{align*}
	& - \frac{\kappa}{2}\sum_{k=0}^{K}\!{}^{\prime\prime}\!\left\{ \delta_{k}^{+}\!\left(\! \alpha\!\left(\! H_{k}^{(j+1\! )} \!\right)\! \gamma_{\varepsilon}' (\delta_{k}^{-}\Theta_{k}^{(j+1)}) \!\right) \!\right\} \! \left(\! \Theta_{k}^{(j + 1)} \! - \! \Theta_{k}^{(j)} \!\right)\Delta x \\
	= & \frac{\kappa}{2}\sum_{k=0}^{K}\!{}^{\prime\prime} \! \alpha\!\left(\! H_{k}^{(j+1\! )} \!\right)\! \gamma_{\varepsilon}' (\delta_{k}^{-}\Theta_{k}^{(j+1)}) \delta_{k}^{-}\!\left(\! \Theta_{k}^{(j + 1)} \! - \! \Theta_{k}^{(j)} \!\right)\Delta x \\
	& \underset{\rm =: (d.B.T.)}{\underline{- \frac{\kappa}{4}\left[ \alpha\!\left(\! H_{k+1}^{(j+1\! )} \!\right)\! \gamma_{\varepsilon}' (\delta_{k}^{-}\Theta_{k+1}^{(j+1)}) \! \left(\! \Theta_{k}^{(j + 1)} \! - \! \Theta_{k}^{(j)} \!\right) + \alpha\!\left(\! H_{k}^{(j+1\! )} \!\right)\! \gamma_{\varepsilon}' (\delta_{k}^{-}\Theta_{k}^{(j+1)}) \! \left(\! \Theta_{k-1}^{(j + 1)} \! - \! \Theta_{k-1}^{(j)} \!\right) \right]_{0}^{K} }} \\
	= & \frac{\kappa}{2}\sum_{k=0}^{K}\!{}^{\prime\prime} \alpha\!\left(\! H_{k}^{(j+1\! )} \!\right)\! \frac{\delta_{k}^{-}\Theta_{k}^{(j+1)}}{ \gamma_{\varepsilon}(\delta_{k}^{-}\Theta_{k}^{(j+1)}) } \! \left(\! \delta_{k}^{-}\Theta_{k}^{(j + 1)} \! - \! \delta_{k}^{-}\Theta_{k}^{(j)} \!\right)\Delta x + \mbox{(d.B.T.)}. 
\end{align*}
Besides, using the following inequality: 
\begin{equation*}
	\frac{a}{\gamma_{\varepsilon}(a)}(a - b) \geq \gamma_{\varepsilon}(a) - \gamma_{\varepsilon}(b) \quad \mbox{for all} \ a, b \in \mathbb{R}, 
\end{equation*}
we have 
\begin{align*}
	& - \frac{\kappa}{2}\sum_{k=0}^{K}\!{}^{\prime\prime}\!\left\{ \delta_{k}^{+}\!\left(\! \alpha\!\left(\! H_{k}^{(j+1\! )} \!\right)\! \gamma_{\varepsilon}' (\delta_{k}^{-}\Theta_{k}^{(j+1)}) \!\right) \!\right\} \! \left(\! \Theta_{k}^{(j + 1)} \! - \! \Theta_{k}^{(j)} \!\right)\Delta x \\
	\geq & \frac{\kappa}{2}\sum_{k=0}^{K}\!{}^{\prime\prime}\!\alpha\!\left(\!H_{k}^{(j + 1\!)}\!\right)\!\gamma_{\varepsilon}\bigl( \delta_{k}^{-}\Theta_{k}^{(j+1)}\bigr) \Delta x 
	- \frac{\kappa}{2}\sum_{k=0}^{K}\!{}^{\prime\prime}\!\alpha\!\left(\!H_{k}^{(j + 1\!)}\!\right)\!\gamma_{\varepsilon}\bigl( \delta_{k}^{-}\Theta_{k}^{(j)}\bigr) \Delta x + \mbox{(d.B.T.)}.
\end{align*}
In the same manner as above, it holds that 
\begin{align*}
	& - \frac{\kappa}{2}\sum_{k=0}^{K}\!{}^{\prime\prime}\!\left\{ \delta_{k}^{-}\!\left(\! \alpha\!\left(\! H_{k}^{(j+1\! )} \!\right)\! \gamma_{\varepsilon}' (\delta_{k}^{+}\Theta_{k}^{(j+1)}) \!\right) \!\right\} \! \left(\! \Theta_{k}^{(j + 1)} \! - \! \Theta_{k}^{(j)} \!\right)\Delta x \\
	\geq & \frac{\kappa}{2}\sum_{k=0}^{K}\!{}^{\prime\prime}\!\alpha\!\left(\!H_{k}^{(j + 1\!)}\!\right)\!\gamma_{\varepsilon}\bigl( \delta_{k}^{+}\Theta_{k}^{(j+1)}\bigr) \Delta x 
	- \frac{\kappa}{2}\sum_{k=0}^{K}\!{}^{\prime\prime}\!\alpha\!\left(\!H_{k}^{(j + 1\!)}\!\right)\!\gamma_{\varepsilon}\bigl( \delta_{k}^{+}\Theta_{k}^{(j)}\bigr) \Delta x + \mbox{(d.B.T.)}.
\end{align*}
Since $\gamma_{\varepsilon}'$ is an odd function, it follows from \eqref{bc1} and \eqref{bc2} that $\mbox{(d.B.T.)} = 0$. 
Therefore, we obtain 
\begin{align*}
	0 \! \geq & \frac{1}{\Delta t}\sum_{k=0}^{K}{}^{\prime\prime}\alpha_{0,k}^{(j + 1)}\left|\Theta_{k}^{(j + 1)} - \Theta_{k}^{(j)}\right|^{2}\Delta x
		+ \frac{\nu^{2}}{2}\sum_{k=0}^{K-1}\left|\delta_{k}^{+}\Theta_{k}^{(j + 1)}\right|^{2}\Delta x - \frac{\nu^{2}}{2}\sum_{k=0}^{K-1}\left|\delta_{k}^{+}\Theta_{k}^{(j)}\right|^{2}\Delta x \\
	    &	+ \!\kappa\!\sum_{k=0}^{K}\!{}^{\prime\prime}\!\alpha\!\left(\!H_{k}^{(j + 1\!)}\!\right)\!\frac{\gamma_{\varepsilon}\bigl( \delta_{k}^{+}\Theta_{k}^{(j+1)}\bigr) + \gamma_{\varepsilon}\bigl( \delta_{k}^{-}\Theta_{k}^{(j+1)}\bigr)}{2}\Delta x \\
	    & 	- \! \kappa\!\sum_{k=0}^{K}\!{}^{\prime\prime}\!\alpha\!\left(\!H_{k}^{(j+1\!)}\!\right)\!\frac{\gamma_{\varepsilon}\bigl( \delta_{k}^{+}\Theta_{k}^{(j)}\bigr) + \gamma_{\varepsilon}\bigl( \delta_{k}^{-}\Theta_{k}^{(j)}\bigr)}{2}\Delta x.
\end{align*}
Here, for any $\{f_{k}\}_{k=-1}^{K+1} \in \mathbb{R}^{K+3}$ satisfying the discrete homogeneous Neumann boundary condition $\delta_{k}^{\langle 1 \rangle}f_{k} = 0 \ (k = 0,K)$, 
the following equality holds: 
\begin{equation*}
	\sum _ {k=0}^{K}{}^{\prime \prime}\frac{ |\delta_{k}^{+} f_{k}|^{2} + | \delta_{k}^{-} f_{k} |^{2} }{2}\Delta x = \sum _ {k=0}^{K-1} |\delta_{k}^{+} f_{k}|^{2} \Delta x. \vspace{-1mm} 
\end{equation*}
Consequently, we conclude that 
\begin{align*}
        0 \geq & \mathscr{F}_{1, \rm d}(\bm{H}^{(j+1)}, \bm{\Theta}^{(j+1)}) - \mathscr{F}_{1, \rm d}(\bm{H}^{(j)}, \bm{\Theta}^{(j)}) \\ 
	& + \frac{1}{\Delta t}\sum_{k=0}^{K}{}^{\prime\prime}\left|H_{k}^{(j + 1)} - H_{k}^{(j)}\right|^{2}\Delta x + \frac{1}{\Delta t}\sum_{k=0}^{K}{}^{\prime\prime}\alpha_{0,k}^{(j + 1)}\left|\Theta_{k}^{(j + 1)} - \Theta_{k}^{(j)}\right|^{2}\Delta x. 
\end{align*}

\subsubsection{Proof of (I)}
From (A1) and (A2), we see that 
\begin{equation*}
	\left|\xi_{i, k}^{(j+1)}\right| \to 0, \ \mbox{as} \ \Delta x, \Delta t \to 0 \quad (i = 1,2,\ldots, 6), 
\end{equation*}
where $\xi_{i, k}^{(j+1)} \ (k = 0,1,\ldots, K,\ j = 0,1,\ldots, N - 1,\ i = 1,2,\ldots, 6)$ are residual terms defined by \eqref{xi1}--\eqref{xi6}. 
Thus, we obtain from Lemma \ref{e_eta_theta_est} that 
\begin{equation*}
        \sup_{0 \leq n \leq N} \bigl\| \bm{e}_\eta^{(j)} \bigr\|_{L_{\rm d}^2} +\sup_{0 \leq n \leq N} \bigl\| \bm{e}_\theta^{(j)} \bigr\|_{L_{\rm d}^2} \to 0, \mbox{ as $ \Delta x, \Delta t \to 0 $.}
\end{equation*}

\subsection{Proof of Main Theorem 2}
% Let us define 
% \begin{gather*}
% 	M_{i,j}(v) := \max\left\{\left|\frac{\partial^{i+j}v}{\partial x^{i}\partial t^{j}}\right| ; (x,t) \in [0,L] \times [0,T]\right\} \quad {\rm for\ all\ } i,j \in \mathbb{Z}, \\
% 	\tilde{M}_{i,j}(\tilde{v}) := \max\left\{\left|\frac{\partial^{i+j}\tilde{v} }{\partial x^{i}\partial t^{j}}\right| ; (x,t) \in [-\Delta x,L + \Delta x] \times [0,T]\right\} \quad {\rm for\ all\ } i,j \in \mathbb{Z}. 
% \end{gather*} 
Firstly, we estimate $\bm{\xi}_{1}^{(j+1)} = \{\xi_{1, k}^{(j+1)}\}_{k = 0}^{K} \ \in \mathbb{R}^{K+1}$ and $\bm{\xi}_{4}^{(j+1)} = \{\xi_{4, k}^{(j+1)}\}_{k = 0}^{K} \ \in \mathbb{R}^{K+1}$. 
It follows from the Taylor theorem and (A3) that
\begin{equation}
	\left|\xi_{1, k}^{(j+1)}\right| 
	= \left| \frac{ \eta_{k}^{(j + 1)} - \eta_{k}^{(j)} }{\Delta t} - \partial_{t} \eta_{k}^{(j + 1)} \right|
	\leq C(\Delta t)^{\sigma} \quad (k=0, \ldots, K). \label{xi1_est}
\end{equation}
As a remark, throughout this proof, we need the reader to keep in mind that the meaning of $C$ changes from line to line, whereas $C$ always denote those constants. 
Similarly, we see from the Taylor theorem, (A1), and (A3) that  
\begin{equation}
	\left|\xi_{4, k}^{(j+1)}\right| 
	= \left|\alpha_{0,k}^{(j+1)}\right| \left| \frac{ \theta_{k}^{(j + 1)} - \theta_{k}^{(j)} }{\Delta t} - \partial_{t} \theta_{k}^{(j + 1)} \right| 
	\leq C(\Delta t)^{\sigma} \quad (k=0, \ldots, K). \label{xi4_est}
\end{equation}
Next, we estimate $\bm{\xi}_{2}^{(j+1)} = \{\xi_{2, k}^{(j+1)}\}_{k = 0}^{K} \ \in \mathbb{R}^{K+1}$ and $\bm{\xi}_{5}^{(j+1)} = \{\xi_{5, k}^{(j+1)}\}_{k = 0}^{K} \ \in \mathbb{R}^{K+1}$. 
For any $t \in [0,T]$ and $k=0,\ldots, K$, applying the Taylor theorem to $\tilde{\eta}$, there exists $r_{1} \in (0,1)$ such that  
\begin{align}
& \frac{ \tilde{\eta}(t,(k + 1)\Delta x) -2\tilde{\eta}(t,k\Delta x) + \tilde{\eta}(t,(k - 1)\Delta x) }{(\Delta x)^{2}} \nonumber\\
	= & \frac{1}{2}\left\{ \partial_{x}^{2}\tilde{\eta}( t, (k + r_{1})\Delta x ) 
	+ \partial_{x}^{2}\tilde{\eta}( t, (k - r_{1})\Delta x ) \right\}. \label{eta_avg}
\end{align}
From the definition \eqref{def_ext} of $\tilde{\eta}$, it holds that $\partial_{x}^{2}\tilde{\eta}(t, x) = \partial_{x}^{2}\eta(t, -x)$ for all $x \in (-\Delta x,0)$. 
Hence, we have $\partial_{x}^{2}\tilde{\eta}(t, -r_{1}\Delta x) = \partial_{x}^{2}\eta(t,r_{1}\Delta x)$. 
From a similar observation, we obtain $\partial_{x}^{2}\tilde{\eta}(t, (K+r_{1})\Delta x) = \partial_{x}^{2}\tilde{\eta}(t, (K-r_{1})\Delta x)$. 
Namely, substituting $( j + 1 )\Delta t$ into $t$ in \eqref{eta_avg}, we get 
\begin{equation*}
\delta_{k}^{\langle 2 \rangle}\tilde{\eta}_{k}^{ (j + 1) }  
	= \begin{cases}
		\displaystyle \partial_{x}^{2}\eta_{r_{1}}^{ (j + 1) } & (k=0), \\[5pt]
		\displaystyle \frac{1}{2}\left( \partial_{x}^{2}\eta_{k+r_{1}}^{ (j + 1) } + \partial_{x}^{2}\eta_{k-r_{1}}^{ (j + 1) } \right) & (k=1,\ldots, K-1), \\[5pt]
		\displaystyle \partial_{x}^{2}\eta_{K-r_{1}}^{ ( j + 1) }, & (k=K). 
	\end{cases}
\end{equation*} 
Hence, we conclude that 
\begin{equation}
	\left|\xi_{2, k}^{(j+1)}\right| = \kappa_{0}^{2}\left| \delta_{k}^{\langle 2 \rangle}\tilde{\eta}_{k}^{(j+1)} - \partial_{x}^{2}\eta_{k}^{( j + 1)} \right| \leq C(\Delta x)^{\sigma} \quad (k=0, \ldots, K).
	\label{xi2_est}
\end{equation}
In the same manner, as described above, we have
\begin{equation}
	\left|\xi_{5, k}^{(j+1)}\right| = \nu^{2}\left| \delta_{k}^{\langle 2 \rangle}\tilde{\theta}_{k}^{(j + 1)} - \partial_{x}^{2} \theta_{k}^{(j + 1)} \right| \leq C(\Delta x)^{\sigma} \quad (k=0, \ldots, K).
	\label{xi5_est}
\end{equation}
Lastly, we estimate $\bm{\xi}_{3}^{(j+1)} = \{\xi_{3, k}^{(j+1)}\}_{k = 0}^{K} \ \in \mathbb{R}^{K+1}$ and $\bm{\xi}_{6}^{(j+1)} = \{\xi_{6, k}^{(j+1)}\}_{k = 0}^{K} \ \in \mathbb{R}^{K+1}$. 
By using (A2), we have 
\begin{align*}
	\left|\xi_{3, k}^{(j+1)}\right| 
	= & \kappa\left|\eta_{k}^{( j + 1)}\left(\frac{\gamma_{\varepsilon}\bigl( \delta_{k}^{+}\tilde{\theta}_{k}^{(j)}\bigr) + \gamma_{\varepsilon}\bigl( \delta_{k}^{-}\tilde{\theta}_{k}^{(j)}\bigr)}{2} - \gamma_{\varepsilon}\bigl( \partial_{x}\theta_{k}^{(j+1)}\bigr) \right)\right| \\
	\leq & \kappa\frac{\left|\eta_{k}^{( j + 1)}\right|}{2}
	\left(\left|\gamma_{\varepsilon}\bigl( \delta_{k}^{+}\tilde{\theta}_{k}^{(j)}\bigr) - \gamma_{\varepsilon}\bigl( \partial_{x}\theta_{k}^{(j+1)}\bigr) \right| 
	+ \left|\gamma_{\varepsilon}\bigl( \delta_{k}^{-}\tilde{\theta}_{k}^{(j)}\bigr) - \gamma_{\varepsilon}\bigl( \partial_{x}\theta_{k}^{(j+1)}\bigr) \right| \right) \\
	\leq & \frac{C_{1}\kappa}{2}
	\left(\left|\gamma_{\varepsilon}\bigl( \delta_{k}^{+}\tilde{\theta}_{k}^{(j)}\bigr) \! - \! \gamma_{\varepsilon}\bigl( \partial_{x}\theta_{k}^{(j+1)}\bigr) \right| 
	\! + \! \left|\gamma_{\varepsilon}\bigl( \delta_{k}^{-}\tilde{\theta}_{k}^{(j)}\bigr) \! - \! \gamma_{\varepsilon}\bigl( \partial_{x}\theta_{k}^{(j+1)}\bigr) \right| \right) \quad (k=0, \ldots, K),
\end{align*}
where $C_{1}$ is a constant defined by \eqref{eta_bound}.
Additionally, we see from the property of the difference quotient that
\begin{gather*}
	\gamma_{\varepsilon}\bigl( \delta_{k}^{+}\tilde{\theta}_{k}^{(j)}\bigr) - \gamma_{\varepsilon}\bigl( \partial_{x}\theta_{k}^{(j+1)}\bigr) 
	= \frac{ d\gamma_{\varepsilon} }{ d(\delta_{k}^{+}\tilde{\theta}_{k}^{(j)}, \partial_{x}\theta_{k}^{(j+1)}) }\left(\delta_{k}^{+}\tilde{\theta}_{k}^{(j)} - \partial_{x}\theta_{k}^{(j+1)}\right) \quad (k=0, \ldots, K), \\
	\gamma_{\varepsilon}\bigl( \delta_{k}^{-}\tilde{\theta}_{k}^{(j)}\bigr) - \gamma_{\varepsilon}\bigl( \partial_{x}\theta_{k}^{(j+1)}\bigr) 
	= \frac{ d\gamma_{\varepsilon} }{ d(\delta_{k}^{-}\tilde{\theta}_{k}^{(j)}, \partial_{x}\theta_{k}^{(j+1)}) }\left(\delta_{k}^{-}\tilde{\theta}_{k}^{(j)} - \partial_{x}\theta_{k}^{(j+1)}\right) \quad (k=0, \ldots, K).
\end{gather*}
Moreover, by using the following inequality: 
\begin{equation}
	% \frac{|a + b|}{\sqrt{\varepsilon^{2} + a^{2}} + \sqrt{\varepsilon^{2} + b^{2}}} 
	% \leq \frac{|a| + |b|}{\sqrt{a^{2}} + \sqrt{b^{2}}} 
	% = \frac{|a| + |b|}{|a| + |b|} 
	% = 1 \quad (\forall a,b \in \mathbb{R})
	\left| \frac{ d\gamma_{\varepsilon} }{ d(a, b) } \right| \leq 1 \quad \mbox{for all} \ a,b \in \mathbb{R}, \label{diff_f}
\end{equation}
we obtain  
\begin{equation*}
	\left|\xi_{3, k}^{(j+1)}\right| \leq \frac{C_{1}\kappa}{2}\left(\left|\delta_{k}^{+}\tilde{\theta}_{k}^{(j)} - \partial_{x}\theta_{k}^{(j+1)}\right| + \left|\delta_{k}^{-}\tilde{\theta}_{k}^{(j)} - \partial_{x}\theta_{k}^{(j+1)}\right| \right) 
	\quad (k=0, \ldots, K). 
\end{equation*}
Applying the Taylor theorem to $\tilde{\theta}$, we see from the definition of $\tilde{\theta}$ and (A2) that 
\begin{gather*}
	\left|\delta_{k}^{+}\tilde{\theta}_{k}^{(j)} - \partial_{x}\theta_{k}^{(j+1)}\right| 
	\leq C(\Delta t + \Delta x) \quad (k=0, \ldots, K), \\
	\left|\delta_{k}^{-}\tilde{\theta}_{k}^{(j)} - \partial_{x}\theta_{k}^{(j+1)}\right| 
	\leq C(\Delta t + \Delta x) \quad (k=0, \ldots, K).
\end{gather*}
From the above, we have 
\begin{equation}
	\left|\xi_{3, k}^{(j+1)}\right| \leq C(\Delta t + \Delta x) \quad (k=0, \ldots, K).
	\label{xi3_est}
\end{equation}
Next, we estimate $\bm{\xi}_{6}^{(j+1)}$. 
For the purpose, we consider the following difference:  
\begin{equation*}
	\left\{\delta_{k}^{+}\left( \alpha( \tilde{\eta}_{k}^{(j+1)} ) \gamma_{\varepsilon}'(\delta_{k}^{-}\tilde{\theta}_{k}^{(j+1)}) \right) + \delta_{k}^{-}\left( \alpha( \tilde{\eta}_{k}^{(j+1)} ) \gamma_{\varepsilon}'(\delta_{k}^{+}\tilde{\theta}_{k}^{(j+1)}) \right)\right\} - 2\partial_{x}\left( \alpha( \eta_{k}^{(j+1)} ) \gamma_{\varepsilon}'(\partial_{x}\theta_{k}^{(j+1)}) \right). 
\end{equation*}
For reasons of space limitation, the superscript $(j+1)$ is omitted hereafter.
Here, we see from Lemma \ref{prod} that 
\begin{gather*}
	\begin{split}
		\delta_{k}^{+}\left( \alpha( \tilde{\eta}_{k} ) \gamma_{\varepsilon}'(\delta_{k}^{-}\tilde{\theta}_{k}) \right) 
			= & \frac{1}{2}\delta_{k}^{+}\left( \alpha( \tilde{\eta}_{k} ) \right) \left(\gamma_{\varepsilon}'(\delta_{k}^{-}\tilde{\theta}_{k+1}) + \gamma_{\varepsilon}'(\delta_{k}^{-}\tilde{\theta}_{k})\right) 
			+ \mu_{k}^{+}\left( \alpha( \tilde{\eta}_{k} ) \right)\delta_{k}^{+}\left(  \gamma_{\varepsilon}'(\delta_{k}^{-}\tilde{\theta}_{k}) \right) \\
			= & \frac{1}{2}\delta_{k}^{+}\left( \alpha( \tilde{\eta}_{k} ) \right) \left(\gamma_{\varepsilon}'(\delta_{k}^{+}\tilde{\theta}_{k}) + \gamma_{\varepsilon}'(\delta_{k}^{-}\tilde{\theta}_{k})\right) 
			+ \mu_{k}^{+}\left( \alpha( \tilde{\eta}_{k} ) \right)\delta_{k}^{+}\left(  \gamma_{\varepsilon}'(\delta_{k}^{-}\tilde{\theta}_{k}) \right),
	\end{split}\\
	\begin{split}
		\delta_{k}^{-}\left( \alpha( \tilde{\eta}_{k} ) \gamma_{\varepsilon}'(\delta_{k}^{+}\tilde{\theta}_{k}) \right) 
			= & \frac{1}{2}\delta_{k}^{-}\left( \alpha( \tilde{\eta}_{k} ) \right) \left(\gamma_{\varepsilon}'(\delta_{k}^{+}\tilde{\theta}_{k}) + \gamma_{\varepsilon}'(\delta_{k}^{+}\tilde{\theta}_{k-1})\right) 
			+ \mu_{k}^{-}\left( \alpha( \tilde{\eta}_{k} ) \right)\delta_{k}^{-}\left(  \gamma_{\varepsilon}'(\delta_{k}^{+}\tilde{\theta}_{k}) \right) \\
			= & \frac{1}{2}\delta_{k}^{-}\left( \alpha( \tilde{\eta}_{k} ) \right) \left(\gamma_{\varepsilon}'(\delta_{k}^{+}\tilde{\theta}_{k}) + \gamma_{\varepsilon}'(\delta_{k}^{-}\tilde{\theta}_{k})\right) 
			+ \mu_{k}^{-}\left( \alpha( \tilde{\eta}_{k} ) \right)\delta_{k}^{-}\left(  \gamma_{\varepsilon}'(\delta_{k}^{+}\tilde{\theta}_{k}) \right)
	\end{split}	
\end{gather*}
for $k=0, \ldots, K$. 
From the above, we obtain 
\begin{align*}
& \delta_{k}^{+}\left( \alpha( \tilde{\eta}_{k} ) \gamma_{\varepsilon}'(\delta_{k}^{-}\tilde{\theta}_{k}) \right) + \delta_{k}^{-}\left( \alpha( \tilde{\eta}_{k} ) \gamma_{\varepsilon}'(\delta_{k}^{+}\tilde{\theta}_{k}) \right) \\
	= & \delta_{k}^{\langle 1 \rangle}\left( \alpha( \tilde{\eta}_{k} ) \right) \left(\gamma_{\varepsilon}'(\delta_{k}^{+}\tilde{\theta}_{k}) + \gamma_{\varepsilon}'(\delta_{k}^{-}\tilde{\theta}_{k})\right) 
	+ \mu_{k}^{+}\!\left( \alpha( \tilde{\eta}_{k} ) \right)\!\delta_{k}^{+}\!\left(  \gamma_{\varepsilon}'(\delta_{k}^{-}\tilde{\theta}_{k}) \right) 
	+ \mu_{k}^{-}\!\left( \alpha( \tilde{\eta}_{k} ) \right)\!\delta_{k}^{-}\!\left(  \gamma_{\varepsilon}'(\delta_{k}^{+}\tilde{\theta}_{k}) \right)
\end{align*}
for $k=0, \ldots, K$. 
Thus, we have 
\begin{align}
	\left|\xi_{6, k}\right|
	\leq & \frac{\kappa}{2}\left(
		\left|\delta_{k}^{\langle 1 \rangle}\left( \alpha( \tilde{\eta}_{k} ) \right) \left(\gamma_{\varepsilon}'(\delta_{k}^{+}\tilde{\theta}_{k}) + \gamma_{\varepsilon}'(\delta_{k}^{-}\tilde{\theta}_{k})\right) - 2\partial_{x} \alpha( \eta_{k} ) \gamma_{\varepsilon}'(\partial_{x}\theta_{k})\right|
	\right. \notag\\
	& + \left|
		\mu_{k}^{+}\left( \alpha( \tilde{\eta}_{k} ) \right)\delta_{k}^{+}\left(  \gamma_{\varepsilon}'(\delta_{k}^{-}\tilde{\theta}_{k}) \right) 
		+ \mu_{k}^{-}\left( \alpha( \tilde{\eta}_{k} ) \right)\delta_{k}^{-}\left(  \gamma_{\varepsilon}'(\delta_{k}^{+}\tilde{\theta}_{k}) \right) 
		- 2\alpha( \eta_{k} )\partial_{x}\gamma_{\varepsilon}'(\partial_{x}\theta_{k})
	\right| 
	\Bigr) \label{xi6_1} 
\end{align}
for $k=0, \ldots, K$. 
Next, we evaluate the first term on the right-hand side of \eqref{xi6_1}. 
Here, we see from (A2) and the mean value theorem that $| \delta_{k}^{\langle 1 \rangle}\left( \alpha( \tilde{\eta}_{k} ) \right) | \leq C \ (k=0, \ldots, K)$.
Hence, we obtain from the above inequality and the following inequality: $|\gamma_{\varepsilon}'(u)| \leq 1$ that 
\begin{align*}
	& \left|
		\delta_{k}^{\langle 1 \rangle}\left( \alpha( \tilde{\eta}_{k} ) \right) \left(\gamma_{\varepsilon}'(\delta_{k}^{+}\tilde{\theta}_{k}) + \gamma_{\varepsilon}'(\delta_{k}^{-}\tilde{\theta}_{k})\right) 
		- 2\partial_{x} \alpha( \eta_{k} ) \gamma_{\varepsilon}'(\partial_{x}\theta_{k})
	\right| \\ 
	\leq & 2\left|
		\delta_{k}^{\langle 1 \rangle}\left( \alpha( \tilde{\eta}_{k} ) \right) 
		- \partial_{x} \alpha( \eta_{k} ) 
	\right| 
	\left| \gamma_{\varepsilon}'(\partial_{x}\theta_{k}) \right|
	+ \left| \delta_{k}^{\langle 1 \rangle}\left( \alpha( \tilde{\eta}_{k} ) \right) \right|
	\left|
		\gamma_{\varepsilon}'(\delta_{k}^{+}\tilde{\theta}_{k}) + \gamma_{\varepsilon}'(\delta_{k}^{-}\tilde{\theta}_{k}) 
		- 2\gamma_{\varepsilon}'(\partial_{x}\theta_{k})
	\right| \\
	\leq & C\left(
		\left|
			\delta_{k}^{\langle 1 \rangle}\left( \alpha( \tilde{\eta}_{k} ) \right) 
			- \partial_{x} \alpha( \eta_{k} ) 
		\right| 
		+ \left|
			\gamma_{\varepsilon}'(\delta_{k}^{+}\tilde{\theta}_{k}) + \gamma_{\varepsilon}'(\delta_{k}^{-}\tilde{\theta}_{k}) 
			- 2\gamma_{\varepsilon}'(\partial_{x}\theta_{k})
		\right| 
	\right) \quad (k=0, \ldots, K).
\end{align*}
Moreover, we see from the chain rule, the mean value theorem, and (A2) that 
\begin{align*}
	\left| \delta_{k}^{\langle 1 \rangle}\!\left( \alpha( \tilde{\eta}_{k} ) \right) - \partial_{x} \alpha( \eta_{k} ) \right|
	= & \left| \frac{ \tilde{\eta}_{k+1} + \tilde{\eta}_{k-1}}{2}\delta_{k}^{\langle 1 \rangle}\tilde{\eta}_{k} - \eta_{k} \partial_{x} \eta_{k} \right| \\
	\leq & \left|\frac{ \tilde{\eta}_{k+1} + \tilde{\eta}_{k-1}}{2} - \eta_{k}\right| \left|\delta_{k}^{\langle 1 \rangle}\tilde{\eta}_{k} \right| 
	+ \left|\eta_{k}\right| \left|\delta_{k}^{\langle 1 \rangle}\tilde{\eta}_{k} - \partial_{x} \eta_{k}\right| \\ 
	\leq & C\left(
		\left|\frac{\tilde{\eta}_{k+1} + \tilde{\eta}_{k-1}}{2} - \eta_{k}\right| 
		+ \left|\delta_{k}^{\langle 1 \rangle}\tilde{\eta}_{k} - \partial_{x} \eta_{k}\right|
		\right) 
		\quad (k=0, \ldots, K). 
\end{align*}
Additionally, substituting $t=(j + 1)\Delta t$ in \eqref{eta_avg} and using (A2), we get
\begin{equation*}
	\left|\frac{\tilde{\eta}_{k+1} + \tilde{\eta}_{k-1}}{2} - \eta_{k}\right| 
	\leq \frac{(\Delta x)^{2}}{4}\left( \left|\partial_{x}^{2} \tilde{\eta}_{k + r_{1} }\right| + \left|\partial_{x}^{2} \tilde{\eta}_{k - r_{1} }\right| \right) 
	\leq C(\Delta x)^{2} \quad (k=0, \ldots, K). %\label{eta_avg1}
\end{equation*}
On another note, for $k=0, \ldots, K$, applying the Taylor theorem to $\tilde{\eta}$, there exists $r_{2} \in (0,1)$ such that 
\begin{equation*}
	\frac{\tilde{\eta}_{k+1} - \tilde{\eta}_{k-1}}{2\Delta x} - \partial_{x}\tilde{\eta}_{k} 
	= \frac{\Delta x}{4}\left( \partial_{x}^{2} \tilde{\eta}_{k + r_{2} } - \partial_{x}^{2} \tilde{\eta}_{k - r_{2} } \right) \quad (k=0, \ldots, K). %\label{eta_avg3}
\end{equation*}
Namely, we see from the definition of $\tilde{\eta}$ and (A2) that 
\begin{align*}
	\left|\delta_{k}^{\langle 1 \rangle}\tilde{\eta}_{k} - \partial_{x} \eta_{k}\right| 
	= \left|\frac{\tilde{\eta}_{k+1} - \tilde{\eta}_{k-1}}{2\Delta x} - \partial_{x}\tilde{\eta}_{k}\right| 
	\leq & \frac{\Delta x}{4}\left( \left|\partial_{x}^{2} \tilde{\eta}_{k + r_{2} } \right|
	+ \left|\partial_{x}^{2} \tilde{\eta}_{k - r_{2} }\right| \right) \\
	\leq & C\Delta x \quad (k=0, \ldots, K). 
\end{align*}
As a result, we have 
\begin{equation*}
	\left| \delta_{k}^{\langle 1 \rangle}\!\left( \alpha( \tilde{\eta}_{k} ) \right) - \partial_{x} \alpha( \eta_{k} ) \right| 
	\leq C\left( \Delta x + (\Delta x)^{2} \right) \quad (k=0, \ldots, K). 
\end{equation*}
Next, we consider the following difference:  
\begin{align*}
& \gamma_{\varepsilon}'(\delta_{k}^{+}\tilde{\theta}_{k}) + \gamma_{\varepsilon}'(\delta_{k}^{-}\tilde{\theta}_{k}) - 2\gamma_{\varepsilon}'(\partial_{x}\theta_{k}) \\
	% = & \frac{\delta_{k}^{+}\theta_{k}^{(j+1)}}{ \mo{\gamma_{\varepsilon}}(\delta_{k}^{+}\theta_{k}^{(j+1)}) }
	%  + \frac{\delta_{k}^{-}\theta_{k}^{(j+1)}}{ \mo{\gamma_{\varepsilon}}(\delta_{k}^{-}\theta_{k}^{(j+1)}) }
	%  - \frac{2\partial_{x}\theta_{k}^{(j+1)}}{ \mo{\gamma_{\varepsilon}}(\partial_{x}\theta_{k}^{(j+1)}) } \\
	= & \frac{\delta_{k}^{+}\tilde{\theta}_{k}}{ \gamma_{\varepsilon}(\delta_{k}^{+}\tilde{\theta}_{k}) }
	 - \frac{\partial_{x}\theta_{k}}{ \gamma_{\varepsilon}(\partial_{x}\theta_{k}) }
	 + \frac{\delta_{k}^{-}\tilde{\theta}_{k}}{ \gamma_{\varepsilon}(\delta_{k}^{-}\tilde{\theta}_{k}) }
	 - \frac{\partial_{x}\theta_{k}}{ \gamma_{\varepsilon}(\partial_{x}\theta_{k}) } \\
	= & \frac{\delta_{k}^{+}\tilde{\theta}_{k} \! - \! \partial_{x}\theta_{k}}{ \gamma_{\varepsilon}(\delta_{k}^{+}\tilde{\theta}_{k}) } \! + \! \partial_{x}\theta_{k} \! \left(\! \frac{1}{ \gamma_{\varepsilon}(\delta_{k}^{+}\tilde{\theta}_{k}) } \! - \! \frac{1}{ \gamma_{\varepsilon}(\partial_{x}\theta_{k}) } \!\right) \\
	& + \frac{\delta_{k}^{-}\tilde{\theta}_{k} \! - \! \partial_{x}\theta_{k}}{ \gamma_{\varepsilon}(\delta_{k}^{-}\tilde{\theta}_{k}) } \! + \! \partial_{x}\theta_{k} \! \left(\! \frac{1}{ \gamma_{\varepsilon}(\delta_{k}^{-}\tilde{\theta}_{k}) } \! - \! \frac{1}{ \gamma_{\varepsilon}(\partial_{x}\theta_{k}) } \!\right) \quad (k=0, \ldots, K). 
\end{align*}
By direct calculation, we see that 
\begin{align*}
	\frac{1}{ \gamma_{\varepsilon}(\delta_{k}^{\ast}\tilde{\theta}_{k}) } - \frac{1}{ \gamma_{\varepsilon}(\partial_{x}\theta_{k}) } 
	= & \frac{ \gamma_{\varepsilon}(\partial_{x}\theta_{k}) - \gamma_{\varepsilon}(\delta_{k}^{\ast}\tilde{\theta}_{k}) }{ \gamma_{\varepsilon}(\delta_{k}^{\ast}\tilde{\theta}_{k}) \gamma_{\varepsilon}(\partial_{x}\theta_{k}) } \\
	= & \frac{ \left(\gamma_{\varepsilon}(\partial_{x}\theta_{k}) \right)^{2} - \left(\gamma_{\varepsilon}(\delta_{k}^{\ast}\tilde{\theta}_{k}) \right)^{2} }{ \gamma_{\varepsilon}(\delta_{k}^{\ast}\tilde{\theta}_{k}) \gamma_{\varepsilon}(\partial_{x}\theta_{k}) \left(\gamma_{\varepsilon}(\partial_{x}\theta_{k}) + \gamma_{\varepsilon}(\delta_{k}^{\ast}\tilde{\theta}_{k})\right) } \\
	= & \frac{ \partial_{x}\theta_{k} + \delta_{k}^{\ast}\tilde{\theta}_{k} }{ \gamma_{\varepsilon}(\delta_{k}^{\ast}\tilde{\theta}_{k}) \gamma_{\varepsilon}(\partial_{x}\theta_{k}) \left(\gamma_{\varepsilon}(\partial_{x}\theta_{k}) \! + \! \gamma_{\varepsilon}(\delta_{k}^{\ast}\tilde{\theta}_{k})\right) } \left(\partial_{x}\theta_{k} - \delta_{k}^{\ast}\tilde{\theta}_{k}\right)
\end{align*}
for $k=0, \ldots, K$, where the symbol ``$\ast$'' denotes $+$ or $-$. 
From the above, we obtain from (A2) and the inequality $1/\gamma_{\varepsilon}(u) \leq 1/\varepsilon$ that 
\begin{align*}
	& \left|\gamma_{\varepsilon}'(\delta_{k}^{+}\tilde{\theta}_{k}) + \gamma_{\varepsilon}'(\delta_{k}^{-}\tilde{\theta}_{k}) - 2\gamma_{\varepsilon}'(\partial_{x}\theta_{k})\right| \\
	\leq & \frac{ |\delta_{k}^{+}\tilde{\theta}_{k} \! - \! \partial_{x}\theta_{k}| }{ \gamma_{\varepsilon}(\delta_{k}^{+}\tilde{\theta}_{k}) } \! + \! |\partial_{x}\theta_{k}| \! \left| \frac{1}{ \gamma_{\varepsilon}(\delta_{k}^{+}\tilde{\theta}_{k}) } \! - \! \frac{1}{ \gamma_{\varepsilon}(\partial_{x}\theta_{k}) } \right| 
	+ \frac{ |\delta_{k}^{-}\tilde{\theta}_{k} \! - \! \partial_{x}\theta_{k}| }{ \gamma_{\varepsilon}(\delta_{k}^{-}\tilde{\theta}_{k}) } \! + \! |\partial_{x}\theta_{k}| \! \left| \frac{1}{ \gamma_{\varepsilon}(\delta_{k}^{-}\tilde{\theta}_{k}) } \! - \! \frac{1}{ \gamma_{\varepsilon}(\partial_{x}\theta_{k}) } \right| \\
	\leq & \frac{1}{\varepsilon}|\delta_{k}^{+}\tilde{\theta}_{k} \! - \! \partial_{x}\theta_{k}| \! + \! C \left| \frac{1}{ \gamma_{\varepsilon}(\delta_{k}^{+}\tilde{\theta}_{k}) } \! - \! \frac{1}{ \gamma_{\varepsilon}(\partial_{x}\theta_{k}) } \right| 
	+ \frac{1}{\varepsilon}|\delta_{k}^{-}\tilde{\theta}_{k} \! - \! \partial_{x}\theta_{k}| \! + \! C \left| \frac{1}{ \gamma_{\varepsilon}(\delta_{k}^{-}\tilde{\theta}_{k}) } \! - \! \frac{1}{ \gamma_{\varepsilon}(\partial_{x}\theta_{k}) } \right| \\
	\leq & C\left(
		|\delta_{k}^{+}\tilde{\theta}_{k} - \partial_{x}\theta_{k}| 
		+ |\delta_{k}^{-}\tilde{\theta}_{k} - \partial_{x}\theta_{k}|
	\right) \quad (k=0, \ldots, K). 
\end{align*}
Here, 
% for $k=0, \ldots, K$, applying the Taylor theorem to $\theta$, there exists $r_{3} \in (0,1)$ such that 
% \begin{equation*}
% \partial_{x} \theta_{k}^{(j+1)} - \delta_{k}^{-}\theta_{k}^{(j+1)}  
% 	= \frac{\Delta x}{2}\partial_{x}^{2} \theta_{k-r_{3}}^{(j+1)} \quad (k=0, \ldots, K). %\label{et1}
% \end{equation*}
% Likewise, for $k=0, \ldots, K$, applying the Taylor theorem to $\theta$, there exists $r_{4} \in (0,1)$ such that 
% \begin{equation*}
% \partial_{x} \theta_{k}^{(j+1)} - \delta_{k}^{+}\theta_{k}^{(j+1)}  
% 	= - \frac{\Delta x}{2}\partial_{x}^{2} \theta_{k+r_{4}}^{(j+1)} \quad (k=0, \ldots, K). %\label{et1}
% \end{equation*}
applying the Taylor theorem to $\tilde{\theta}$, we see from the definition of $\tilde{\theta}$ and (A2) that 
\begin{equation*}
	\left|\delta_{k}^{+}\tilde{\theta}_{k} - \partial_{x}\theta_{k}\right| 
	\leq C\Delta x \quad (k=0, \ldots, K), \quad 
	\left|\delta_{k}^{-}\tilde{\theta}_{k} - \partial_{x}\theta_{k}\right| 
	\leq C\Delta x \quad (k=0, \ldots, K).
\end{equation*}
Therefore, we have 
\begin{equation*}
	\left|\gamma_{\varepsilon}'(\delta_{k}^{+}\tilde{\theta}_{k}) + \gamma_{\varepsilon}'(\delta_{k}^{-}\tilde{\theta}_{k}) - 2\gamma_{\varepsilon}'(\partial_{x}\theta_{k})\right|
	\leq C\Delta x \quad (k=0, \ldots, K).
\end{equation*}
From the above, we estimate the first term on the right-hand side of \eqref{xi6_1} as follows: 
\begin{equation*}
	\left|
		\delta_{k}^{\langle 1 \rangle} \! \left( \alpha( \tilde{\eta}_{k} ) \right) 
		\! \left(\gamma_{\varepsilon}'(\delta_{k}^{+}\tilde{\theta}_{k}) \! + \! \gamma_{\varepsilon}'(\delta_{k}^{-}\tilde{\theta}_{k})\right) 
		\! - \! 2\partial_{x} \alpha( \eta_{k} ) \gamma_{\varepsilon}'(\partial_{x}\theta_{k})
	\right| 
	\leq C\!\left( \Delta x \! + \! (\Delta x)^{2} \right) \quad (k=0, \ldots, K). 
\end{equation*}
Lastly, we evaluate the second term on the right-hand side of \eqref{xi6_1}. 
By direct calculation, we have 
\begin{align*}
	& \left| \mu_{k}^{+}\!\left( \alpha( \tilde{\eta}_{k} ) \right)\!\delta_{k}^{+}\!\left(  \gamma_{\varepsilon}'(\delta_{k}^{-}\tilde{\theta}_{k}) \right) 
	+ \mu_{k}^{-}\!\left( \alpha( \tilde{\eta}_{k} ) \right)\!\delta_{k}^{-}\!\left(  \gamma_{\varepsilon}'(\delta_{k}^{+}\tilde{\theta}_{k}) \right) 
	- 2\alpha( \eta_{k} )\partial_{x}\gamma_{\varepsilon}'(\partial_{x}\theta_{k}) 
	\right| \\
	\leq & \left|
		\mu_{k}^{+}\!\left( \alpha( \tilde{\eta}_{k} ) \right)\!\delta_{k}^{+}\!\left(  \gamma_{\varepsilon}'(\delta_{k}^{-}\tilde{\theta}_{k}) \right) 
		- \alpha( \eta_{k} )\partial_{x}\gamma_{\varepsilon}'(\partial_{x}\theta_{k}) 
	\right| \\
	& + \left|
		\mu_{k}^{-}\!\left( \alpha( \tilde{\eta}_{k} ) \right)\!\delta_{k}^{-}\!\left(  \gamma_{\varepsilon}'(\delta_{k}^{+}\tilde{\theta}_{k}) \right) 
		- \alpha( \eta_{k} )\partial_{x}\gamma_{\varepsilon}'(\partial_{x}\theta_{k}) 
	\right|
	\quad (k=0, \ldots, K).
\end{align*}
Moreover, we see from (A2), the following inequality: $\gamma_{\varepsilon}''(u) \leq 1/\varepsilon$, and \eqref{diff_f} that 
\begin{align*}
	& \left| \mu_{k}^{+}\!\left( \alpha( \tilde{\eta}_{k} ) \right)\delta_{k}^{+}\!\left(  \gamma_{\varepsilon}'(\delta_{k}^{-}\tilde{\theta}_{k}) \right) - \alpha( \eta_{k} )\partial_{x}\gamma_{\varepsilon}'(\partial_{x}\theta_{k}) \right| \\
	\leq & \left| \mu_{k}^{+}\!\left( \alpha( \tilde{\eta}_{k} ) \right) \right| 
	\left|\delta_{k}^{+}\!\left(  \gamma_{\varepsilon}'(\delta_{k}^{-}\tilde{\theta}_{k}) \right) - \partial_{x}\gamma_{\varepsilon}'(\partial_{x}\theta_{k}) \right| 
	+ \left|\mu_{k}^{+}\!\left( \alpha( \tilde{\eta}_{k} ) \right) - \alpha( \eta_{k} )\right| |\partial_{x}\gamma_{\varepsilon}'(\partial_{x}\theta_{k})| \\
	\leq & C \left|
		\frac{d\gamma_{\varepsilon}'}{d(\delta_{k}^{-}\tilde{\theta}_{k+1}, \delta_{k}^{-}\tilde{\theta}_{k})}\delta_{k}^{\langle 2\rangle}\tilde{\theta}_{k}
		- \gamma_{\varepsilon}''(\partial_{x}\theta_{k})\partial_{x}^{2}\theta_{k} 
		\right| 
	+ \frac{1}{4}|\tilde{\eta}_{k+1} + \tilde{\eta}_{k}| |\tilde{\eta}_{k+1} - \tilde{\eta}_{k}| 
	\left|\gamma_{\varepsilon}''(\partial_{x}\theta_{k}) \right| \left|\partial_{x}^{2}\theta_{k} \right| \\
	\leq & C\left(
		\left|\frac{d\gamma_{\varepsilon}'}{d(\delta_{k}^{-}\tilde{\theta}_{k+1}, \delta_{k}^{-}\tilde{\theta}_{k})}\right| 
		\left| \delta_{k}^{\langle 2\rangle}\tilde{\theta}_{k} - \partial_{x}^{2}\theta_{k}\right| 
		+ \left| \frac{d\gamma_{\varepsilon}'}{d(\delta_{k}^{-}\tilde{\theta}_{k+1}, \delta_{k}^{-}\tilde{\theta}_{k})} - \gamma_{\varepsilon}''(\partial_{x}\theta_{k})\right| 
		\left|\partial_{x}^{2}\theta_{k} \right|
	\right) \\
	& + C|\tilde{\eta}_{k+1} - \tilde{\eta}_{k}| \\
	\leq & C\left(
		\left| \delta_{k}^{\langle 2\rangle}\tilde{\theta}_{k} - \partial_{x}^{2}\theta_{k}\right| 
		+ \left| \frac{d\gamma_{\varepsilon}'}{d(\delta_{k}^{-}\tilde{\theta}_{k+1}, \delta_{k}^{-}\tilde{\theta}_{k})} - \gamma_{\varepsilon}''(\partial_{x}\theta_{k})\right| 
		+ |\tilde{\eta}_{k+1} - \tilde{\eta}_{k}| 
	\right) \quad (k=0, \ldots, K).
\end{align*}
Additionally, applying the Taylor theorem to $\tilde{\eta}$, we see from the definition of $\tilde{\eta}$ and (A2) that 
\begin{equation*}
	|\tilde{\eta}_{k + 1} - \tilde{\eta}_{k}| \leq C\Delta x \quad (k=0, \ldots, K). 
\end{equation*}
Furthermore, it holds from Lemma \ref{lem:4.2} that 
\begin{align*}
	\frac{d\gamma_{\varepsilon}'}{d(\delta_{k}^{-}\tilde{\theta}_{k+1}, \delta_{k}^{-}\tilde{\theta}_{k})} \! - \! \gamma_{\varepsilon}''(\partial_{x}\theta_{k}) 
	= & \frac{d\gamma_{\varepsilon}'}{d(\delta_{k}^{-}\tilde{\theta}_{k+1}, \delta_{k}^{-}\tilde{\theta}_{k})} - \frac{d\gamma_{\varepsilon}'}{d(\partial_{x}\theta_{k}, \partial_{x}\theta_{k})} \\
	= & \frac{1}{2}\overline{\gamma_{\varepsilon}'}^{\prime\prime}(\delta_{k}^{-}\tilde{\theta}_{k+1},\partial_{x}\theta_{k}; \delta_{k}^{-}\tilde{\theta}_{k},\partial_{x}\theta_{k})(\delta_{k}^{-}\tilde{\theta}_{k+1} \! - \! \partial_{x}\theta_{k}) \\
	& + \frac{1}{2}\overline{\gamma_{\varepsilon}'}^{\prime\prime}\!(\delta_{k}^{-}\tilde{\theta}_{k},\partial_{x}\theta_{k}; \delta_{k}^{-}\tilde{\theta}_{k+1},\partial_{x}\theta_{k})(\delta_{k}^{-}\tilde{\theta}_{k} \! - \! \partial_{x}\theta_{k}) \\
	= & \frac{1}{2}\overline{\gamma_{\varepsilon}'}^{\prime\prime}(\delta_{k}^{-}\tilde{\theta}_{k+1},\partial_{x}\theta_{k}; \delta_{k}^{-}\tilde{\theta}_{k},\partial_{x}\theta_{k})(\delta_{k}^{+}\tilde{\theta}_{k} \! - \! \partial_{x}\theta_{k}) \\
	& + \frac{1}{2}\overline{\gamma_{\varepsilon}'}^{\prime\prime}\!(\delta_{k}^{-}\tilde{\theta}_{k},\partial_{x}\theta_{k}; \delta_{k}^{-}\tilde{\theta}_{k+1},\partial_{x}\theta_{k})(\delta_{k}^{-}\tilde{\theta}_{k} \! - \! \partial_{x}\theta_{k})
\end{align*}
for $k=0, \ldots, K$. 
In addition, it follows from Lemma \ref{lem:4.1} that 
\begin{equation*}
	|\overline{\gamma_{\varepsilon}'}^{\prime\prime}(s_{1}, s_{2}; s_{3}, s_{4})| \leq \sup_{s \in \mathbb{R}}|\gamma_{\varepsilon}'''(s)| \quad \mbox{for all} \ s_{1}, s_{2}, s_{3}, s_{4} \in \mathbb{R}.
\end{equation*}
Thus, from the following inequality: 
\begin{equation*}
	|\gamma_{\varepsilon}'''(u)| 
	= \frac{3\varepsilon^{2}|u|}{(\varepsilon^{2} + u^{2})^{\frac{5}{2}}}
	\leq \frac{3\varepsilon^{2}(\varepsilon^{2} + u^{2})^{\frac{1}{2}}}{(\varepsilon^{2} + u^{2})^{\frac{5}{2}}}
	= \frac{3\varepsilon^{2}}{(\varepsilon^{2} + u^{2})^{2}}
	\leq \frac{3\varepsilon^{2}}{\varepsilon^{4}} 
	= \frac{3}{\varepsilon^{2}} \quad \mbox{for all} \ u \in \mathbb{R}, 
\end{equation*}
we obtain 
\begin{equation*}
	\left| \frac{d\gamma_{\varepsilon}'}{d(\delta_{k}^{-}\tilde{\theta}_{k+1}, \delta_{k}^{-}\tilde{\theta}_{k})} \! - \! \gamma_{\varepsilon}''(\partial_{x}\theta_{k})\right| 
	\! \leq \! \frac{3}{2\varepsilon^{2}}\!\left(
		\left|\delta_{k}^{+}\tilde{\theta}_{k} \! - \! \partial_{x}\theta_{k}\right| 
		\! + \! \left|\delta_{k}^{-}\tilde{\theta}_{k} \! - \! \partial_{x}\theta_{k}\right| 
	\right)
	\! \leq \! C\Delta x \quad (k=0, \ldots, K). 
\end{equation*}
From the above, we estimate the second term on the right-hand side of \eqref{xi6_1} as follows: 
\begin{align*}
	& \left| \mu_{k}^{+}\!\left( \alpha( \tilde{\eta}_{k} ) \right)\!\delta_{k}^{+}\!\left(  \gamma_{\varepsilon}'(\delta_{k}^{-}\tilde{\theta}_{k}) \right) 
	+ \mu_{k}^{-}\!\left( \alpha( \tilde{\eta}_{k} ) \right)\!\delta_{k}^{-}\!\left(  \gamma_{\varepsilon}'(\delta_{k}^{+}\tilde{\theta}_{k}) \right) 
	- 2\alpha( \eta_{k} )\partial_{x}\gamma_{\varepsilon}'(\partial_{x}\theta_{k}) 
	\right| \\
	\leq & C\left( (\Delta x)^{\sigma} + \Delta x \right) \quad (k=0, \ldots, K). 
\end{align*}
Therefore, we estimate $\bm{\xi}_{6}^{(j+1)}$ as follows: 
\begin{equation}
	\left|\xi_{6, k}^{(j+1)}\right| \leq C\left( (\Delta x)^{\sigma} + \Delta x + (\Delta x)^{2} \right) \quad (k=0, \ldots, K).
	\label{xi6_est}
\end{equation}
Hence, we obtain from \eqref{xi1_est}, \eqref{xi2_est}, and \eqref{xi3_est} that 
\begin{align*}
	\left\| \bm{\xi}_{1-3}^{(j + 1)} \right\|_{ L_{\rm d}^{2} }^{2} 
	\leq & \left(\left\| \bm{\xi}_{1}^{(j + 1)} \right\|_{ L_{\rm d}^{2} } + \left\| \bm{\xi}_{2}^{(j + 1)} \right\|_{ L_{\rm d}^{2} } + \left\| \bm{\xi}_{3}^{(j + 1)} \right\|_{ L_{\rm d}^{2} } \right)^{2} \\
	\leq & C^{2}\left(
		(\Delta t)^{\sigma} + (\Delta x)^{\sigma} + \Delta t + \Delta x + (\Delta x)^{2}
	\right)^{2} \quad (j = 0,1,\ldots, N - 1), 
\end{align*}
where $\bm{\xi}_{1-3}^{(j + 1)} := \{\xi_{1-3,k}^{(j + 1)}\}_{k=0}^{K} = \bm{\xi}_{1}^{(j + 1)} + \bm{\xi}_{2}^{(j + 1)} + \bm{\xi}_{3}^{(j + 1)} \ \in \mathbb{R}^{K+1}$. 
Similarly, we have from \eqref{xi4_est}, \eqref{xi5_est}, and \eqref{xi6_est} that 
\begin{align*}
	\left\| \bm{\xi}_{4-6}^{(j + 1)} \right\|_{ L_{\rm d}^{2} }^{2} 
	\leq & \left(\left\| \bm{\xi}_{4}^{(j + 1)} \right\|_{ L_{\rm d}^{2} } + \left\| \bm{\xi}_{5}^{(j + 1)} \right\|_{ L_{\rm d}^{2} } + \left\| \bm{\xi}_{6}^{(j + 1)} \right\|_{ L_{\rm d}^{2} } \right)^{2} \\
	\leq & C^{2}\left(
		(\Delta t)^{\sigma} + (\Delta x)^{\sigma} + \Delta t + \Delta x + (\Delta x)^{2}
	\right)^{2} \quad (j = 0,1,\ldots, N - 1), 
\end{align*}
where $\bm{\xi}_{4-6}^{(j + 1)} := \{\xi_{4-6,k}^{(j + 1)}\}_{k=0}^{K} = \bm{\xi}_{4}^{(j + 1)} + \bm{\xi}_{5}^{(j + 1)} + \bm{\xi}_{6}^{(j + 1)} \ \in \mathbb{R}^{K+1}$. 
As a result, we see from the above estimates and Lemma \ref{e_eta_theta_est} that 
\begin{align*}
	\left\| \bm{e}_{\eta}^{(j)}\right\|_{ L_{\rm d}^{2} }^{2} 
	+ \left\| \sqrt{\bm{\alpha_{0}^{(j)}}}\bm{e}_{\theta}^{(j)}\right\|_{ L_{\rm d}^{2} }^{2} 
	\leq & C\sum_{i=0}^{j-1}\left( \left\| \bm{\xi}_{1-3}^{(j-i)} \right\|_{L_{\rm d}^{2}}^{2} + \left\| \bm{\xi}_{4-6}^{(j-i)} \right\|_{L_{\rm d}^{2}}^{2}\right)\Delta t \\
	\leq & C\left(
		(\Delta t)^{\sigma} + (\Delta x)^{\sigma} + \Delta t + \Delta x + (\Delta x)^{2}
	\right)^{2}\sum_{i=0}^{N-1} \Delta t\\
	= & CT\left(
		(\Delta t)^{\sigma} + (\Delta x)^{\sigma} + \Delta t + \Delta x + (\Delta x)^{2}
	\right)^{2} \quad (j=1,\ldots,N),
\end{align*} 
where $\sqrt{\bm{\alpha_{0}^{(j)}}}\bm{e}_{\theta}^{(j)}$ is defined by \eqref{vec_a_th} in Lemma \ref{e_theta_est}. 
In other words, we conclude that 
\begin{equation*}
	\sqrt{
		\left\| \bm{e}_{\eta}^{(j)}\right\|_{ L_{\rm d}^{2} }^{2} 
		+ \left\| \sqrt{\bm{\alpha_{0}^{(j)}}}\bm{e}_{\theta}^{(j)}\right\|_{ L_{\rm d}^{2} }^{2}
	}
	\leq C \left(
		(\Delta t)^{\sigma} + (\Delta x)^{\sigma} + \Delta t + \Delta x + (\Delta x)^{2}
	\right) \quad (j=1,\ldots,N). 
\end{equation*}
Here, since it holds from (A1) that $\alpha_{0} \geq \delta_{0}$, we obtain \eqref{e_err_est} and \eqref{t_err_est}. 

\section{Computation examples}
In this section, we demonstrate through computation examples that the numerical solution of our proposed scheme is efficient and that the scheme inherits the dissipative property from the original problem in a discrete sense. 
To obtain the next time-step for the numerical solutions to our nonlinear scheme, we use the package “NLsolve” in Julia, a computer language used in our numerical computations.

We choose $K = 50$ so that $\Delta x = 1/50$. 
Besides, we fix the parameters $\varepsilon = 0.01$, $\delta_{0} = 0.01$, $\nu = 0.01$, $\kappa_{0} = 0.01$, $c = 1.0$, and $\kappa = 0.1$. 

\subsection{Computation example 1}
As the initial condition, we consider 
\begin{gather*}
  \eta(0,x) = \eta_{0}(x) 
    = \begin{cases}
	A\cosh(x) + 1 & (0 \leq x < 0.5),\\
	B\cosh(x - 1) + 1 & (0.5 \leq x \leq 1),
      \end{cases}\\
      \theta(0,x) = \theta_{0}(x) = 0.5\pi x - 0.25\pi, 
\end{gather*}
where 
\begin{gather*}
	A = B = -\frac{1}{\cosh(0.5)}\frac{|\theta_{2} - \theta_{1}|}{|\theta_{2} - \theta_{1}| + \tanh(0.5) + \tanh(0.5)}, \\
	\theta_{1} = -0.25\pi, \quad \theta_{2} = 0.25\pi.
\end{gather*}
Also, we choose $\Delta t = 0.06$. 

\begin{figure}[H]
	\begin{minipage}{0.49\hsize}
		\centering
		\includegraphics[width=50mm]{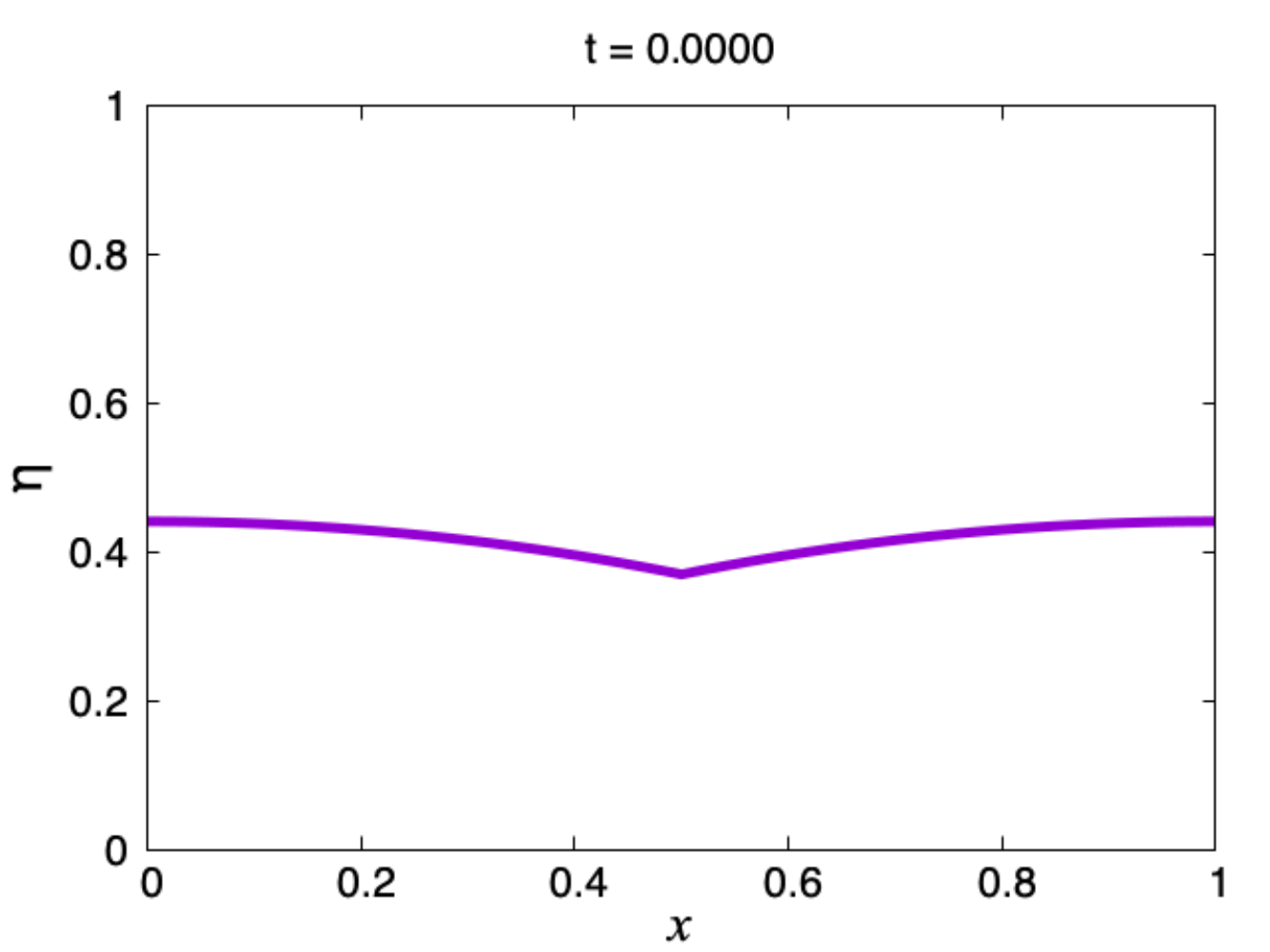}
	\end{minipage}
	\begin{minipage}{0.49\hsize}
		\centering
		\includegraphics[width=50mm]{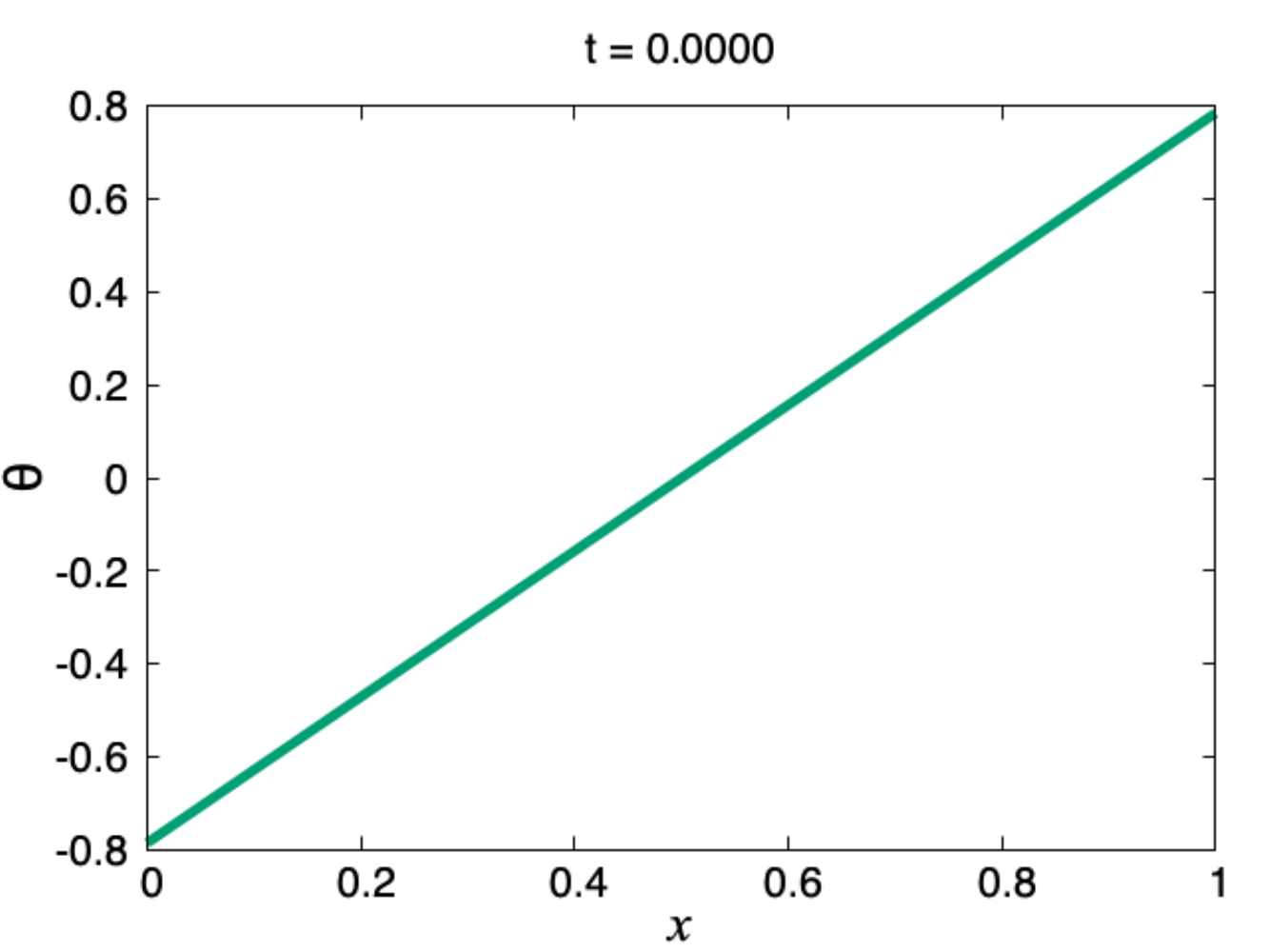}
	\end{minipage}
	\caption{\protect\raggedright The initial data}
	% \label{fig:}
\end{figure} 
Figures \ref{fig:eta_init1}--\ref{fig:theta_init1} show the time development of the solutions to our scheme, respectively. 
Figure \ref{fig:energy_init1} shows the time development of $\mathscr{F}_{1, \rm d}$. 
This graph shows that the energy decreases numerically.
\begin{figure}[H]
	\setlength\abovecaptionskip{0pt}
	\begin{minipage}{0.24\hsize}
		\centering
		\includegraphics[width=38mm]{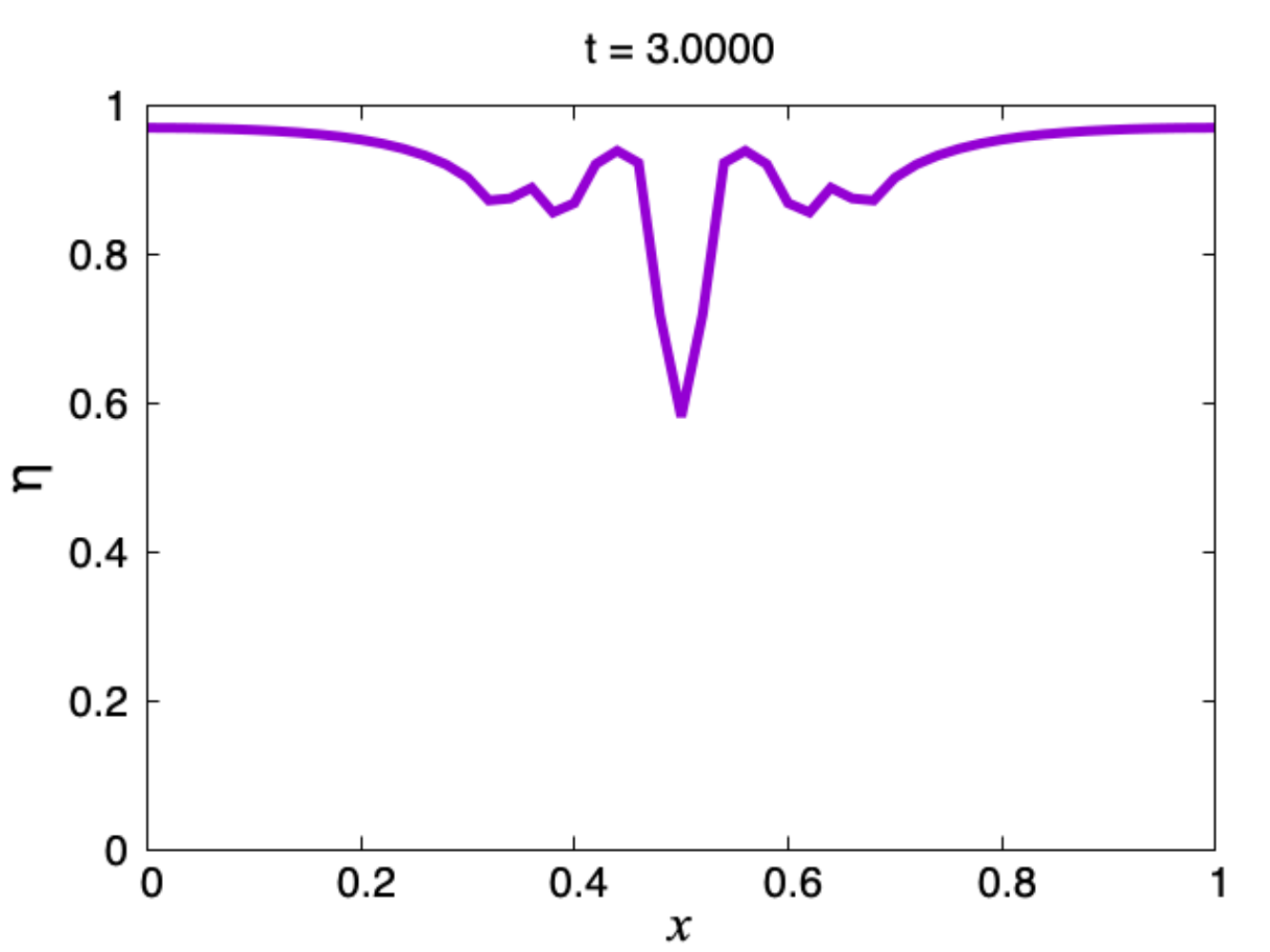}
	\end{minipage}
	\begin{minipage}{0.24\hsize}
		\centering
		\includegraphics[width=38mm]{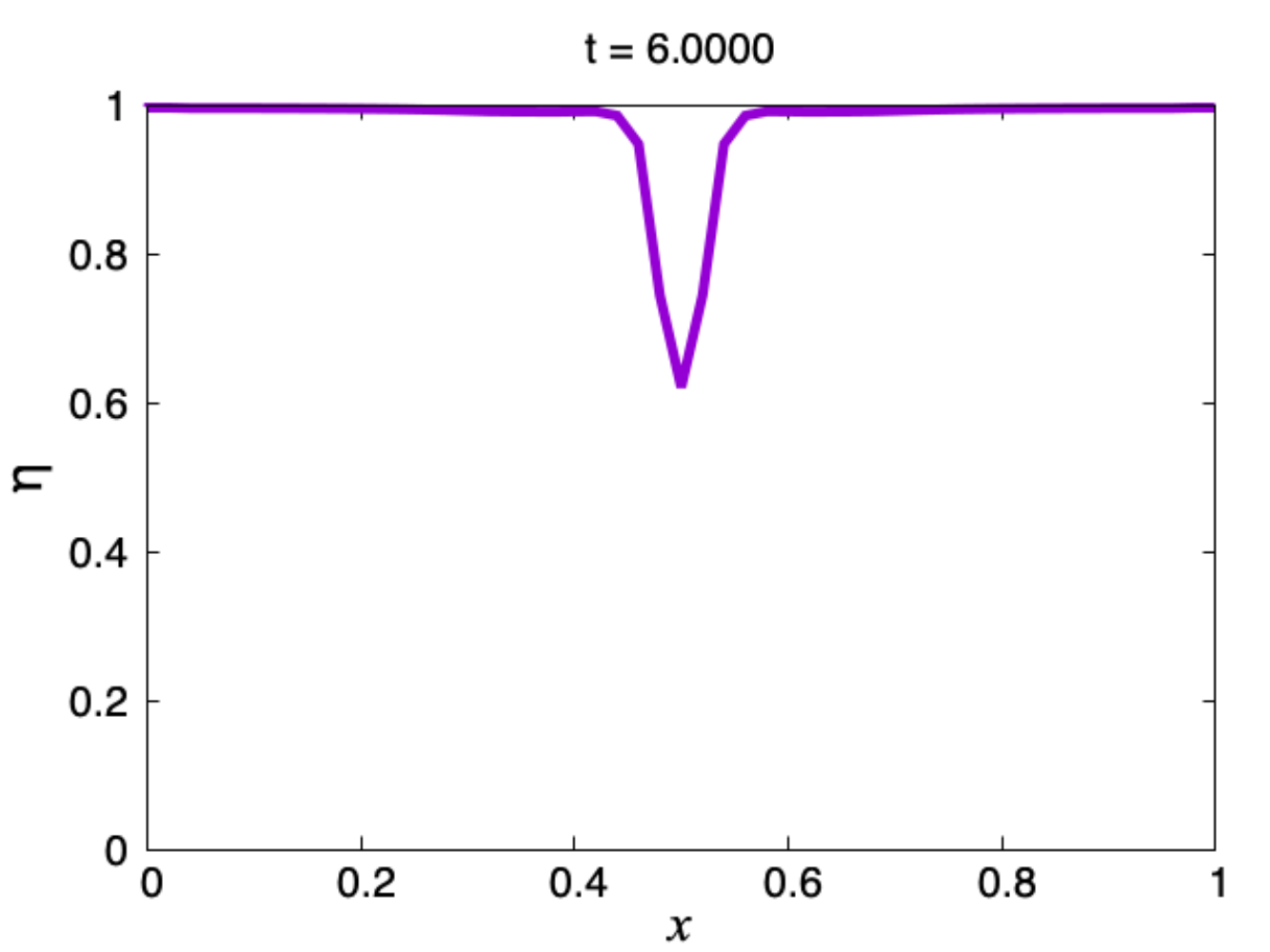}
	\end{minipage}
	\begin{minipage}{0.24\hsize}
		\centering
		\includegraphics[width=38mm]{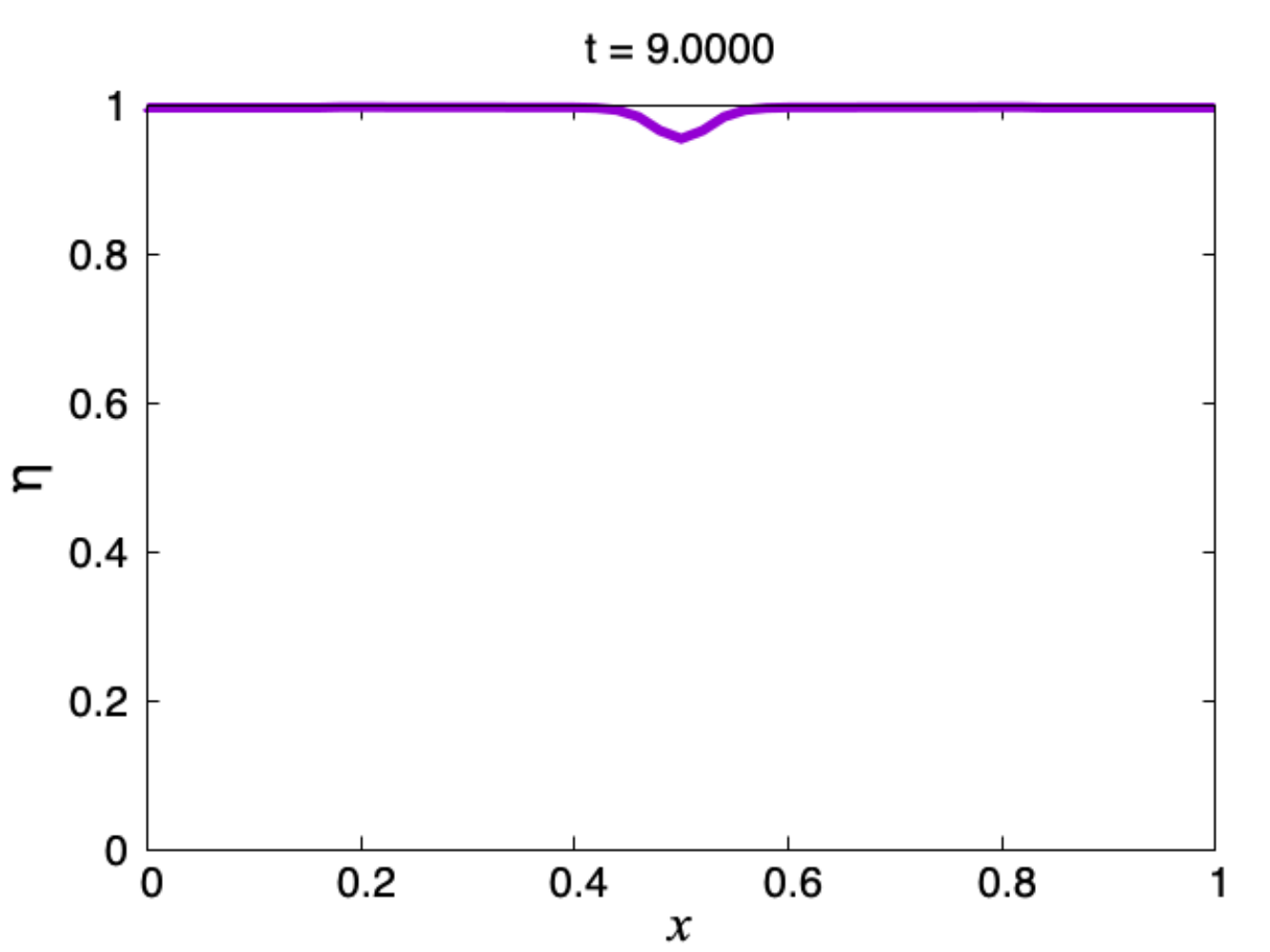}
	\end{minipage}
	\begin{minipage}{0.24\hsize}
		\centering
		\includegraphics[width=38mm]{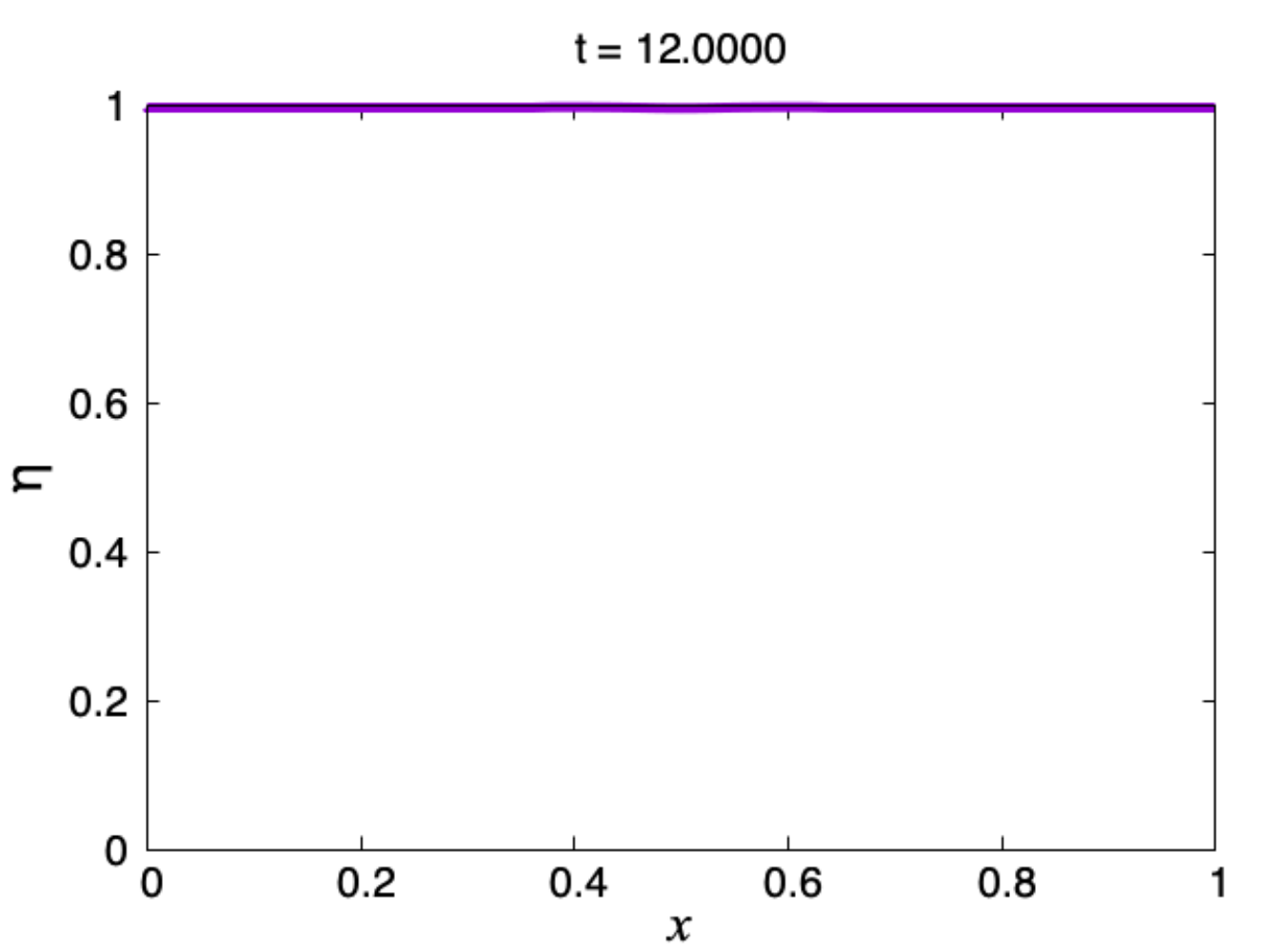}
	\end{minipage}
	\caption{\protect\raggedright Numerical solutions $\bm{H}^{(j)}$ to our scheme at time $t = 3.0$, $t = 6.0$, $t = 9.0$, $t = 12.0$}
	\label{fig:eta_init1}
\end{figure} 

\begin{figure}[H]
	\setlength\abovecaptionskip{0pt}
	\begin{minipage}{0.24\hsize}
		\centering
		\includegraphics[width=38mm]{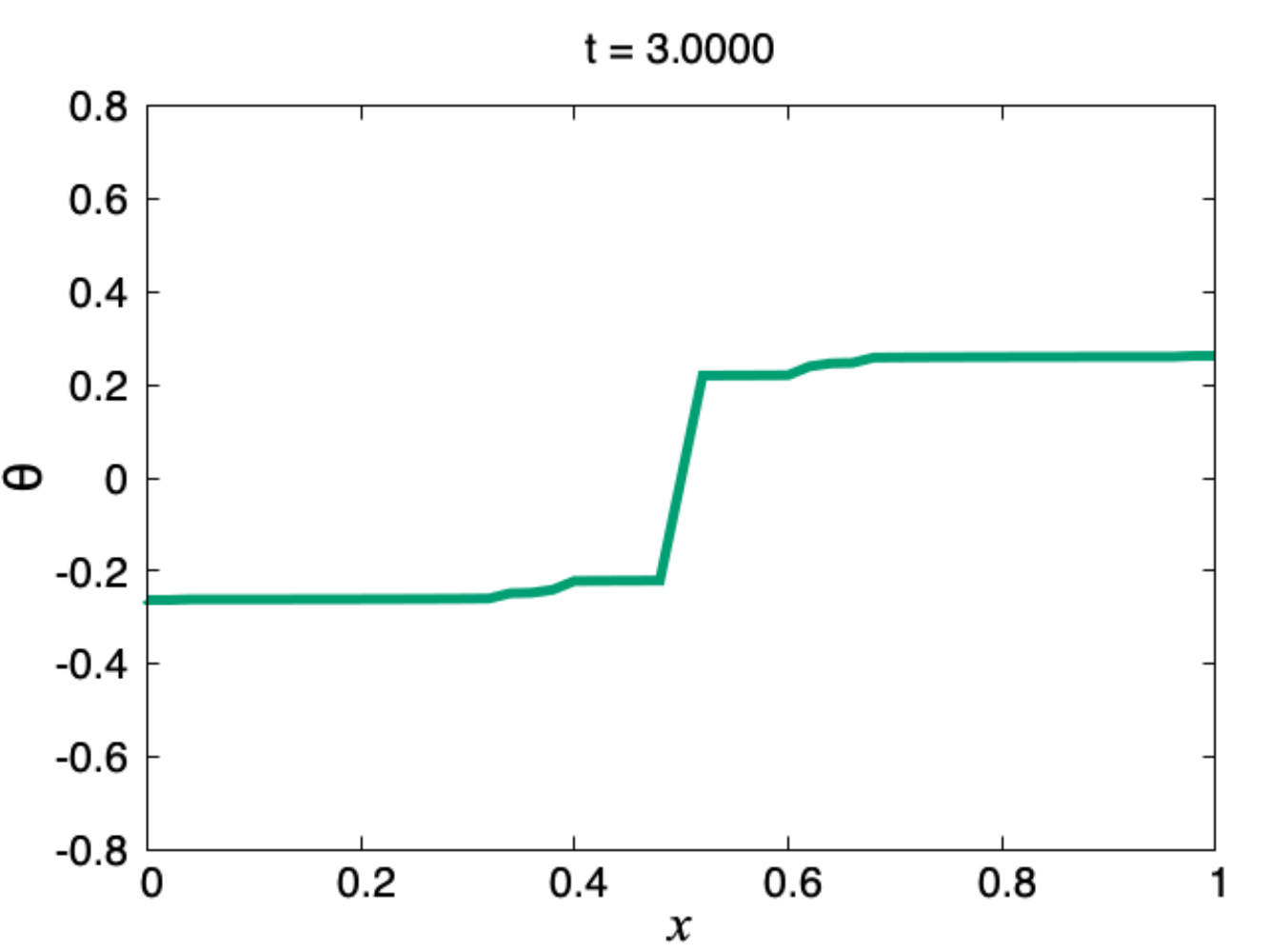}
	\end{minipage}
	\begin{minipage}{0.24\hsize}
		\centering
		\includegraphics[width=38mm]{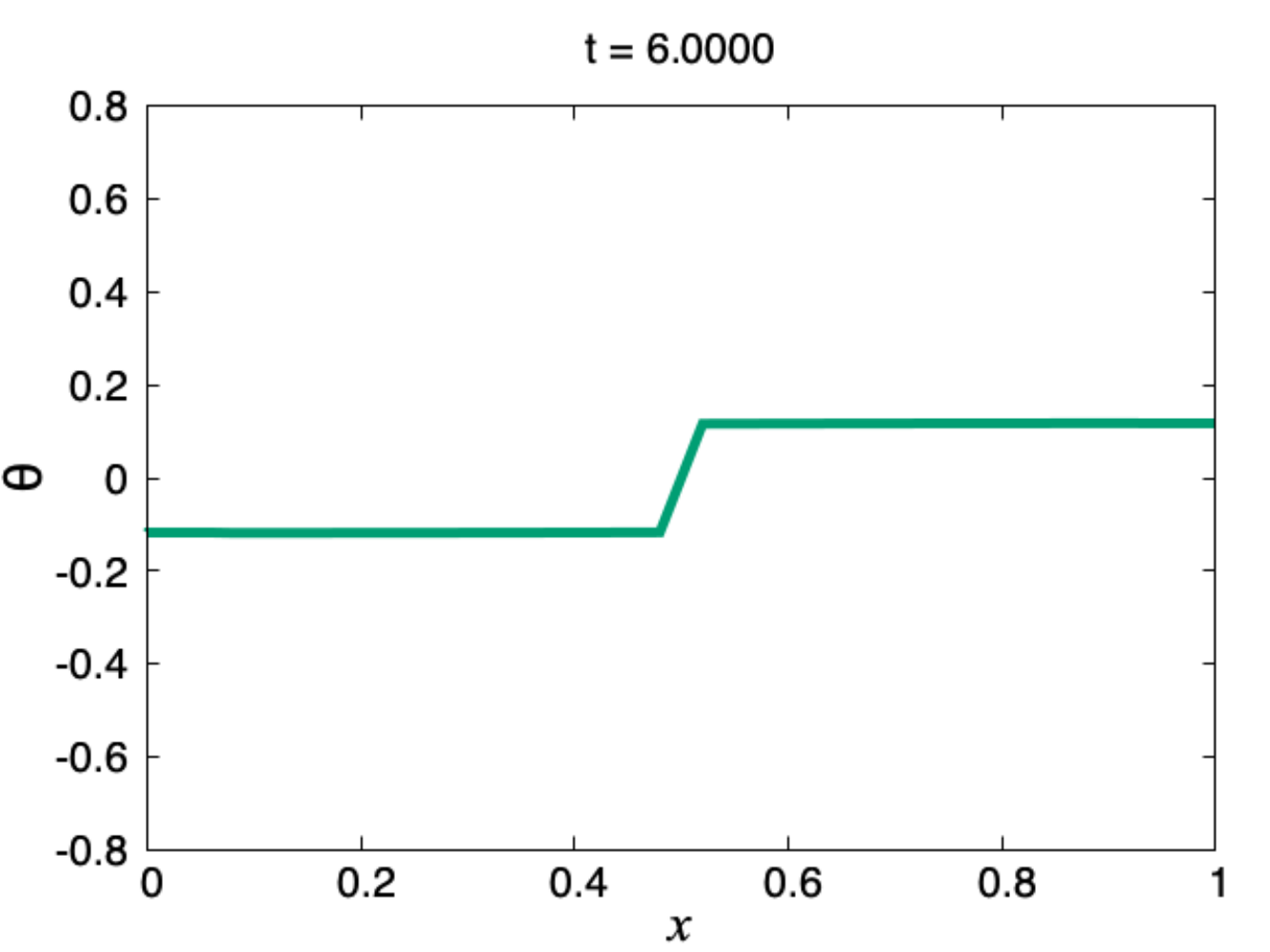}
	\end{minipage}
	\begin{minipage}{0.24\hsize}
		\centering
		\includegraphics[width=38mm]{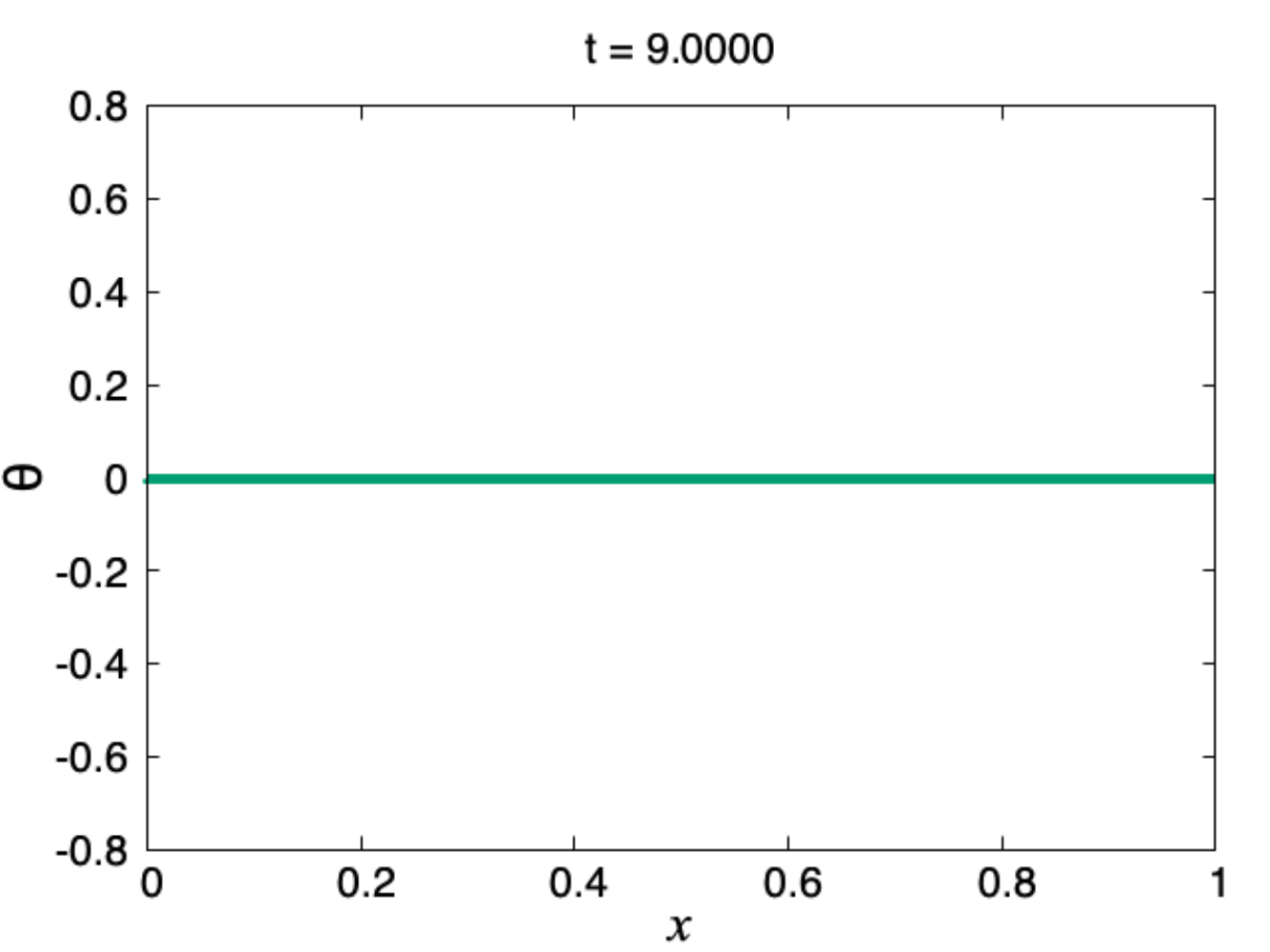}
	\end{minipage}
	\begin{minipage}{0.24\hsize}
		\centering
		\includegraphics[width=38mm]{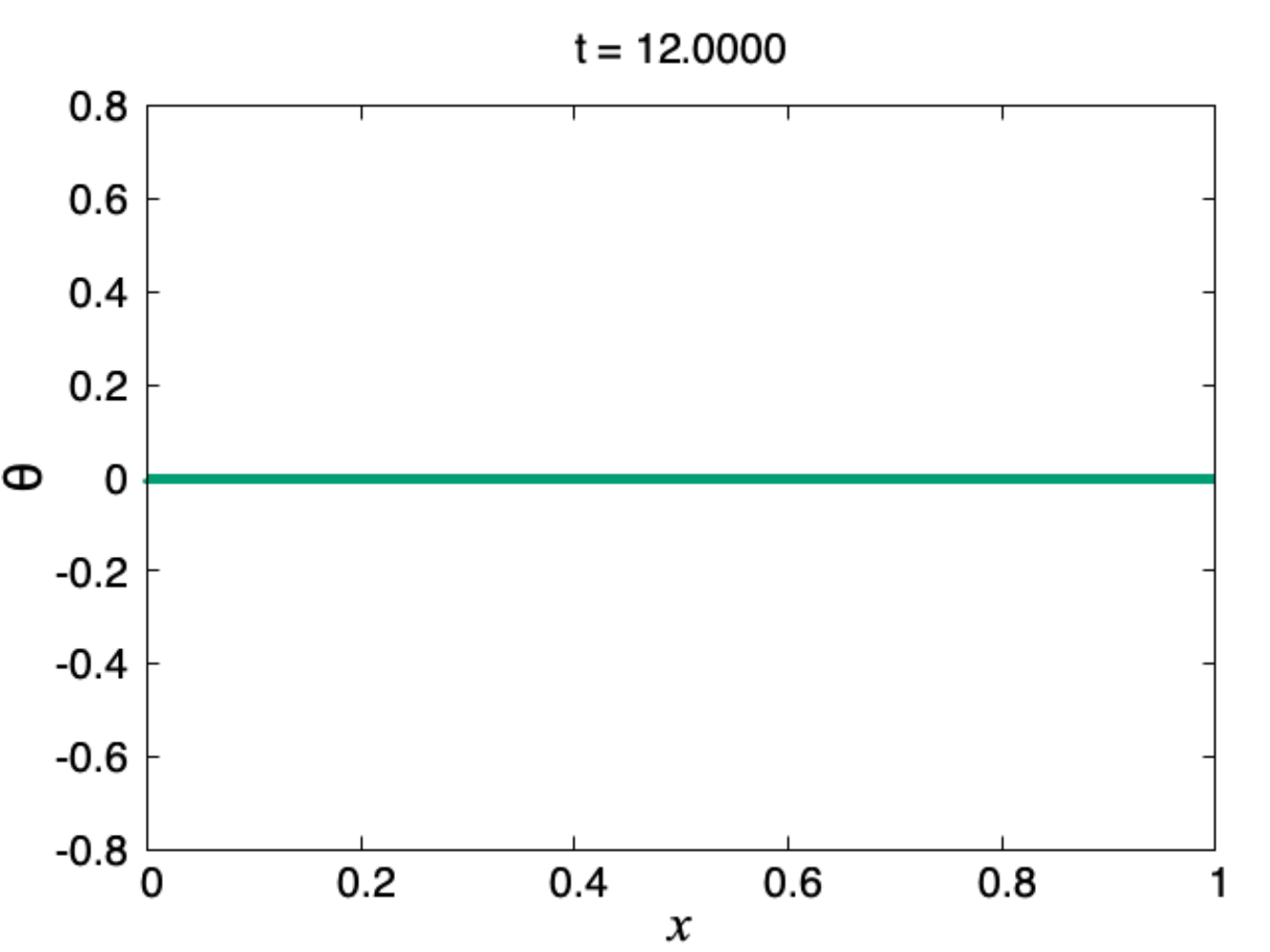}
	\end{minipage}
	\caption{\protect\raggedright Numerical solutions $\bm{\Theta}^{(j)}$ to our scheme at time $t = 3.0$, $t = 6.0$, $t = 9.0$, $t = 12.0$}
	\label{fig:theta_init1}
\end{figure} 

\begin{figure}[H]
	\centering
	\includegraphics[width=60mm]{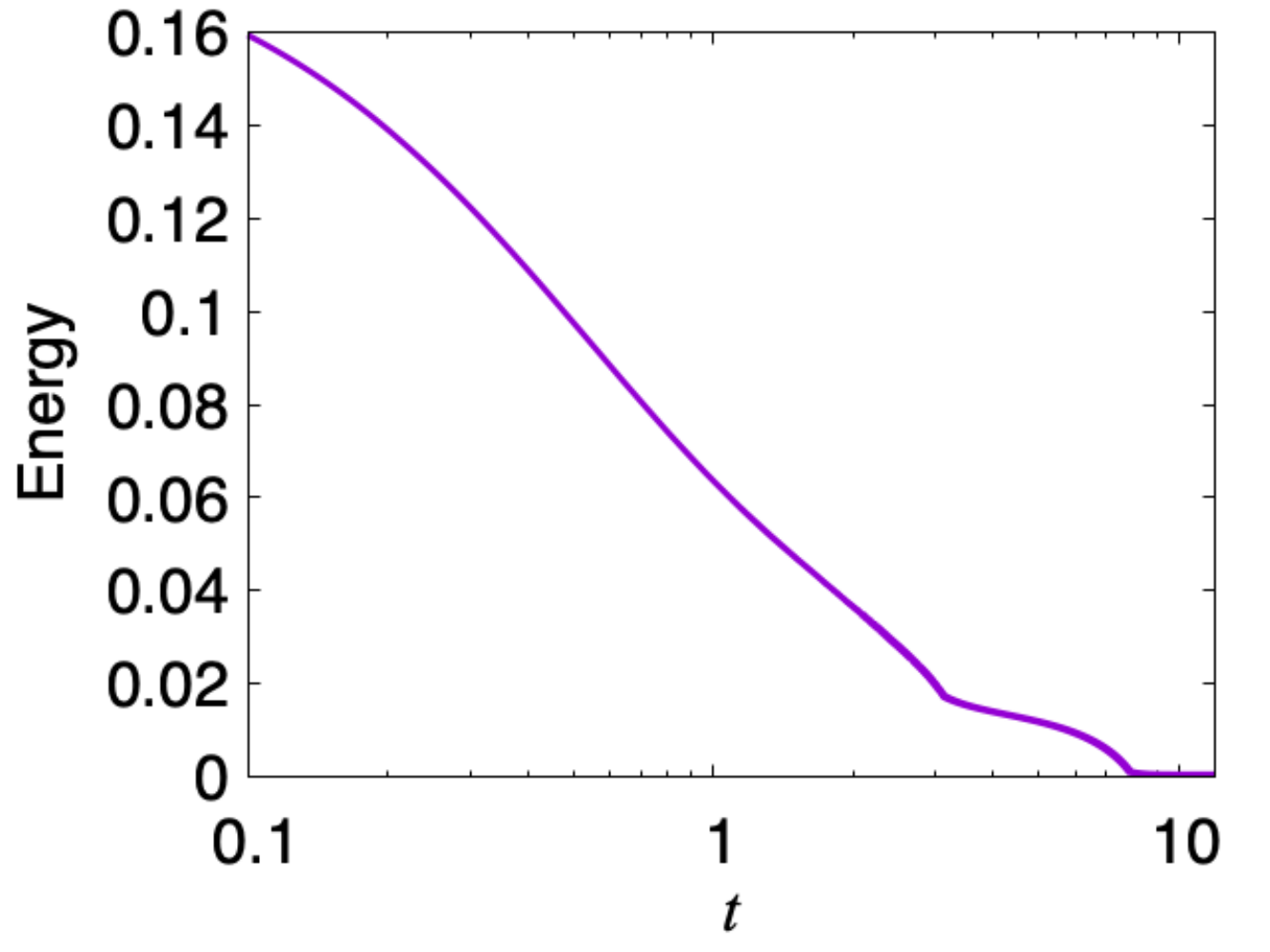}
	\caption{\protect\raggedright Time development of energy (The time axis is in log-scale)}
	\label{fig:energy_init1}
\end{figure}

\subsection{Computation example 2}
As the initial condition, we consider 
\begin{gather*}
  \eta(0,x) = \eta_{0}(x) 
    = \begin{cases}
	A\cosh(x) + 1 & (0 \leq x < 0.25),\\
	B\cosh(x - 1) + 1 & (0.25 \leq x \leq 1),
      \end{cases}\\
      \theta(0,x) = \theta_{0}(x) = 0.5\pi x - 0.25\pi, 
\end{gather*}
where 
\begin{gather*}
	A = -\frac{1}{\cosh(0.25)}\frac{|\theta_{2} - \theta_{1}|}{|\theta_{2} - \theta_{1}| + \tanh(0.25) + \tanh(0.75)}, \\
	B = -\frac{1}{\cosh(0.75)}\frac{|\theta_{2} - \theta_{1}|}{|\theta_{2} - \theta_{1}| + \tanh(0.25) + \tanh(0.75)}, \\
	\theta_{1} = -0.25\pi, \quad \theta_{2} = 0.25\pi.
\end{gather*}
Also, we choose $\Delta t = 0.075$.

\begin{figure}[H]
	\begin{minipage}{0.49\hsize}
		\centering
		\includegraphics[width=50mm]{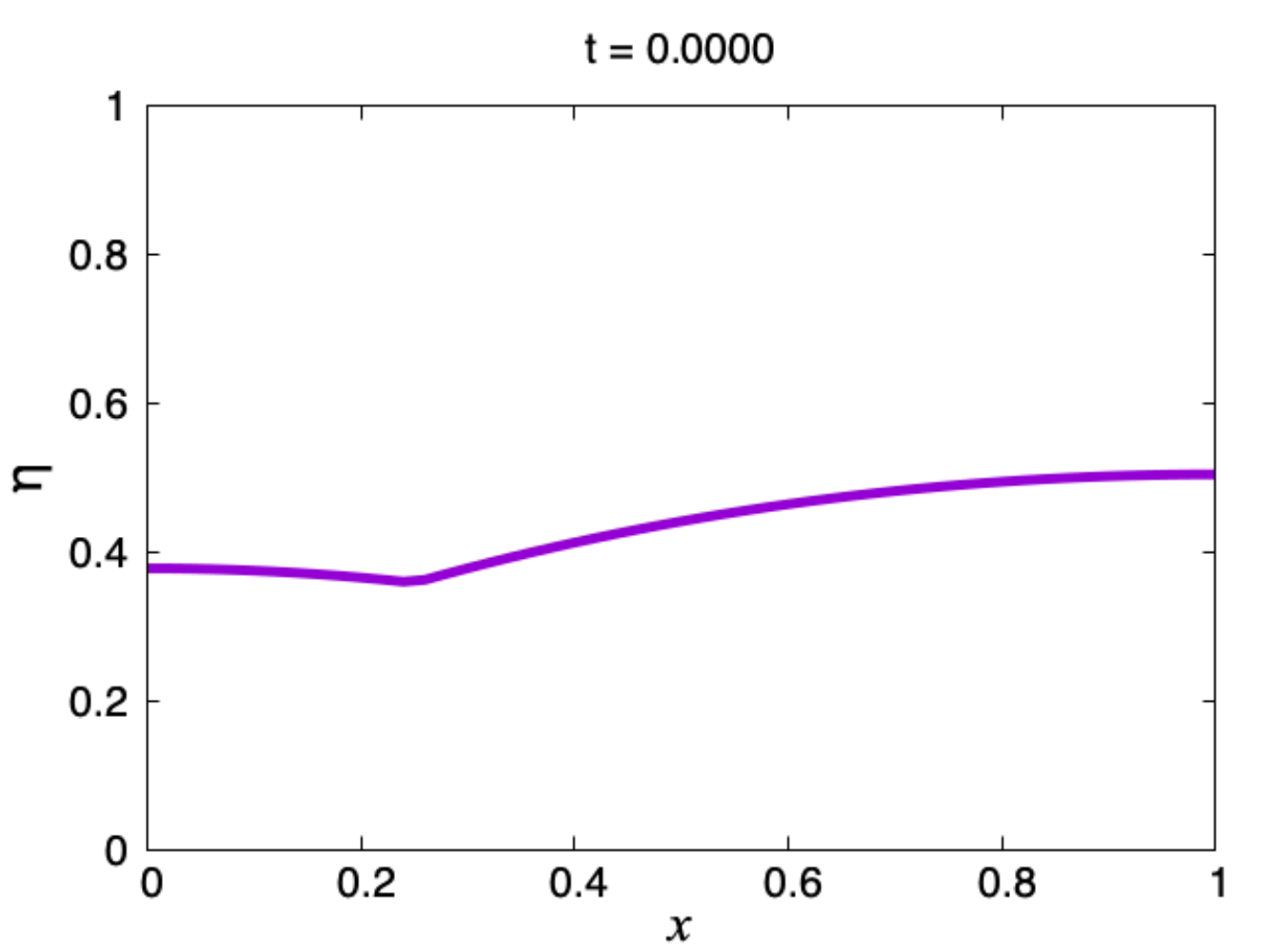}
	\end{minipage}
	\begin{minipage}{0.49\hsize}
		\centering
		\includegraphics[width=50mm]{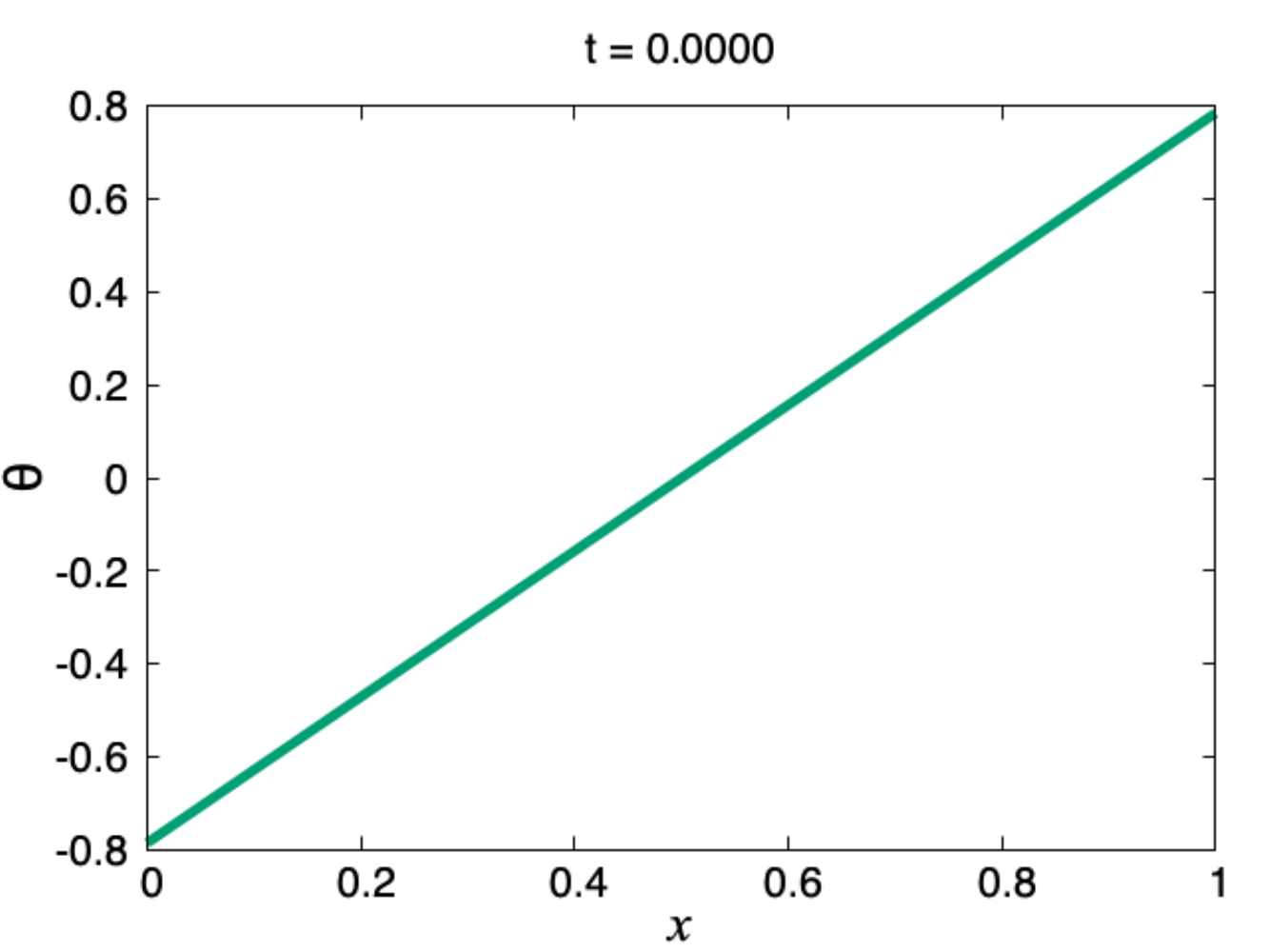}
	\end{minipage}
	\caption{\protect\raggedright The initial data}
	% \label{fig:}
\end{figure} 
Figures \ref{fig:eta_init2}--\ref{fig:theta_init2} show the time development of the solutions to our scheme, respectively. 
Figure \ref{fig:energy_init2} shows the time development of $\mathscr{F}_{1, \rm d}$. 
This graph shows that the energy decreases numerically.
\begin{figure}[H]
	\setlength\abovecaptionskip{0pt}
	\begin{minipage}{0.24\hsize}
		\centering
		\includegraphics[width=38mm]{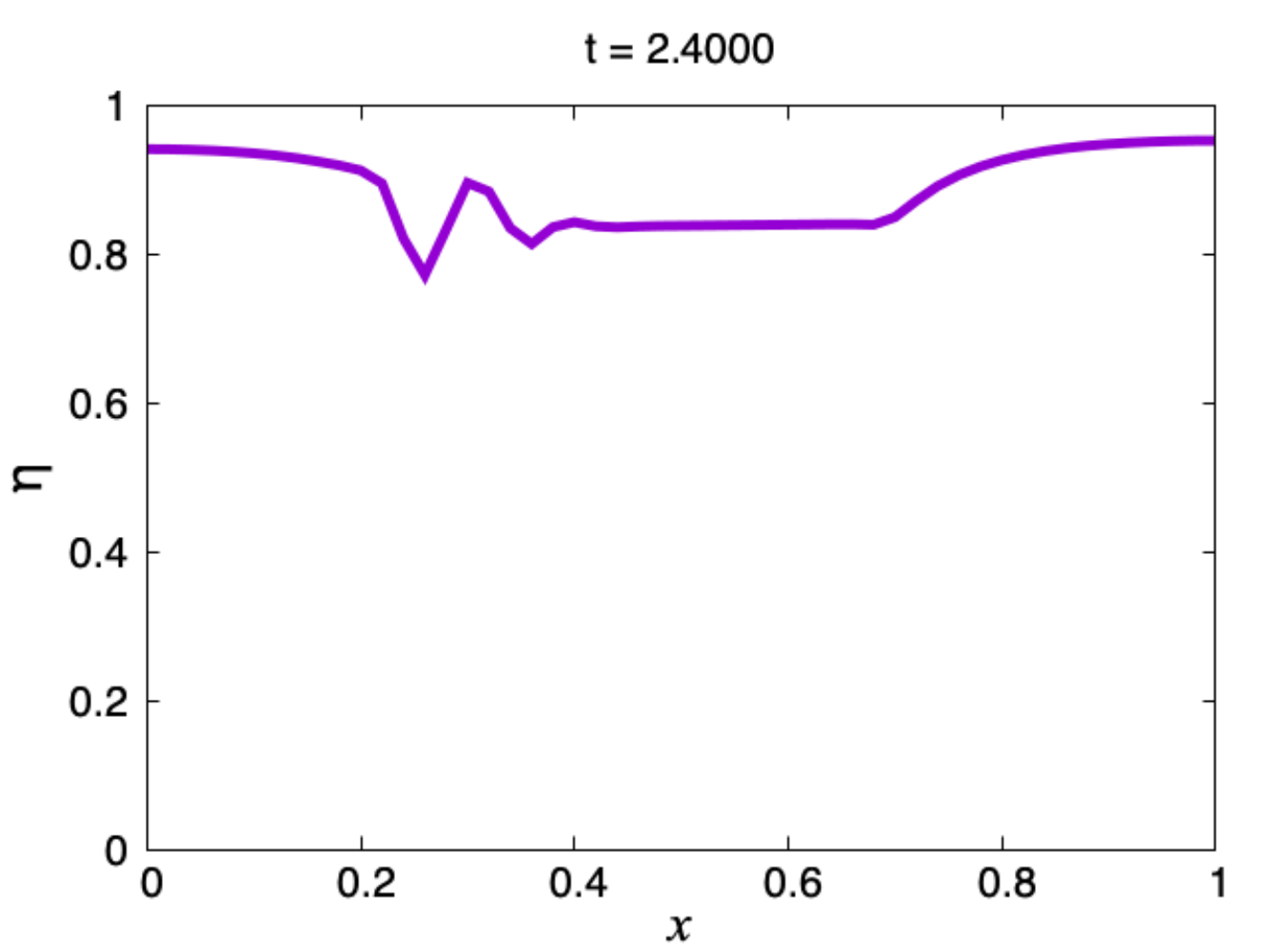}
	\end{minipage}
	\begin{minipage}{0.24\hsize}
		\centering
		\includegraphics[width=38mm]{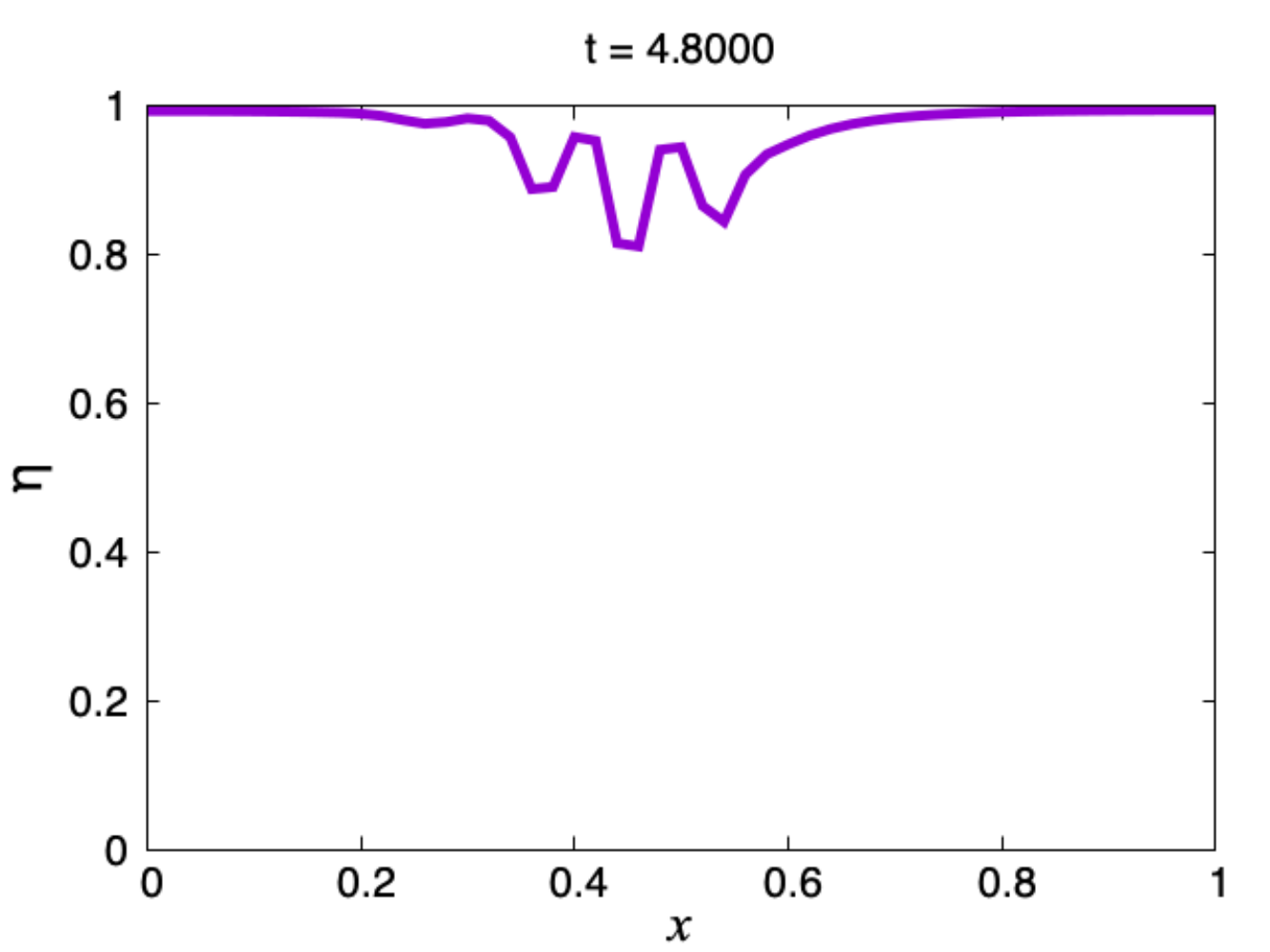}
	\end{minipage}
	\begin{minipage}{0.24\hsize}
		\centering
		\includegraphics[width=38mm]{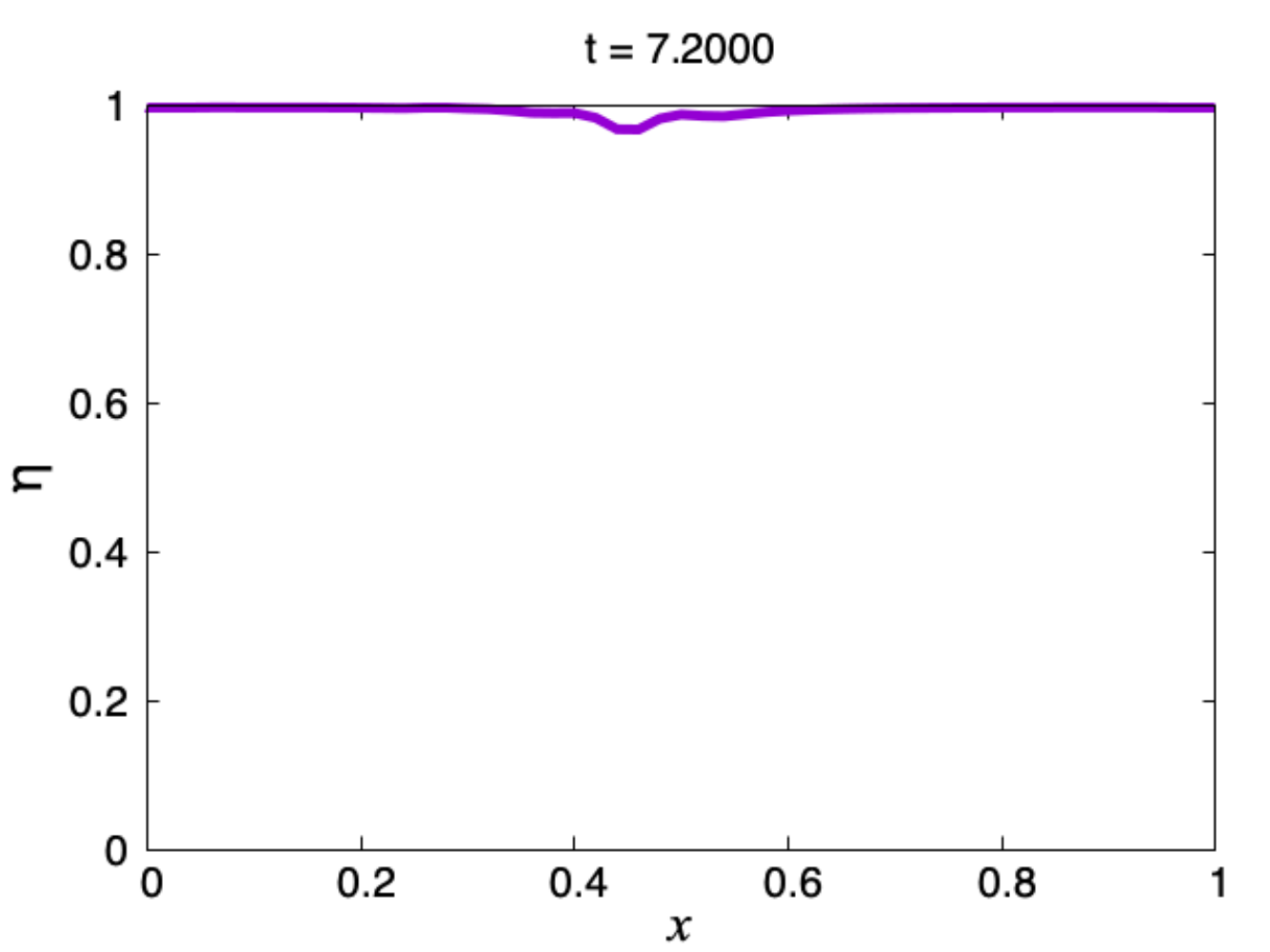}
	\end{minipage}
	\begin{minipage}{0.24\hsize}
		\centering
		\includegraphics[width=38mm]{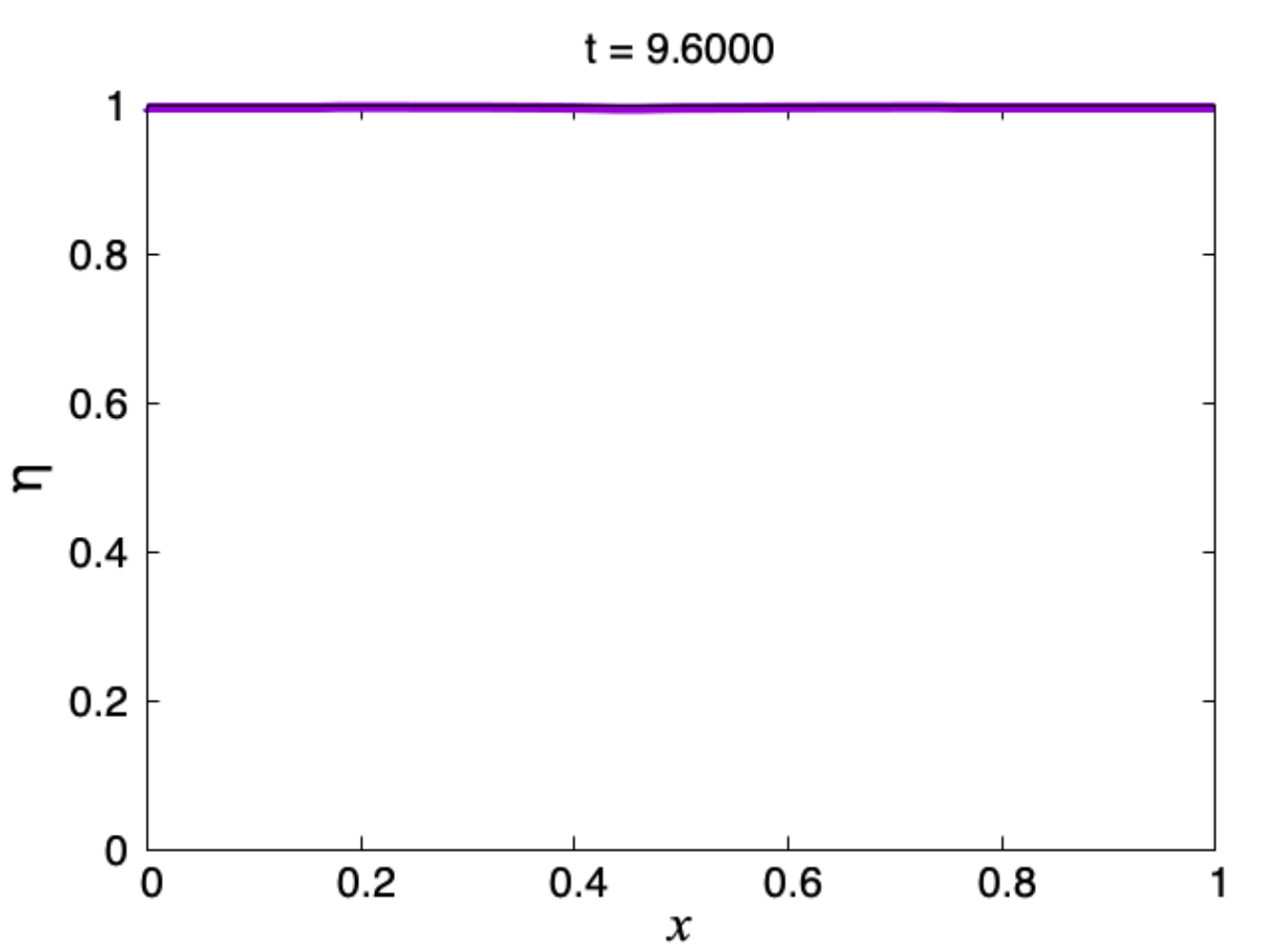}
	\end{minipage}
	\caption{\protect\raggedright Numerical solutions $\bm{H}^{(j)}$ to our scheme at time $t = 2.4$, $t = 4.8$, $t = 7.2$, $t = 9.6$}
	\label{fig:eta_init2}
\end{figure} 

\begin{figure}[H]
	\setlength\abovecaptionskip{0pt}
	\begin{minipage}{0.24\hsize}
		\centering
		\includegraphics[width=38mm]{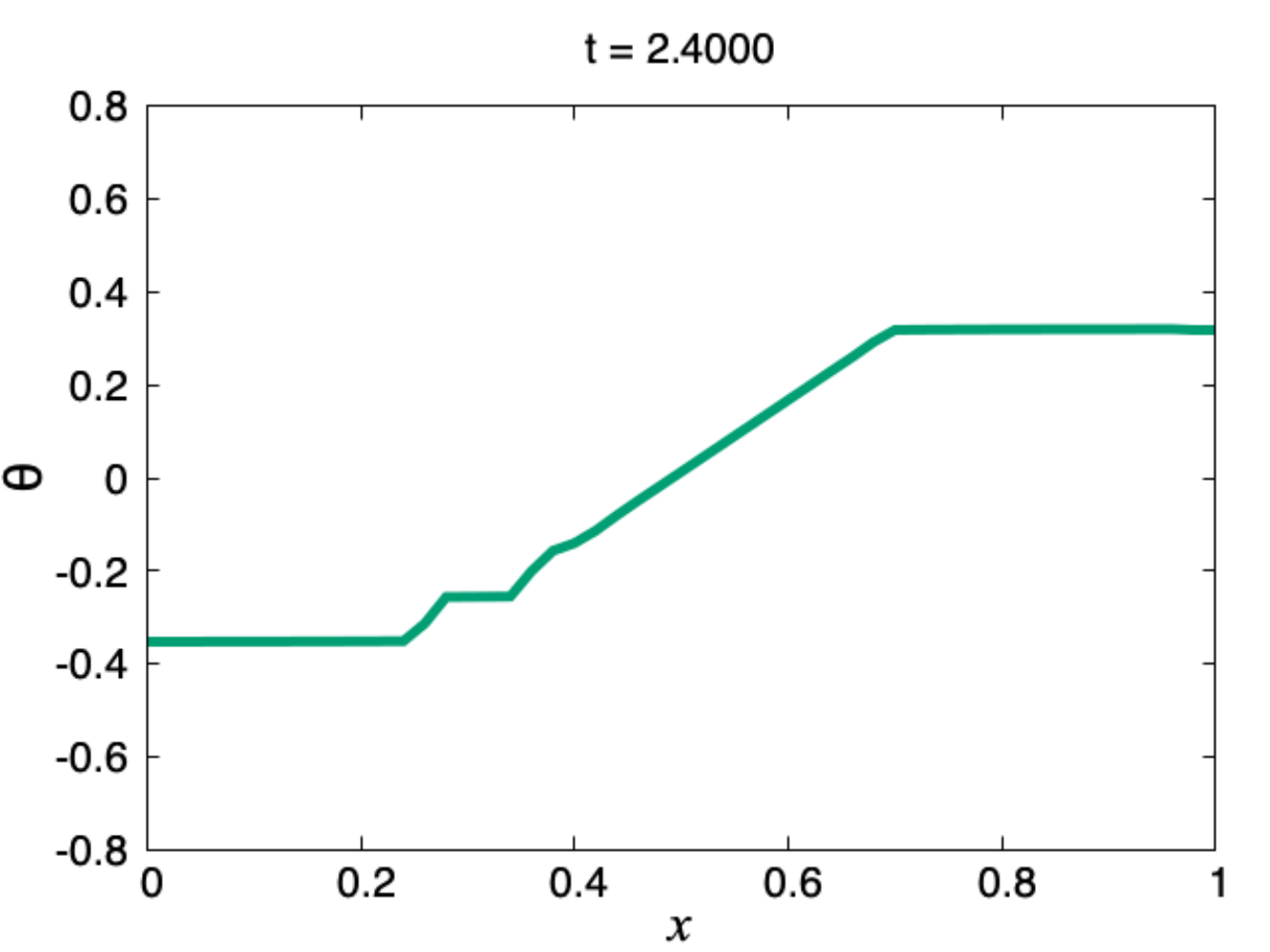}
	\end{minipage}
	\begin{minipage}{0.24\hsize}
		\centering
		\includegraphics[width=38mm]{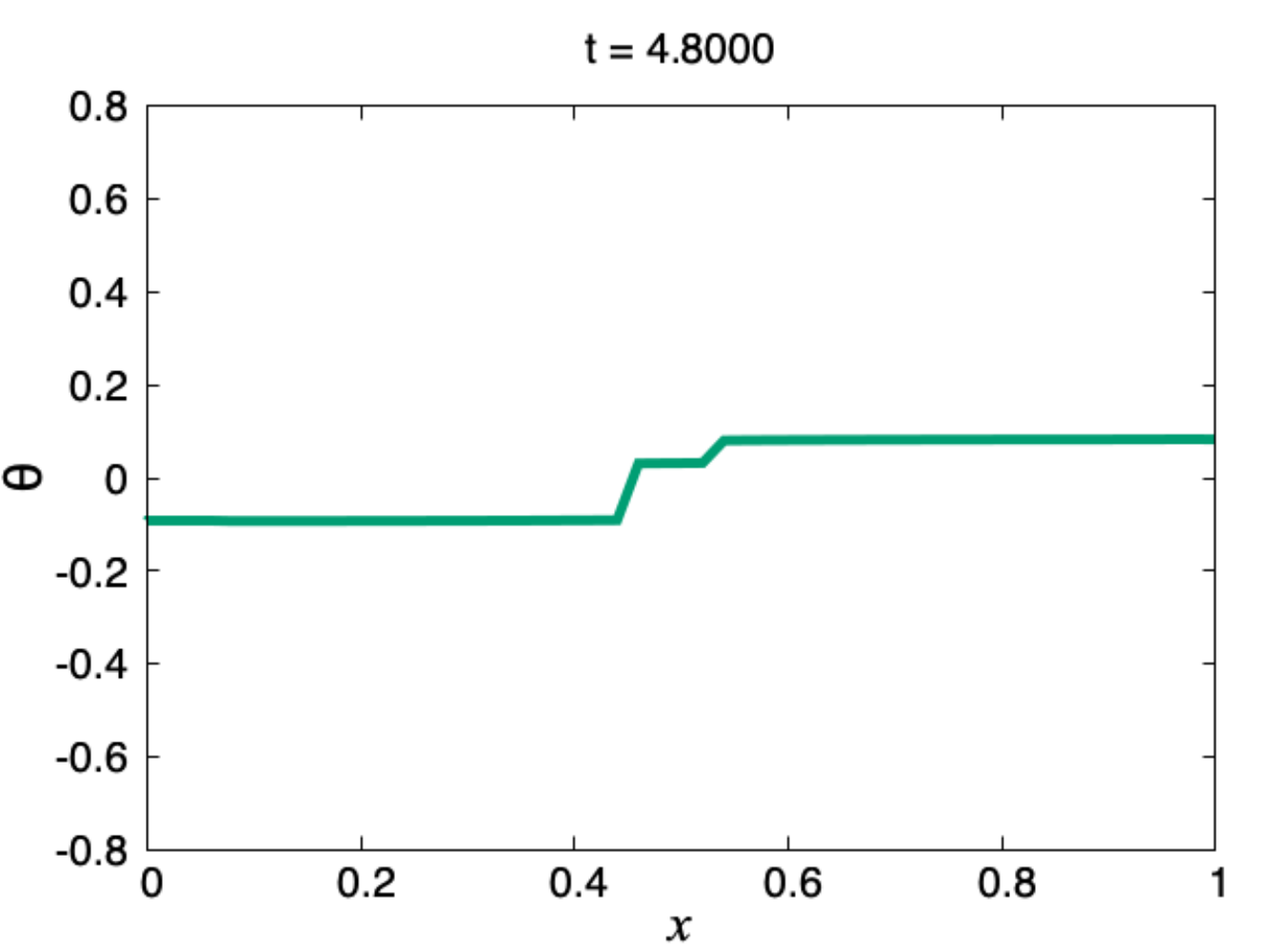}
	\end{minipage}
	\begin{minipage}{0.24\hsize}
		\centering
		\includegraphics[width=38mm]{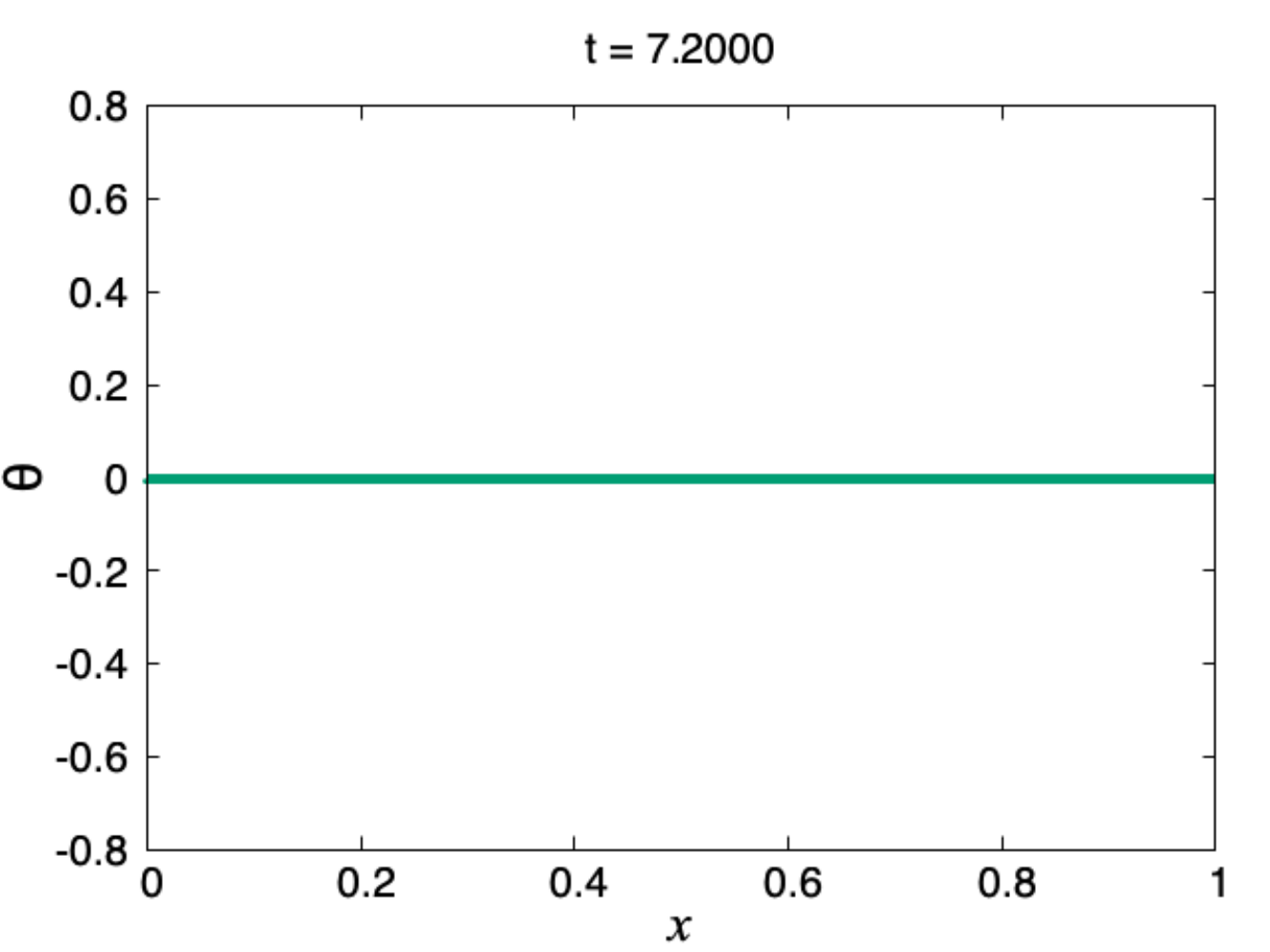}
	\end{minipage}
	\begin{minipage}{0.24\hsize}
		\centering
		\includegraphics[width=38mm]{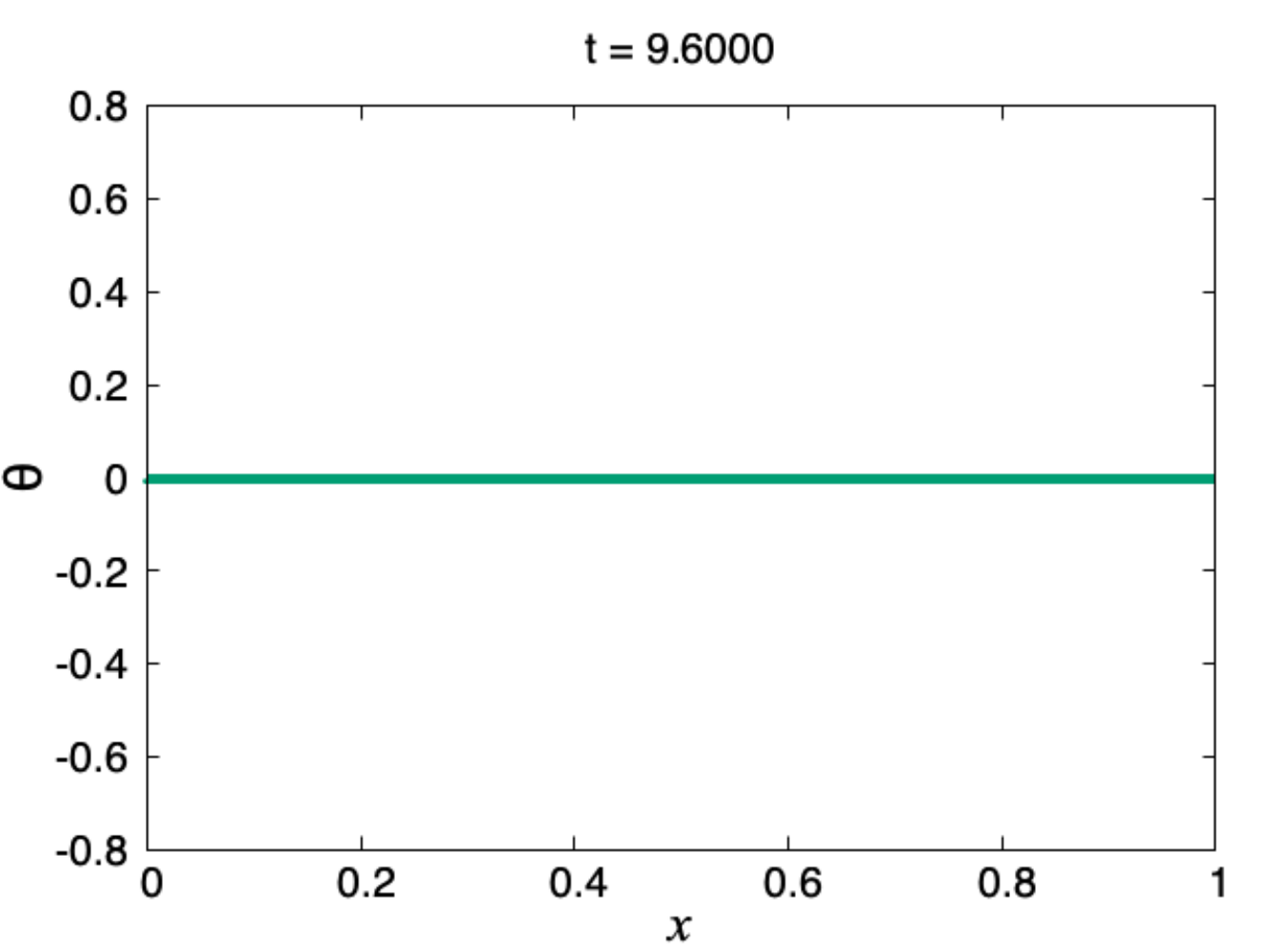}
	\end{minipage}
	\caption{\protect\raggedright Numerical solutions $\bm{\Theta}^{(j)}$ to our scheme at time $t = 2.4$, $t = 4.8$, $t = 7.2$, $t = 9.6$}
	\label{fig:theta_init2}
\end{figure} 

\begin{figure}[H]
	\centering
	\includegraphics[width=60mm]{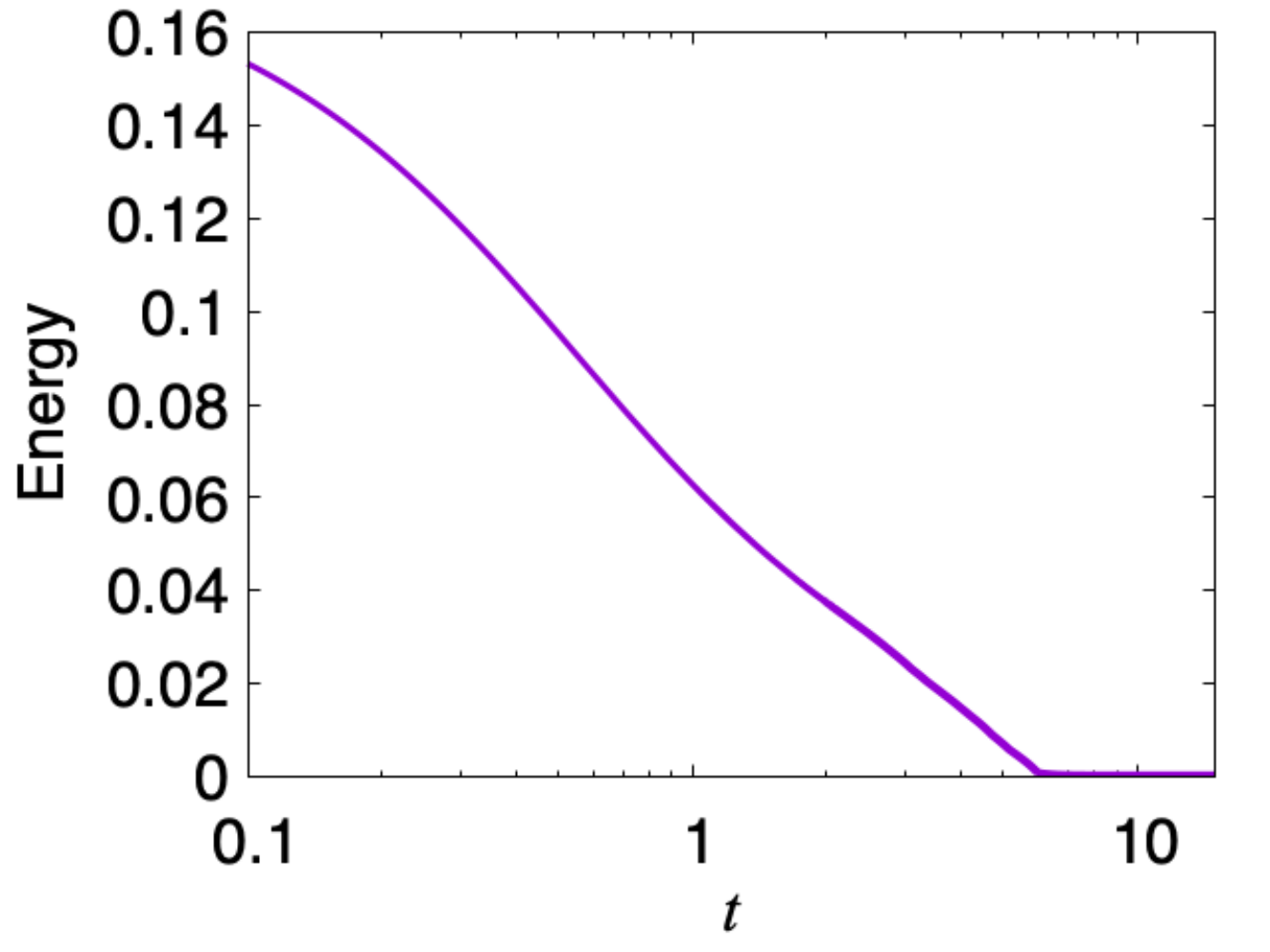}
	\caption{\protect\raggedright Time development of energy (The time axis is in log-scale)}
	\label{fig:energy_init2}
\end{figure}

\subsection{Computation example 3}
As the initial condition, we consider 
\begin{gather*}
  \eta(0,x) = \eta_{0}(x) 
    = \begin{cases}
	A\cosh(x) + 1 & (0 \leq x < 0.25),\\
	B\cosh(x - 1) + 1 & (0.25 \leq x \leq 1),
      \end{cases}\\
      \theta(0,x) = 
      \begin{cases}
	-0.25\pi x - 0.1\pi & (0 \leq x < 0.6),\\
	0.125\pi x + 0.125\pi & (0.6 \leq x \leq 1),
      \end{cases} 
\end{gather*}
where 
\begin{gather*}
	A = -\frac{1}{\cosh(0.25)}\frac{|\theta_{2} - \theta_{1}|}{|\theta_{2} - \theta_{1}| + \tanh(0.25) + \tanh(0.75)}, \\
	B = -\frac{1}{\cosh(0.75)}\frac{|\theta_{2} - \theta_{1}|}{|\theta_{2} - \theta_{1}| + \tanh(0.25) + \tanh(0.75)}, \\
	\theta_{1} = -0.25\pi, \quad \theta_{2} = 0.25\pi.
\end{gather*}
Also, we choose $\Delta t = 0.1414$.

\begin{figure}[H]
	\begin{minipage}{0.49\hsize}
		\centering
		\includegraphics[width=50mm]{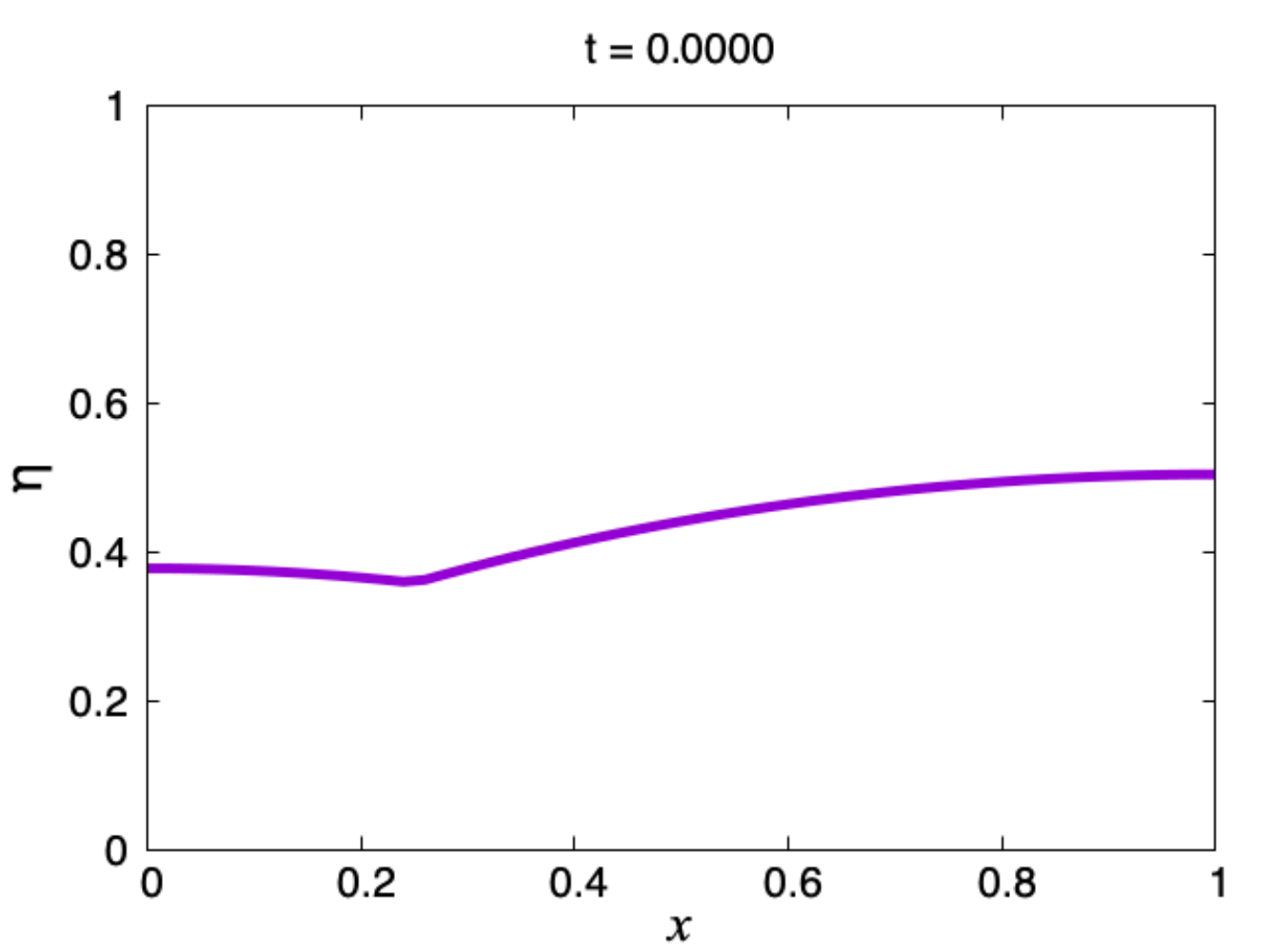}
	\end{minipage}
	\begin{minipage}{0.49\hsize}
		\centering
		\includegraphics[width=50mm]{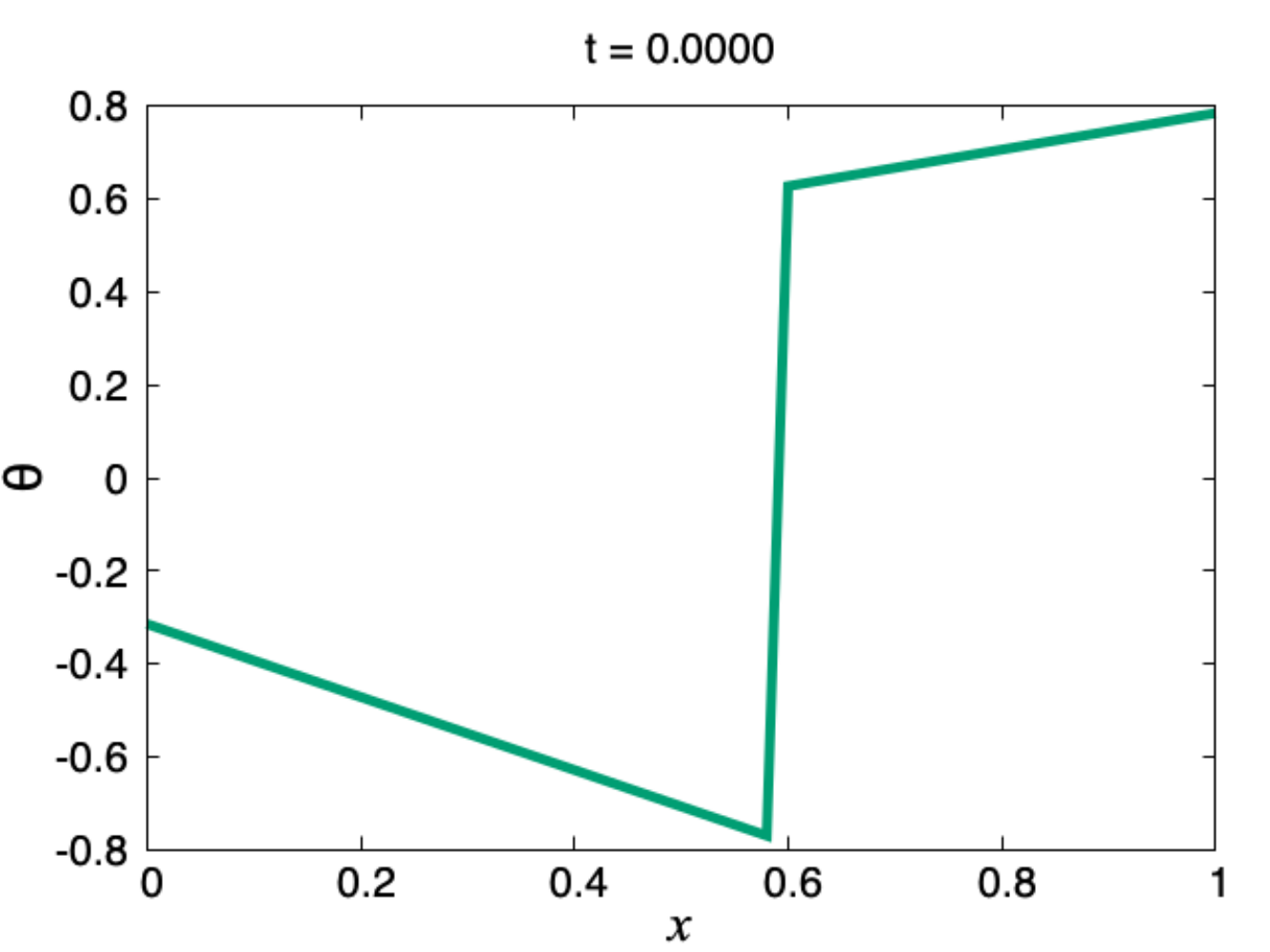}
	\end{minipage}
	\caption{\protect\raggedright The initial data}
	% \label{fig:}
\end{figure} 
Figures \ref{fig:eta_init3}--\ref{fig:theta_init3} show the time development of the solutions to our scheme, respectively. 
Figure \ref{fig:energy_init3} shows the time development of $\mathscr{F}_{1, \rm d}$. 
This graph shows that the energy decreases numerically. 
\begin{figure}[H]
	\setlength\abovecaptionskip{0pt}
	\begin{minipage}{0.24\hsize}
		\centering
		\includegraphics[width=38mm]{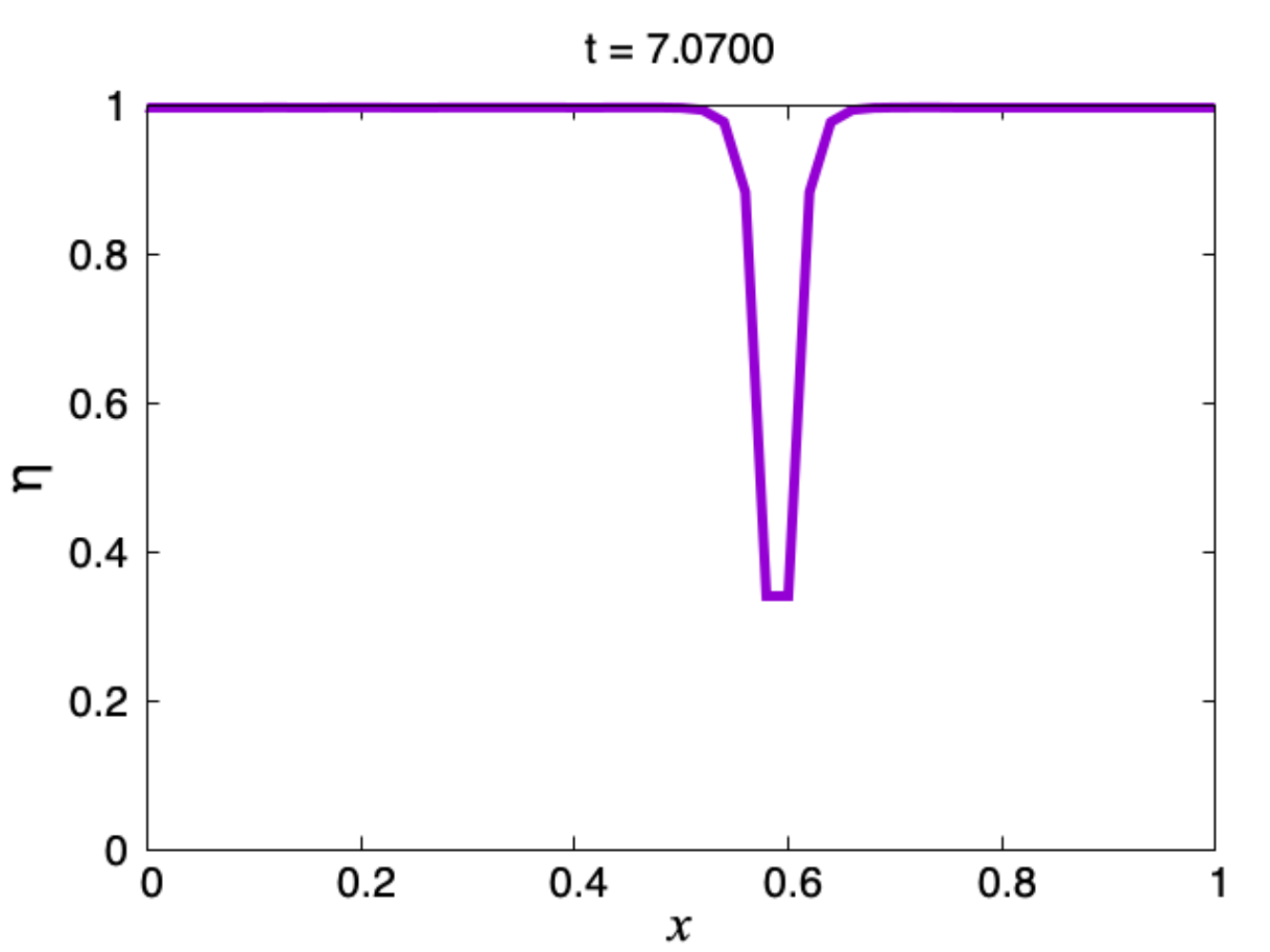}
	\end{minipage}
	\begin{minipage}{0.24\hsize}
		\centering
		\includegraphics[width=38mm]{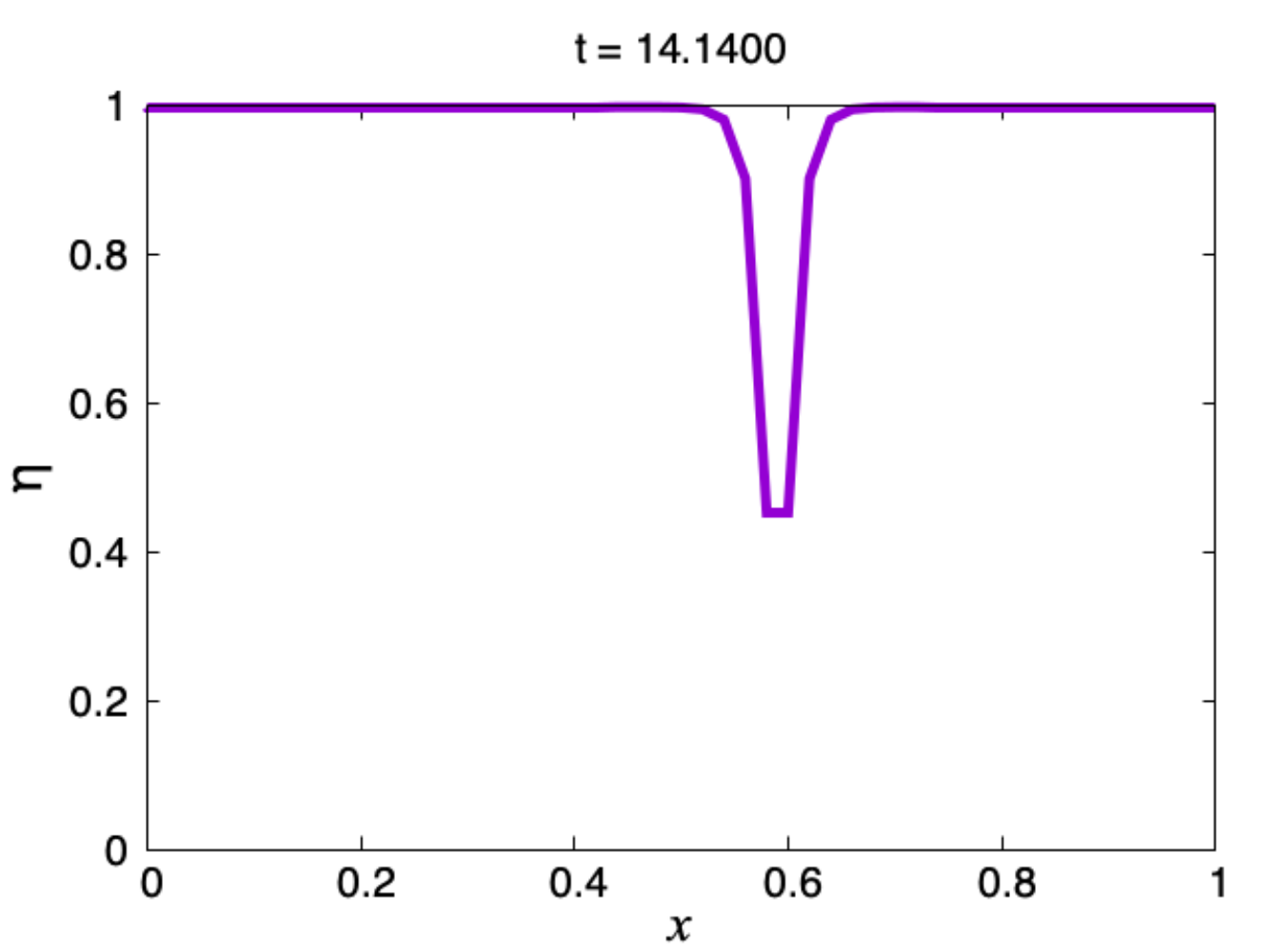}
	\end{minipage}
	\begin{minipage}{0.24\hsize}
		\centering
		\includegraphics[width=38mm]{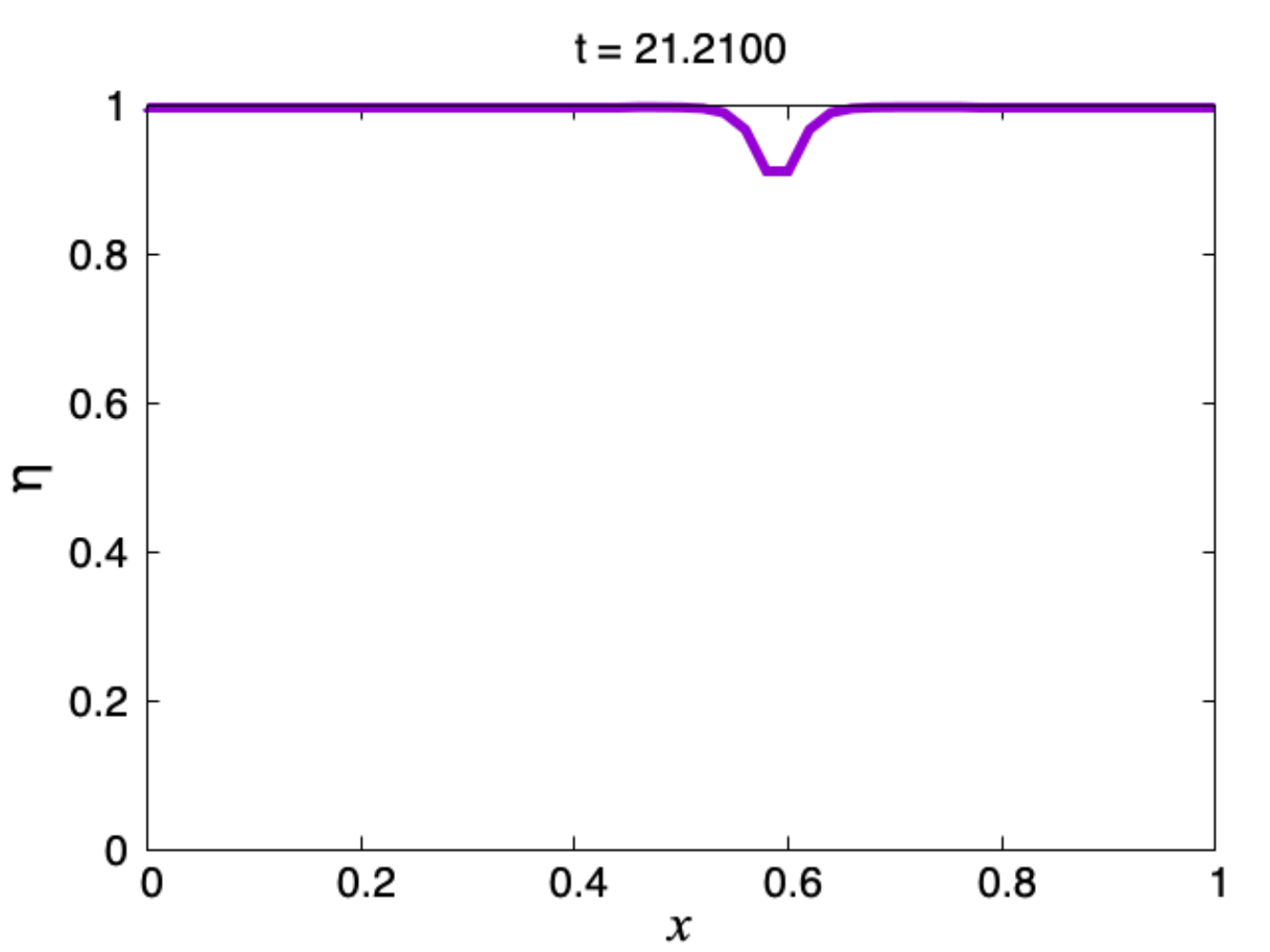}
	\end{minipage}
	\begin{minipage}{0.24\hsize}
		\centering
		\includegraphics[width=38mm]{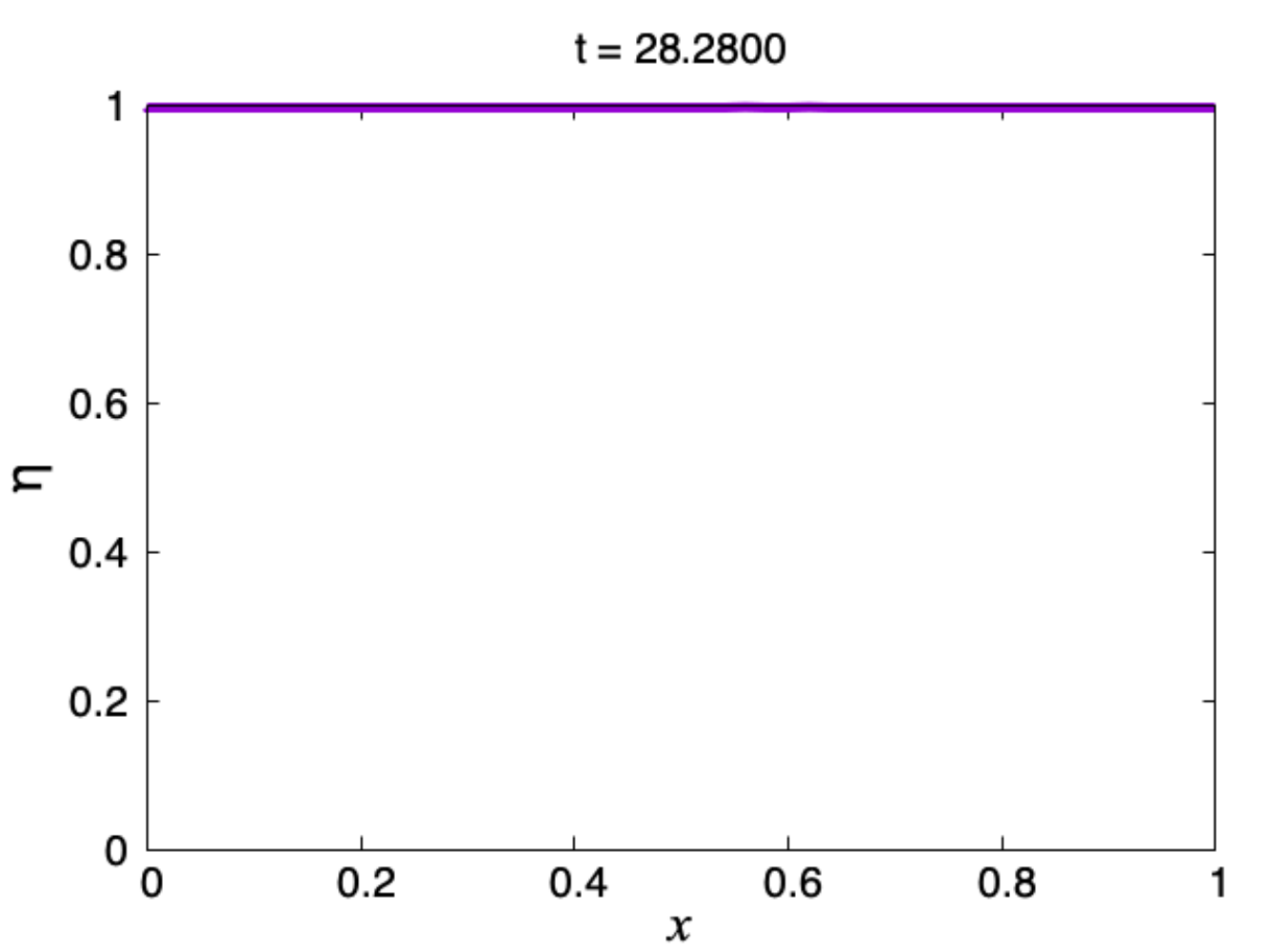}
	\end{minipage}
	\caption{\protect\raggedright Numerical solutions $\bm{H}^{(j)}$ to our scheme at time $t = 7.07$, $t = 14.14$, $t = 21.21$, $t = 28.28$}
	\label{fig:eta_init3}
\end{figure} 

\begin{figure}[H]
	\setlength\abovecaptionskip{0pt}
	\begin{minipage}{0.24\hsize}
		\centering
		\includegraphics[width=38mm]{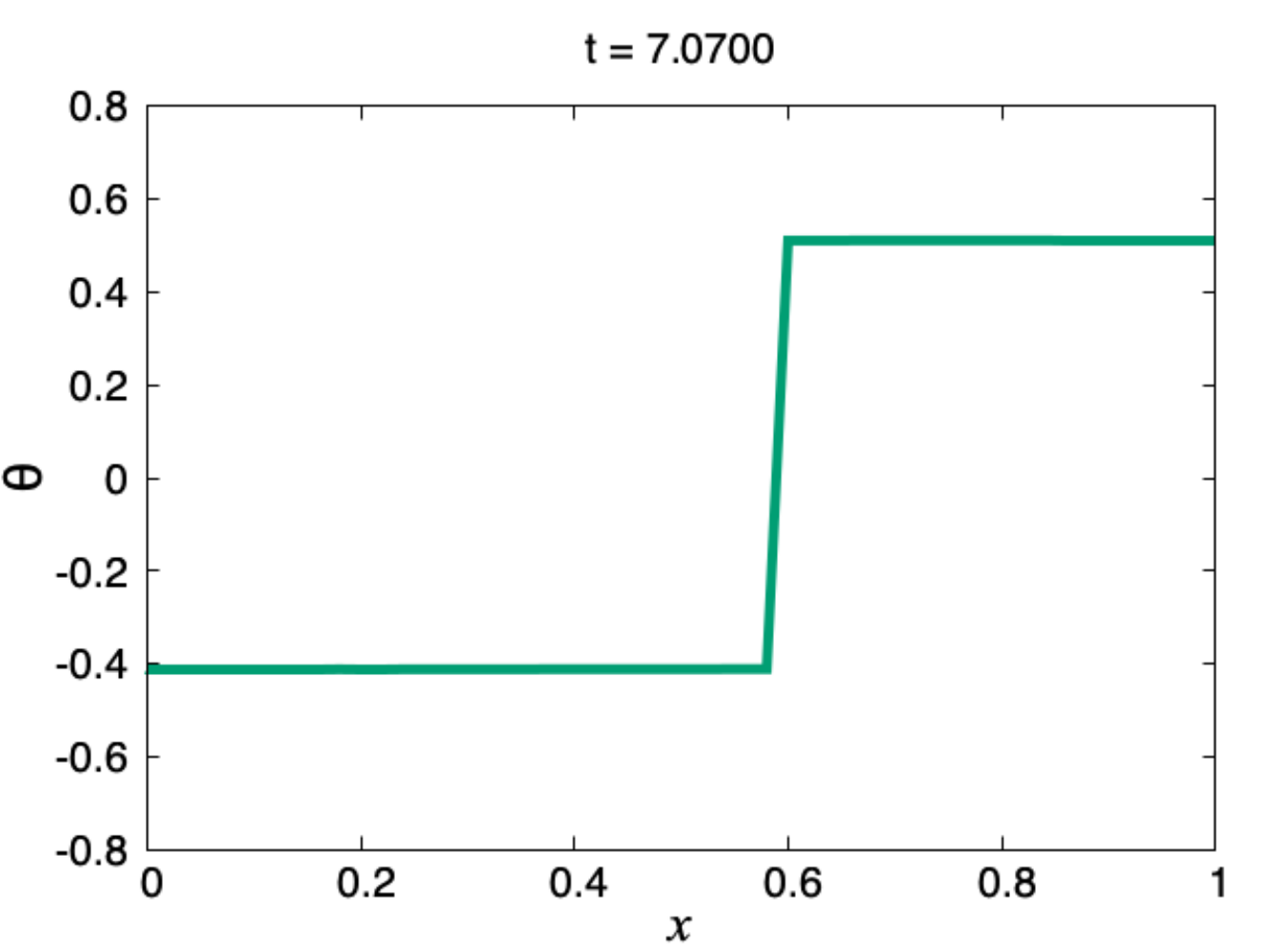}
	\end{minipage}
	\begin{minipage}{0.24\hsize}
		\centering
		\includegraphics[width=38mm]{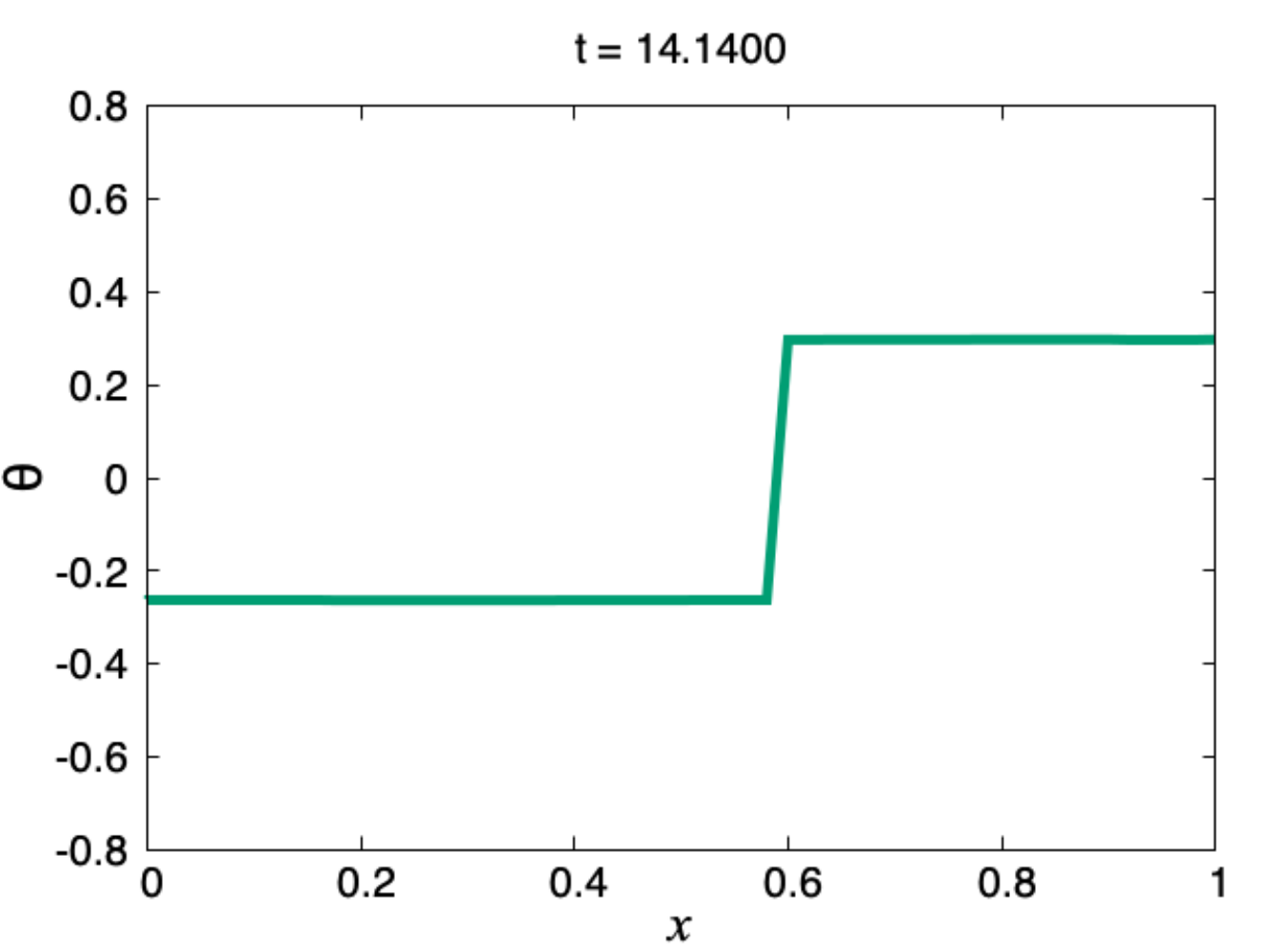}
	\end{minipage}
	\begin{minipage}{0.24\hsize}
		\centering
		\includegraphics[width=38mm]{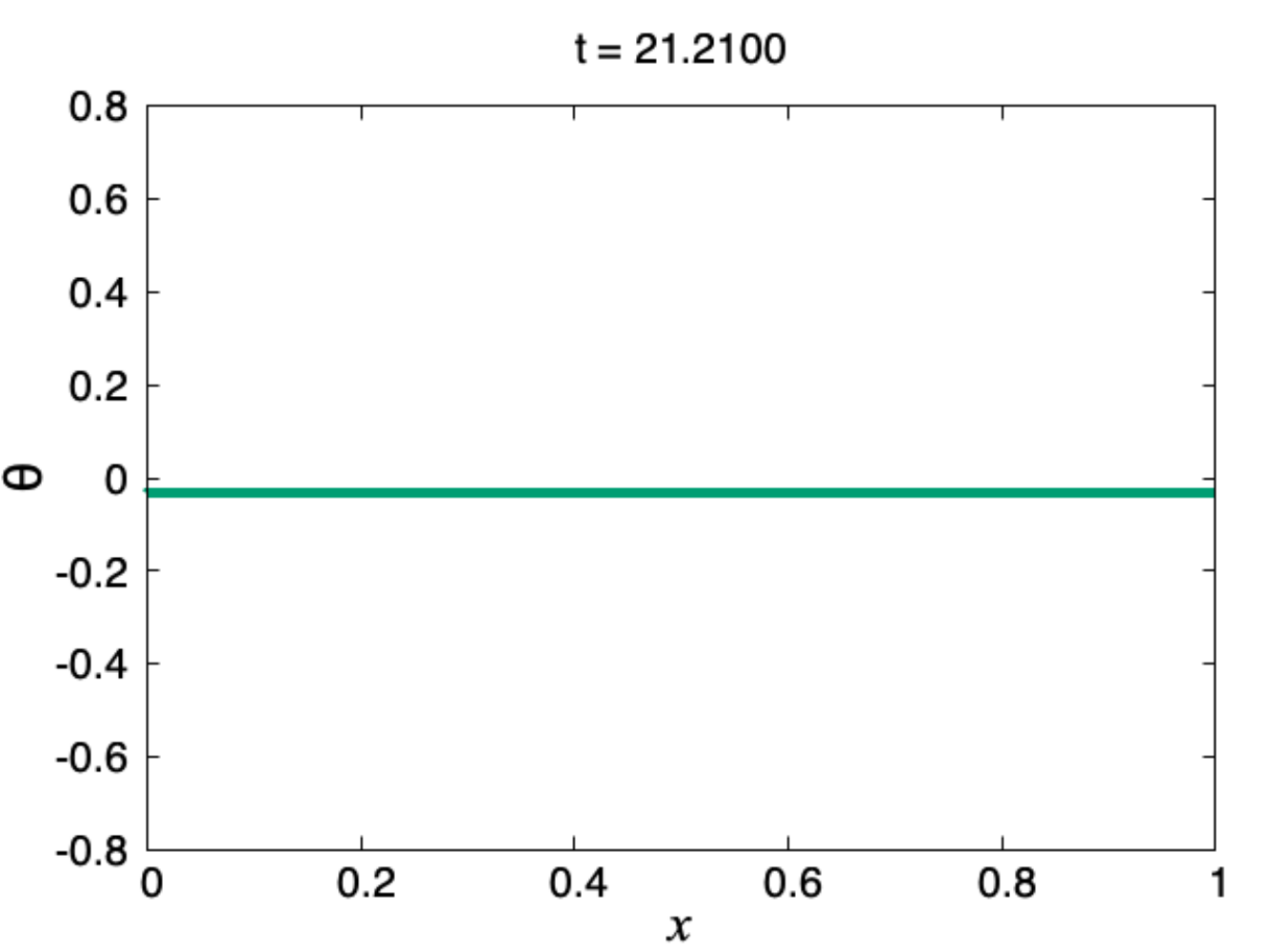}
	\end{minipage}
	\begin{minipage}{0.24\hsize}
		\centering
		\includegraphics[width=38mm]{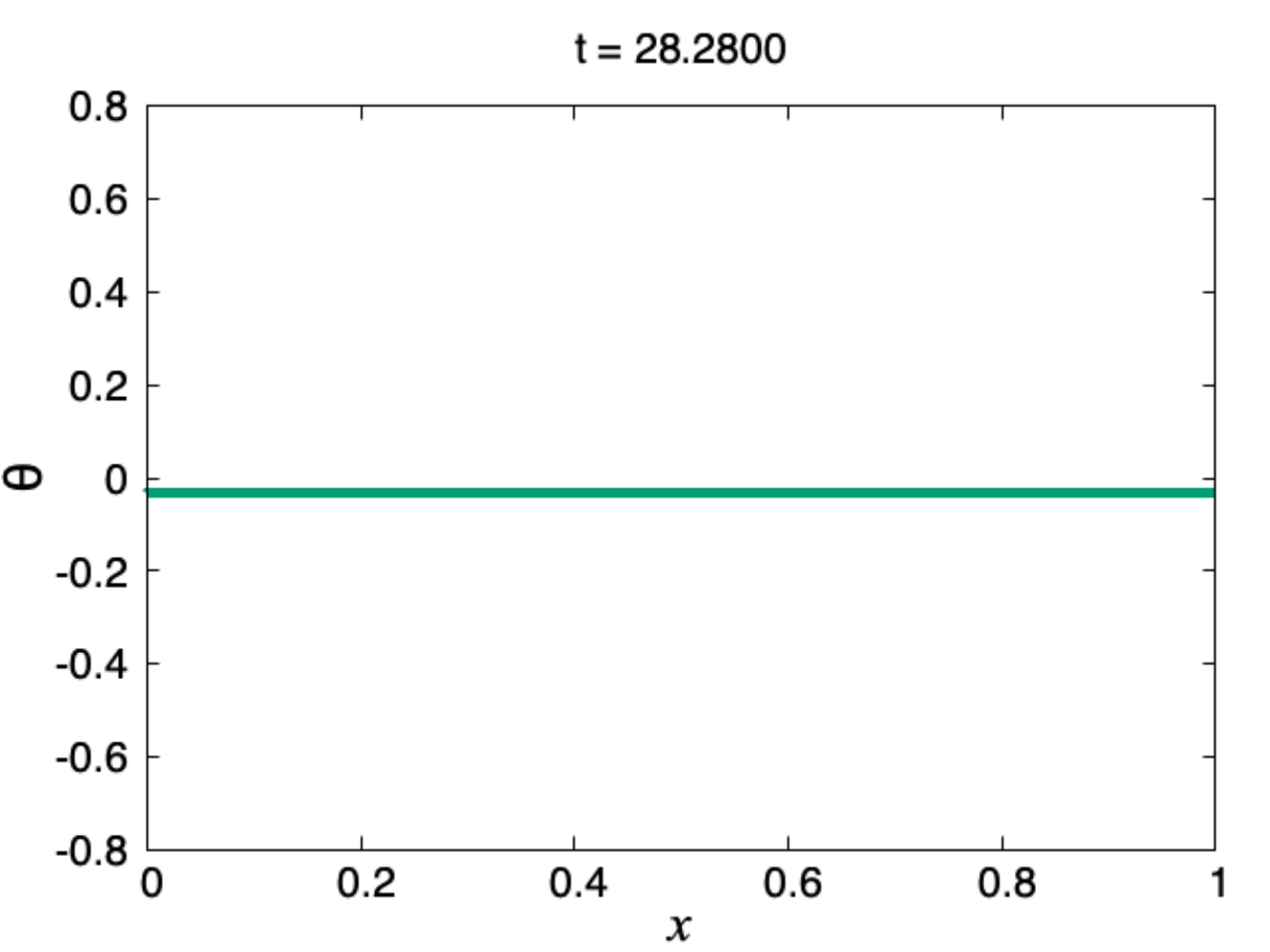}
	\end{minipage}
	\caption{\protect\raggedright Numerical solutions $\bm{\Theta}^{(j)}$ to our scheme at time $t = 7.07$, $t = 14.14$, $t = 21.21$, $t = 28.28$}
	\label{fig:theta_init3}
\end{figure} 

\begin{figure}[H]
	\centering
	\includegraphics[width=60mm]{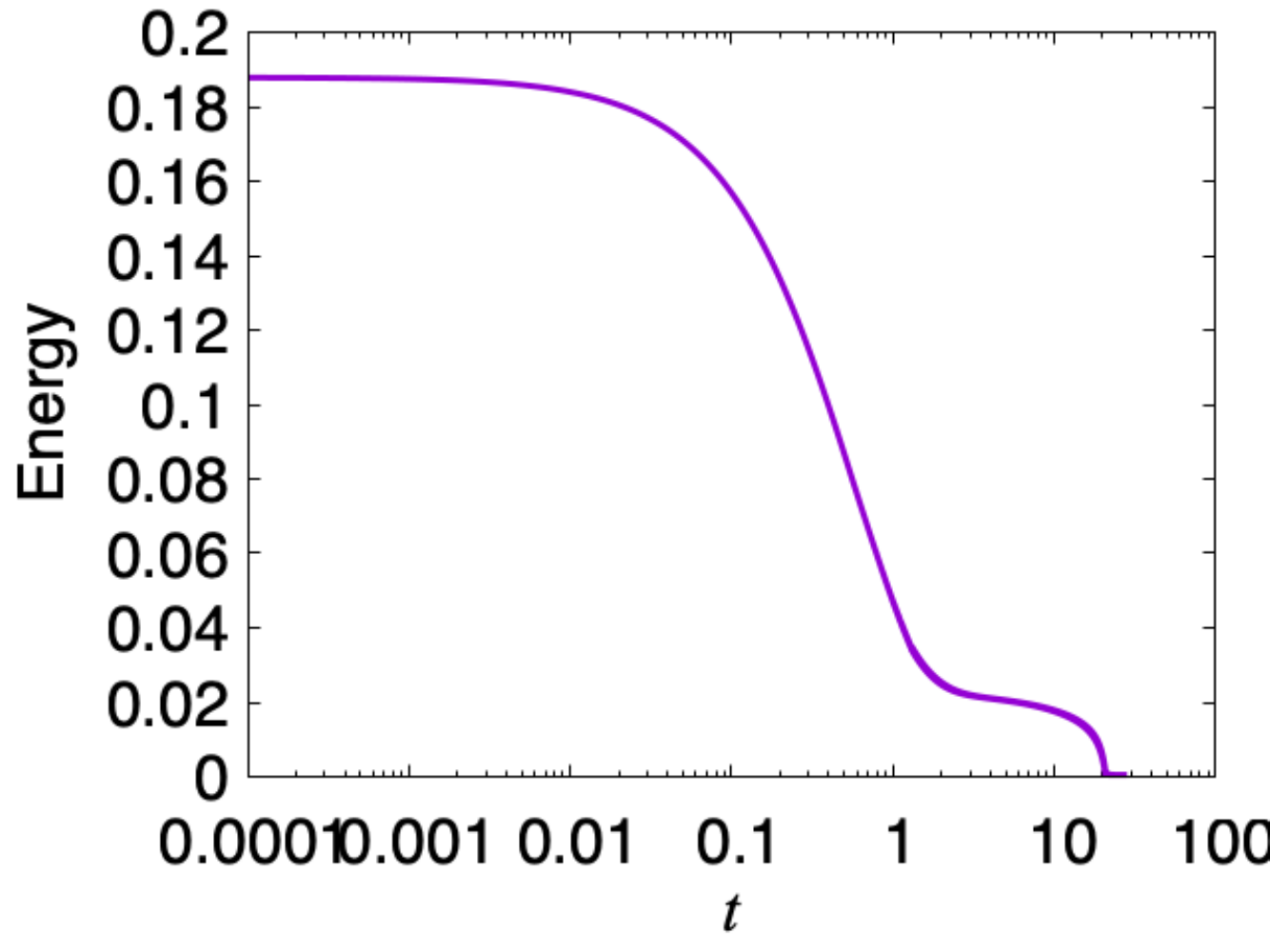}
	\caption{\protect\raggedright Time development of energy (The time axis is in log-scale)}
	\label{fig:energy_init3}
\end{figure}
Through the above computation examples, the following conjectures can be obtained. 
\begin{description}
	\item[(O)]The following convergences hold:
	\begin{description}
	\item[(i)]$ \eta(t) \to 1 $ in $ L^2(\Omega) $, as $ t \to \infty $,
	\item[(ii)]there exists $ \theta_\infty \in \theta_0(\Omega) $ such that  $ \theta(t) $ converges the constant $ \theta_\infty $, at a finite time.
	\end{description}
	\item[(I)]If $ \eta(t_0) $, at a time $ t_0 \geq 0 $, has two neighboring local minima $ 0 < x_0 < x_1 < 1 $, then $ \theta(t) $, for $ t \geq t_0 $, tends to form a piecewise constant structure around the interval $ (x_0, x_1) $.
	\item[(II)]If $ \eta(t_*) $, at a large time $ t_* \gg 0 $, has a unique local minimum (global minimum) $ x_* \in \Omega $, then $ \theta(t) $, for $t \geq t_* $, tends to a piecewise constant structure, with a single discontinuity point $ x_* $.
\end{description}
Regarding (II), all numerical data generally show a tendency consistent with the conjecture. 
However, in example 3, it should be noted that $\eta$ takes a unique minimum on a (single) closed interval. 
In view of this, it is considered necessary to make the following revision to the conjecture (II). 
\begin{description}
	\item[(II)']If $ \eta(t_*) $, at a large time $ t_* \gg 0 $, has a unique local minimum (global minimum) on a closed interval $ I_* \in \Omega $, then $ \theta(t) $, for $t \geq t_* $, tends to a piecewise constant structure, with a discontinuity point in $ I_* $.
\end{description}

\section{Appendix} 
\subsection{Preparation}

\begin{lem}\label{prod}
	The following product rules hold: 
	\begin{gather*}
		\delta_{k}^{+}(f_{k}g_{k}) = (\delta_{k}^{+}f_{k})(\mu_{k}^{+}g_{k}) + (\mu_{k}^{+}f_{k})(\delta_{k}^{+}g_{k}) \quad (k = 0,\ldots, K), \\
		\delta_{k}^{-}(f_{k}g_{k}) = (\delta_{k}^{-}f_{k})(\mu_{k}^{-}g_{k}) + (\mu_{k}^{-}f_{k})(\delta_{k}^{-}g_{k}) \quad (k = 0,\ldots, K)
	\end{gather*}
	for all $\{f_{k}\}_{k=-1}^{K+1}, \{g_{k}\}_{k=-1}^{K+1} \in \mathbb{R}^{K+3}$. 
\end{lem}

\begin{lem}\label{sbp1}
The following summation by parts formulas hold:
	\begin{equation*}
	\sum_{k=0}^{K}{}^{\prime\prime}f_{k}(\delta_{k}^{+}g_{k})\Delta x + \sum_{k=0}^{K}{}^{\prime\prime}(\delta_{k}^{-}f_{k})g_{k}\Delta x 
	= \left[\frac{ f_{k}g_{k+1} + f_{k-1}g_{k} }{2}\right]_{0}^{K} 
	\end{equation*}
	for all $\{f_{k}\}_{k=-1}^{K+1}, \{g_{k}\}_{k=-1}^{K+1} \in \mathbb{R}^{K+3}$. 
\end{lem}
	
\begin{lem}\label{sbp2}
	The following summation by parts formulas hold:
	\begin{gather*}
	\sum_{k=0}^{K-1}f_{k}(\delta_{k}^{+}g_{k})\Delta x + \sum_{k=0}^{K}{}^{\prime\prime}(\delta_{k}^{-}f_{k})g_{k}\Delta x 
	= \left[\left(\mu_{k}^{-}f_{k}\right)g_{k} \right]_{0}^{K}, \\
	\sum_{k=1}^{K}f_{k}(\delta_{k}^{-}g_{k})\Delta x + \sum_{k=0}^{K}{}^{\prime\prime}(\delta_{k}^{+}f_{k})g_{k}\Delta x 
	= \left[\left(\mu_{k}^{+}f_{k}\right)g_{k} \right]_{0}^{K} 
	\end{gather*}
	for all $\{f_{k}\}_{k=-1}^{K+1}, \{g_{k}\}_{k=-1}^{K+1} \in \mathbb{R}^{K+3}$. 
\end{lem}

\begin{lem}\label{sbp5}
	The following summation by parts formulas hold:
	\begin{equation*}
	\sum_{k=0}^{K-1}\left(\delta_{k}^{+}f_{k}\right)\left(\delta_{k}^{+}g_{k}\right)\Delta x = -\sum_{k=0}^{K}{}^{\prime\prime}\left(\delta_{k}^{\langle 2 \rangle}f_{k}\right)g_{k}\Delta x + \left[\left(\delta_{k}^{\langle 1 \rangle}f_{k}\right) g_{k}\right]_{0}^{K} 
	\end{equation*}
	for all $\{f_{k}\}_{k=-1}^{K+1}, \{g_{k}\}_{k=-1}^{K+1} \in \mathbb{R}^{K+3}$. 
\end{lem}

\begin{defn}
Let $\Omega$ be a domain in $\mathbb{R}$. 
For a function $F \in C^{2}(\Omega )$, let us define $\bar{F}''$: $\Omega^{4} \to \mathbb{R}$ by 
\begin{numcases}
	{\bar{F}'' (\xi, \tilde{\xi}; \eta, \tilde{\eta}) := \!} 
	\frac{1}{\xi - \tilde{\xi}}\!\left\{\! \left( \frac{dF}{d(\xi, \eta)} + \frac{dF}{d(\xi, \tilde{\eta})} \right) \! - \! \left( \frac{dF}{d(\tilde{\xi}, \eta)} + \frac{dF}{d(\tilde{\xi}, \tilde{\eta})} \right) \!\right\}, & ($\xi \neq \tilde{\xi}$), \nonumber\\
	{\partial}_{\xi} \left. \left( \frac{dF}{d(\xi, \eta)} + \frac{dF}{d(\xi, \tilde{\eta})} \right) \right|_{\xi = \tilde{\xi}} , & ($\xi = \tilde{\xi}$) \nonumber
\end{numcases}
for all $(\xi, \tilde{\xi}, \eta, \tilde{\eta}) \in \Omega^{4}$. 
\end{defn}

\begin{lem}[{\cite[Lemma 2.4]{YOSHIKAWA2017394}}] \label{lem:4.1}
Let $\Omega$ be a domain in $\mathbb{R}$. 
If $F \in C^{2}(\Omega)$, then $\bar{F}'' \in C(\Omega^{4})$. 
Moreover, we have 
\begin{equation*}
\left| \bar{F}'' (\xi, \tilde{\xi}; \eta, \tilde{\eta}) \right| \leq \sup_{\eta, \tilde{\eta} \in \Omega}\sup_{\xi \in \Omega}\left| \frac{\partial}{\partial \xi} \left( \frac{dF}{d(\xi, \eta)} + \frac{dF}{d(\xi, \tilde{\eta})} \right) \right| \leq \sup_{\xi \in \Omega}\left| F'' (\xi) \right| 
\quad \mbox{for all}\ \xi, \tilde{\xi}, \eta, \tilde{\eta} \in \Omega.
\end{equation*} 
\end{lem}

\begin{lem}[{\cite[Proposition 2.5]{YOSHIKAWA2017394}}] \label{lem:4.2}
Let $\Omega$ be a domain in $\mathbb{R}$ and assume that $F \in C^{2}(\Omega)$. For all $\xi$, $\tilde{\xi}$, $\eta$ and $\tilde{\eta} \in \Omega$, we have
\begin{equation*}
\frac{dF}{d(\xi, \eta)} - \frac{dF}{d(\tilde{\xi}, \tilde{\eta})} = \frac{1}{2}\bar{F}'' (\xi, \tilde{\xi}; \eta, \tilde{\eta})(\xi - \tilde{\xi}) + \frac{1}{2}\bar{F}'' (\eta, \tilde{\eta}; \xi, \tilde{\xi})(\eta - \tilde{\eta}). 
\end{equation*}
\end{lem}

\subsection{Auxiliary lemmas}

\begin{lem}\label{e_eq}
Let us assume (A1) and (A2). 
Then, we obtain the following equations on the errors: 
\begin{gather}
	\begin{split}
	\frac{e_{\eta,k}^{(j+1)} - e_{\eta,k}^{(j)}}{\Delta t} 
		= & \kappa_{0}^{2}\delta_{k}^{\langle 2 \rangle}e_{\eta,k}^{(j+1)} - c e_{\eta,k}^{(j+1)} \\
		& - \frac{\kappa H_{k}^{(j+1)}}{2}\left( \frac{ d\gamma_{\varepsilon}} { d(\delta_{k}^{+}\Theta_{k}^{(j)}, \delta_{k}^{+}\tilde{\theta}_{k}^{(j)}) }\delta_{k}^{+}e_{\theta,k}^{(j)} + \frac{ d\gamma_{\varepsilon} }{ d(\delta_{k}^{-}\Theta_{k}^{(j)}, \delta_{k}^{-}\tilde{\theta}_{k}^{(j)}) }\delta_{k}^{-}e_{\theta,k}^{(j)} \!\right) \\
		& - \! \kappa e_{\eta,k}^{(j+1)}\frac{ \gamma_{\varepsilon}\bigl( \delta_{k}^{+}\tilde{\theta}_{k}^{(j)}\bigr) \! + \! \gamma_{\varepsilon}\bigl( \delta_{k}^{-}\tilde{\theta}_{k}^{(j)}\bigr) }{2} \! + \! \xi_{1-3, k}^{(j+1)} \quad (k = 0,\ldots, K,\ j = 0,\ldots, N \! - \! 1), 
	\end{split} \label{e_eta_eq}
\end{gather}
\begin{gather}
	\begin{split}
	& \alpha_{0,k}^{(j+1)}\frac{e_{\theta,k}^{(j+1)} - e_{\theta,k}^{(j)} }{\Delta t} \\
		= & \frac{\kappa}{2}
		\left[\delta_{k}^{+}
			\left\{
				\frac{H_{k}^{(j+1)} + \eta_{k}^{(j+1)}}{2} \gamma_{\varepsilon}' (\delta_{k}^{-}\Theta_{k}^{(j+1)})e_{\eta,k}^{(j+1)} 
				+ \alpha\left( \eta_{k}^{(j+1)} \right) \frac{ d\gamma_{\varepsilon}' }{ d(\delta_{k}^{-}\Theta_{k}^{(j+1)}, \delta_{k}^{-}\tilde{\theta}_{k}^{(j+1)}) }\delta_{k}^{-}e_{\theta,k}^{(j+1)}
			\right\} \right.\\
		& \left. + \delta_{k}^{-}
			\left\{
				\frac{H_{k}^{(j+1)} + \eta_{k}^{(j+1)}}{2} \gamma_{\varepsilon}' (\delta_{k}^{+}\Theta_{k}^{(j+1)})e_{\eta,k}^{(j+1)} 
				+ \alpha\left( \eta_{k}^{(j+1)} \right) \frac{ d\gamma_{\varepsilon}' }{ d(\delta_{k}^{+}\Theta_{k}^{(j+1)}, \delta_{k}^{+}\tilde{\theta}_{k}^{(j+1)}) }\delta_{k}^{+}e_{\theta,k}^{(j+1)}
			\right\}
		\right] \\
		& + \nu^{2}\delta_{k}^{\langle 2 \rangle}e_{\theta,k}^{(j+1)} + \xi_{4-6, k}^{(j+1)} \quad (k = 0,\ldots, K,\ j = 0,\ldots, N \! - \! 1), 
	\end{split} \label{e_theta_eq}\\
	\delta_{k}^{\langle 1 \rangle}e_{\eta,k}^{(j)} = 0 \quad (k=0,K,\ j=0,\ldots, N), \label{e_eta_nbc}\\
	\delta_{k}^{\langle 1 \rangle}e_{\theta,k}^{(j)} = 0 \quad (k=0,K,\ j=0,\ldots, N), \label{e_theta_nbc}
\end{gather}
where $\xi_{1-3}$ and $\xi_{4-6}$ are defined as follows: 
\begin{gather}
	\xi_{1-3, k}^{(j+1)} := \sum_{i=1}^{3}\xi_{i, k}^{(j+1)} \quad (k=0, \ldots, K), \notag\\
	\xi_{1, k}^{(j+1)} := \partial_{t}\eta_{k}^{( j + 1)} - \frac{\eta_{k}^{(j+1)} - \eta_{k}^{(j)}}{\Delta t} \quad (k=0, \ldots, K), \label{xi1}\\
	\xi_{2, k}^{(j+1)} := \kappa_{0}^{2}\left( \delta_{k}^{\langle 2 \rangle}\tilde{\eta}_{k}^{(j+1)} - \partial_{x}^{2}\eta_{k}^{( j + 1)} \right) \quad (k=0, \ldots, K), \label{xi2}\\
	\xi_{3, k}^{(j+1)} := - \kappa \eta_{k}^{(j+1)}\left(\frac{ \gamma_{\varepsilon}\bigl( \delta_{k}^{+}\tilde{\theta}_{k}^{(j)}\bigr) + \gamma_{\varepsilon}\bigl( \delta_{k}^{-}\tilde{\theta}_{k}^{(j)}\bigr) }{2} - \gamma_{\varepsilon}\bigl( \partial_{x}\theta_{k}^{(j+1)}\bigr) \right) \quad (k=0, \ldots, K), \label{xi3}\\
	\xi_{4-6, k}^{(j+1)} := \sum_{i=4}^{6}\xi_{i, k}^{(j+1)} \quad (k=0, \ldots, K), \notag\\
	\xi_{4, k}^{(j+1)} := - \alpha_{0,k}^{(j+1)}\left(\frac{ \theta_{k}^{(j+1)} - \theta_{k}^{(j)} }{\Delta t} - \partial_{t}\theta_{k}^{(j+1)}\right) \quad (k=0, \ldots, K), \label{xi4}\\
	\xi_{5, k}^{(j+1)} := \nu^{2}\left(\delta_{k}^{\langle 2 \rangle}\tilde{\theta}_{k}^{(j+1)} - \partial_{x}^{2}\theta_{k}^{( j + 1)}\right) \quad (k=0, \ldots, K), \label{xi5}\\
	\begin{split}
	\xi_{6, k}^{(j+1)} 
		& := \frac{\kappa}{2}\left\{ \delta_{k}^{+}\!\left( \alpha\!\left( \tilde{\eta}_{k}^{(j+1)} \right)\!\gamma_{\varepsilon}' (\delta_{k}^{-}\tilde{\theta}_{k}^{(j+1)}) \right) + \delta_{k}^{-}\!\left( \alpha\!\left( \tilde{\eta}_{k}^{(j+1)} \right)\!\gamma_{\varepsilon}'(\delta_{k}^{+}\tilde{\theta}_{k}^{(j+1)})\right) \right\} \\
		& \quad - \kappa\partial_{x}\!\left( \alpha( \eta_{k}^{(j+1)} ) \gamma_{\varepsilon}'(\partial_{x}\theta_{k}^{(j+1)}) \right) \quad (k=0, \ldots, K).
	\end{split} \label{xi6}
\end{gather}
\end{lem}

\begin{proof}
	For any fixed $j=0,1,\ldots, N-1$, from the definition of $\bm{e}_{\eta}$, we have 
	\begin{align*}
	\frac{e_{\eta,k}^{(j+1)} - e_{\eta,k}^{(j)}}{\Delta t} 
		& = \frac{H_{k}^{(j+1)} - H_{k}^{(j)}}{\Delta t} - \frac{\eta_{k}^{(j+1)} - \eta_{k}^{(j)}}{\Delta t} \nonumber\\
		& = \frac{H_{k}^{(j+1)} - H_{k}^{(j)}}{\Delta t} - \partial_{t}\eta_{k}^{( j + 1)} + \partial_{t}\eta_{k}^{( j + 1)} - \frac{\eta_{k}^{(j+1)} - \eta_{k}^{(n)}}{\Delta t} \nonumber\\
		& = \kappa_{0}^{2}\delta_{k}^{\langle 2 \rangle}H_{k}^{(j+1)} - c( H_{k}^{(j+1)} - 1) - \kappa H_{k}^{(j+1)}\frac{ \gamma_{\varepsilon}\bigl( \delta_{k}^{+}\Theta_{k}^{(j)}\bigr) + \gamma_{\varepsilon}\bigl( \delta_{k}^{-}\Theta_{k}^{(j)}\bigr)  }{2} \nonumber\\
		& \quad  - \kappa_{0}^{2}\partial_{x}^{2}\eta_{k}^{( j + 1)} + c (\eta_{k}^{( j + 1)} - 1) + \kappa \eta_{k}^{( j + 1)} \gamma_{\varepsilon}\bigl( \partial_{x}\theta_{k}^{(j+1)}\bigr) 
		+ \xi_{1, k}^{(j+1)}. 
	\end{align*}
	Moreover, by direct calculation, we check that 
	\begin{gather*}
		c(H_{k}^{(j+1)} - 1) - c(\eta_{k}^{( j + 1)} - 1) = ce_{\eta,k}^{(j+1)} \quad (k=0, \ldots, K), \\
		\kappa_{0}^{2}(\delta_{k}^{\langle 2 \rangle}H_{k}^{(j+1)} - \partial_{x}^{2}\eta_{k}^{( j + 1)})
		= \kappa_{0}^{2}\left\{\delta_{k}^{\langle 2 \rangle}( e_{\eta,k}^{(j+1)} + \tilde{\eta}_{k}^{(j+1)} ) - \partial_{x}^{2}\eta_{k}^{( j + 1)} \right\} 
		= \kappa_{0}^{2}\delta_{k}^{\langle 2 \rangle}e_{\eta,k}^{(j+1)} + \xi_{2, k}^{(j+1)} 
	\end{gather*}
	for $k=0, \ldots, K$. 
	Furthermore, it follows from the property of the difference quotient that 
	\begin{align*}
		& \kappa\left\{ H_{k}^{(j+1)}\frac{ \gamma_{\varepsilon}\bigl( \delta_{k}^{+}\Theta_{k}^{(j)}\bigr) + \gamma_{\varepsilon}\bigl( \delta_{k}^{-}\Theta_{k}^{(j)}\bigr) }{2} - \eta_{k}^{( j + 1)} \gamma_{\varepsilon}\bigl( \partial_{x}\theta_{k}^{(j+1)}\bigr) \right\} \\
		= & \kappa\left\{  H_{k}^{(j+1)}\frac{ \gamma_{\varepsilon}\bigl( \delta_{k}^{+}\Theta_{k}^{(j)}\bigr) + \gamma_{\varepsilon}\bigl( \delta_{k}^{-}\Theta_{k}^{(j)}\bigr) }{2} - \eta_{k}^{(j+1)}\frac{ \gamma_{\varepsilon}\bigl( \delta_{k}^{+}\tilde{\theta}_{k}^{(j)}\bigr) + \gamma_{\varepsilon}\bigl( \delta_{k}^{-}\tilde{\theta}_{k}^{(j)}\bigr) }{2} \right. \\
		& \left. + \eta_{k}^{(j+1)}\frac{ \gamma_{\varepsilon}\bigl( \delta_{k}^{+}\tilde{\theta}_{k}^{(j)}\bigr) + \gamma_{\varepsilon}\bigl( \delta_{k}^{-}\tilde{\theta}_{k}^{(j)}\bigr) }{2} - \eta_{k}^{( j + 1)} \gamma_{\varepsilon}\bigl( \partial_{x}\theta_{k}^{(j+1)}\bigr) \right\} \\
		= & \kappa H_{k}^{(j+1)}\frac{ \gamma_{\varepsilon}\bigl( \delta_{k}^{+}\Theta_{k}^{(j)}\bigr) + \gamma_{\varepsilon}\bigl( \delta_{k}^{-}\Theta_{k}^{(j)}\bigr) - \gamma_{\varepsilon}\bigl( \delta_{k}^{+}\tilde{\theta}_{k}^{(j)}\bigr) - \gamma_{\varepsilon}\bigl( \delta_{k}^{-}\tilde{\theta}_{k}^{(j)}\bigr) }{2} \\
		& + \kappa \left(H_{k}^{(j+1)} - \eta_{k}^{(j+1)}\right)\frac{ \gamma_{\varepsilon}\bigl( \delta_{k}^{+}\tilde{\theta}_{k}^{(j)}\bigr) + \gamma_{\varepsilon}\bigl( \delta_{k}^{-}\tilde{\theta}_{k}^{(j)}\bigr) }{2} - \xi_{3, k}^{(j+1)} \\
		= & \kappa\frac{H_{k}^{(j+1)}}{2}\left( \frac{ d\gamma_{\varepsilon} }{ d(\delta_{k}^{+}\Theta_{k}^{(j)}, \delta_{k}^{+}\tilde{\theta}_{k}^{(j)}) }\delta_{k}^{+}e_{\theta,k}^{(j)} + \frac{ d\gamma_{\varepsilon} }{ d(\delta_{k}^{-}\Theta_{k}^{(j)}, \delta_{k}^{-}\tilde{\theta}_{k}^{(j)}) }\delta_{k}^{-}e_{\theta,k}^{(j)} \right) \\
		& + \kappa e_{\eta,k}^{(j+1)}\frac{ \gamma_{\varepsilon}\bigl( \delta_{k}^{+}\tilde{\theta}_{k}^{(j)}\bigr) + \gamma_{\varepsilon}\bigl( \delta_{k}^{-}\tilde{\theta}_{k}^{(j)}\bigr) }{2} - \xi_{3, k}^{(j+1)} \quad (k=0, \ldots, K). 
	\end{align*}
	From the above, we obtain \eqref{e_eta_eq}. 
	Similarly, from the definition of $\bm{e}_{\theta}$, we have 
	\begin{gather*}
		\alpha_{0,k}^{(j+1)}\frac{\Theta_{k}^{(j+1)} - \Theta_{k}^{(j)} }{\Delta t} - \alpha_{0,k}^{(j+1)} \partial_{t}\theta_{k}^{(j+1)} 
		= \alpha_{0,k}^{(j+1)}\frac{e_{\theta,k}^{(j+1)} - e_{\theta,k}^{(j)} }{\Delta t} - \xi_{4, k}^{(j+1)} \quad (k=0, \ldots, K), \\ 
		\nu^{2}( \delta_{k}^{\langle 2 \rangle}\Theta_{k}^{(j+1)} - \partial_{x}^{2}\theta_{k}^{( j + 1)} )
		= \nu^{2}\delta_{k}^{\langle 2 \rangle}e_{\theta,k}^{(j+1)} + \xi_{5, k}^{(j+1)} \quad (k=0, \ldots, K). 
	\end{gather*}
	In addition, by direct calculation, we see that 
	\begin{align*}
	& \frac{\kappa}{2}\!\left\{\! \delta_{k}^{+}\!\left(\! \alpha\!\left(\! H_{k}^{(j+1)} \!\right)\! \gamma_{\varepsilon}' (\delta_{k}^{-}\Theta_{k}^{(j+1)}) \!\right) \! + \! \delta_{k}^{-}\!\left(\! \alpha\!\left(\! H_{k}^{(j+1)} \!\right)\! \gamma_{\varepsilon}'(\delta_{k}^{+}\Theta_{k}^{(j+1)}) \!\right) \!\right\}
	\!\! - \! \kappa\partial_{x}\!\!\left(\! \alpha( \eta_{k}^{(j+1)} ) \gamma_{\varepsilon}'(\partial_{x}\theta_{k}^{(j+1)}) \!\right) \\
		= & \frac{\kappa}{2}\left\{ \delta_{k}^{+}\left( \alpha\left( H_{k}^{(j+1)} \right)\gamma_{\varepsilon}' (\delta_{k}^{-}\Theta_{k}^{(j+1)}) \right) + \delta_{k}^{-}\left( \alpha\left( H_{k}^{(j+1)} \right)\gamma_{\varepsilon}'(\delta_{k}^{+}\Theta_{k}^{(j+1)})\right) \right\} \\
		& - \frac{\kappa}{2}\left\{ \delta_{k}^{+}\left( \alpha\left( \tilde{\eta}_{k}^{(j+1)} \right)\gamma_{\varepsilon}' (\delta_{k}^{-}\tilde{\theta}_{k}^{(j+1)}) \right) + \delta_{k}^{-}\left( \alpha\left( \tilde{\eta}_{k}^{(j+1)} \right)\gamma_{\varepsilon}'(\delta_{k}^{+}\tilde{\theta}_{k}^{(j+1)})\right) \right\} + \xi_{6, k}^{(j+1)} \\
		= & \frac{\kappa}{2}
		\left[\delta_{k}^{+}\!
		\left\{
			\left(\! \alpha\!\left(\! H_{k}^{(j+1)} \!\right) \! - \! \alpha\!\left(\! \tilde{\eta}_{k}^{(j+1)} \!\right) \right)\! \gamma_{\varepsilon}' (\delta_{k}^{-}\Theta_{k}^{(j+1)}) 
			+ \alpha\!\left(\! \tilde{\eta}_{k}^{(j+1)} \!\right) \! \left(\! \gamma_{\varepsilon}' (\delta_{k}^{-}\Theta_{k}^{(j+1)}) - \gamma_{\varepsilon}' (\delta_{k}^{-}\tilde{\theta}_{k}^{(j+1)}) \!\right)
		\right\} \right.\\
		& \left. + \delta_{k}^{-}\!
		\left\{
			\left(\! \alpha\!\left(\! H_{k}^{(j+1)} \!\right) \! - \! \alpha\!\left(\! \tilde{\eta}_{k}^{(j+1)} \!\right) \right)\! \gamma_{\varepsilon}' (\delta_{k}^{+}\Theta_{k}^{(j+1)}) 
			+ \alpha\!\left(\! \tilde{\eta}_{k}^{(j+1)} \!\right) \! \left(\! \gamma_{\varepsilon}' (\delta_{k}^{+}\Theta_{k}^{(j+1)}) - \gamma_{\varepsilon}' (\delta_{k}^{+}\tilde{\theta}_{k}^{(j+1)}) \!\right)
		\right\}
		\right] \\
		& + \xi_{6, k}^{(j+1)}
	\end{align*}
	\begin{align*}
			= & \frac{\kappa}{2}
			\left[\delta_{k}^{+}
			\left\{
				\frac{H_{k}^{(j+1)} + \tilde{\eta}_{k}^{(j+1)}}{2} \gamma_{\varepsilon}' (\delta_{k}^{-}\Theta_{k}^{(j+1)})e_{\eta,k}^{(j+1)} 
				+ \alpha\left( \tilde{\eta}_{k}^{(j+1)} \right) \frac{ d\gamma_{\varepsilon}' }{ d(\delta_{k}^{-}\Theta_{k}^{(j+1)}, \delta_{k}^{-}\tilde{\theta}_{k}^{(j+1)}) }\delta_{k}^{-}e_{\theta,k}^{(j+1)}
			\right\} \right.\\
			& \left. + \delta_{k}^{-}
			\left\{
				\frac{H_{k}^{(j+1)} + \tilde{\eta}_{k}^{(j+1)}}{2} \gamma_{\varepsilon}' (\delta_{k}^{+}\Theta_{k}^{(j+1)})e_{\eta,k}^{(j+1)} 
				+ \alpha\left( \tilde{\eta}_{k}^{(j+1)} \right) \frac{ d\gamma_{\varepsilon}' }{ d(\delta_{k}^{+}\Theta_{k}^{(j+1)}, \delta_{k}^{+}\tilde{\theta}_{k}^{(j+1)}) }\delta_{k}^{+}e_{\theta,k}^{(j+1)}
			\right\}
			\right] \\
			& + \xi_{6, k}^{(j+1)} \quad (k=0, \ldots, K). 
	\end{align*}
	From the above, we obtain \eqref{e_theta_eq}. 
	Lastly, it holds from the definitions of $\bm{e}_{\eta}$ and $\bm{e}_{\theta}$, \eqref{bc1}, \eqref{bc2}, and the definitions of $\tilde{\eta}$ and $\tilde{\theta}$ that 
	\begin{gather*}
		0 = \delta_{k}^{\langle 1 \rangle}H_{k}^{(j)} 
		= \delta_{k}^{\langle 1 \rangle}\!\left(\! e_{\eta,k}^{(j)} + \tilde{\eta}_{k}^{( j )} \!\right) 
		= \delta_{k}^{\langle 1 \rangle}e_{\eta,k}^{(j)} + \delta_{k}^{\langle 1 \rangle}\tilde{\eta}_{k}^{( j )} 
		= \delta_{k}^{\langle 1 \rangle}e_{\eta,k}^{(j)} \quad (k=0,K,\ j=0,1,\ldots, N), \\
		0 = \delta_{k}^{\langle 1 \rangle}\Theta_{k}^{(j)} 
		= \delta_{k}^{\langle 1 \rangle}\!\left( e_{\theta,k}^{(j)} + \tilde{\theta}_{k}^{( j )}\right) 
		= \delta_{k}^{\langle 1 \rangle}e_{\theta,k}^{(j)} + \delta_{k}^{\langle 1 \rangle}\tilde{\theta}_{k}^{( j )} 
		= \delta_{k}^{\langle 1 \rangle}e_{\theta,k}^{(j)} \quad (k=0,K,\ j=0,1,\ldots, N). 
	\end{gather*}
\end{proof}

\begin{lem}\label{e_eta_est}
Denotes the bounds by 
\begin{equation}
	\max_{0 \leq j \leq N}\left\{ \left\| \bm{H}^{(j)}\right\|_{L_{\rm d}^{\infty}}, \left\| \bm{\eta}^{(j)}\right\|_{L_{\rm d}^{\infty}}  \right\} \leq C_{1}, \label{eta_bound}
\end{equation}
where $\bm{\eta}^{(j)} := \{\eta_{k}^{(j)}\}_{k = 0}^{K} = \{\eta(j\Delta t, k\Delta x)\}_{k = 0}^{K} \in \mathbb{R}^{K+1}$.
Then, the following inequality holds:
\begin{align}
	\frac{1}{2\Delta t}\left( \left\| \bm{e}_{\eta}^{(j+1)}\right\|_{ L_{\rm d}^{2} }^{2} - \left\| \bm{e}_{\eta}^{(j)}\right\|_{ L_{\rm d}^{2} }^{2} \right)
	\leq & \left\{ \frac{1}{2}\left(\frac{\kappa C_{1} }{\nu}\right)^{2} - \left(\kappa\varepsilon + \frac{c}{2}\right)\right\} \left\| \bm{e}_{\eta}^{(j+1)}\right\|_{ L_{\rm d}^{2} }^{2} \notag\\
	& + \frac{\nu^{2} }{2}\left\| D\bm{e}_{\theta}^{(j)} \right\|^{2} + \frac{1}{2c}\left\| \bm{\xi}_{1-3}^{(j+1)} \right\|_{ L_{\rm d}^{2} }^{2} \quad (j = 0,\ldots, N-1). \label{err_eta_est2}
\end{align}
\end{lem}

\begin{proof}
	Let us fix $j=0,1,\ldots, N-1$. 
	Also, let us multiply both sides of \eqref{e_eta_eq} by $e_{\eta,k}^{(j+1)}$ and sum each of the resulting equality over $k = 0,1, \ldots, K$. 
	Then, by using \eqref{tool_ineq}, we can estimate the left-hand side of the resulting equality as follows:
	\begin{align*}
	\sum_{k=0}^{K}{}^{\prime\prime}\frac{e_{\eta,k}^{(j+1)} - e_{\eta,k}^{(j)} }{\Delta t} e_{\eta,k}^{(j+1)}\Delta x
		& \geq \frac{1}{2\Delta t}\sum_{k=0}^{K}{}^{\prime\prime}\left| e_{\eta,k}^{(j+1)} \right|^{2}\Delta x
		- \frac{1}{2\Delta t}\sum_{k=0}^{K}{}^{\prime\prime}\left| e_{\eta,k}^{(j)} \right|^{2}\Delta x \\
		& = \frac{1}{2\Delta t}\left( \left\| \bm{e}_{\eta}^{(j+1)}\right\|_{ L_{\rm d}^{2} }^{2} - \left\| \bm{e}_{\eta}^{(j)}\right\|_{ L_{\rm d}^{2} }^{2} \right). 
	\end{align*} 
	Next, we estimate the right-hand side of the resulting equality. 
	It follows from \eqref{e_eta_nbc} and the summation by parts formula (Lemma \ref{sbp5}) that 
	\begin{align*}
	\kappa_{0}^{2}\sum_{k=0}^{K}\!{}^{\prime\prime}\!\left(\! \delta_{k}^{\langle 2 \rangle}e_{\eta,k}^{(j+1)} \!\right)\! e_{\eta,k}^{(j+1)} \! \Delta x
		= & - \kappa_{0}^{2}\sum_{k=0}^{K-1}\!\left| \delta_{k}^{+}e_{\eta,k}^{(j+1)} \right|^{2} \! \Delta x
		+ \kappa_{0}^{2}\left[ \left(\! \delta_{k}^{\langle 1 \rangle}e_{\eta,k}^{(j+1)} \!\right)\! e_{\eta,k}^{(j+1)} \right]_{0}^{K} \\
		= & - \kappa_{0}^{2}\!\left\| D\bm{e}_{\eta}^{(j+1)} \right\|^{2}. 
	\end{align*}
	From the definition of the discrete $L^{2}$-norm, we see that 
	\begin{equation*}
	- c\sum_{k=0}^{K}{}^{\prime\prime}\left| e_{\eta,k}^{(j+1)} \right|^{2} \Delta x
		= - c \left\| \bm{e}_{\eta}^{(j+1)} \right\|_{L_{\rm d}^{2}}^{2}. 
	\end{equation*}
	Furthermore, it holds from the inequality $\gamma_{\varepsilon}(u) = \sqrt{\varepsilon^{2} + u^{2}} \geq \varepsilon$ that 
	\begin{equation*}
	-\kappa\sum_{k=0}^{K}{}^{\prime\prime}\frac{ \gamma_{\varepsilon}\bigl( \delta_{k}^{+}\tilde{\theta}_{k}^{(j)}\bigr) + \gamma_{\varepsilon}\bigl( \delta_{k}^{-}\tilde{\theta}_{k}^{(j)}\bigr) }{2}\left| e_{\eta,k}^{(j+1)} \right|^{2} \Delta x
		\leq -\kappa\varepsilon\left\| \bm{e}_{\eta}^{(j+1)} \right\|_{L_{\rm d}^{2}}^{2}. 
	\end{equation*}
	By using \eqref{eta_bound} and the Young inequality, we have 
	\begin{align*}
	& -\frac{\kappa}{2}\sum_{k=0}^{K}{}^{\prime\prime}H_{k}^{(j+1)}\left( \frac{ d\gamma_{\varepsilon} }{ d(\delta_{k}^{+}\Theta_{k}^{(j)}, \delta_{k}^{+}\tilde{\theta}_{k}^{(j)}) }\delta_{k}^{+}e_{\theta,k}^{(j)} + \frac{ d\gamma_{\varepsilon} }{ d(\delta_{k}^{-}\Theta_{k}^{(j)}, \delta_{k}^{-}\tilde{\theta}_{k}^{(j)}) }\delta_{k}^{-}e_{\theta,k}^{(j)} \right) e_{\eta,k}^{(j+1)}\Delta x \\
		\leq & \frac{\kappa C_{1}}{2}\sum_{k=0}^{K}{}^{\prime\prime}\left| \frac{ d\gamma_{\varepsilon} }{ d(\delta_{k}^{+}\Theta_{k}^{(j)}, \delta_{k}^{+}\tilde{\theta}_{k}^{(j)}) }\delta_{k}^{+}e_{\theta,k}^{(j)} + \frac{ d\gamma_{\varepsilon} }{ d(\delta_{k}^{-}\Theta_{k}^{(j)}, \delta_{k}^{-}\tilde{\theta}_{k}^{(j)}) }\delta_{k}^{-}e_{\theta,k}^{(j)} \right| \left| e_{\eta,k}^{(j+1)} \right| \Delta x \\
		\leq & \frac{\kappa C_{1}}{2} \cdot \frac{\kappa C_{1}}{\nu^{2}}\left\| \bm{e}_{\eta}^{(j+1)} \right\|_{L_{\rm d}^{2}}^{2} \\
		& + \frac{\kappa C_{1}}{2} \cdot \frac{\nu^{2}}{4\kappa C_{1}}\sum_{k=0}^{K}{}^{\prime\prime}\left| \frac{ d\gamma_{\varepsilon} }{ d(\delta_{k}^{+}\Theta_{k}^{(j)}, \delta_{k}^{+}\tilde{\theta}_{k}^{(j)}) }\delta_{k}^{+}e_{\theta,k}^{(j)} + \frac{ d\gamma_{\varepsilon} }{ d(\delta_{k}^{-}\Theta_{k}^{(j)}, \delta_{k}^{-}\tilde{\theta}_{k}^{(j)}) }\delta_{k}^{-}e_{\theta,k}^{(j)} \right|^{2} \Delta x \\
		\leq & \frac{1}{2}\left( \frac{\kappa C_{1}}{\nu}\right)^{2}\left\| \bm{e}_{\eta}^{(j+1)} \right\|_{L_{\rm d}^{2}}^{2} \\
		& + \frac{\nu^{2}}{4}\sum_{k=0}^{K}{}^{\prime\prime}\left( \left| \frac{ d\gamma_{\varepsilon} }{ d(\delta_{k}^{+}\Theta_{k}^{(j)}, \delta_{k}^{+}\tilde{\theta}_{k}^{(j)}) }\right|^{2} \left|\delta_{k}^{+}e_{\theta,k}^{(j)}\right|^{2} + \left|\frac{ d\gamma_{\varepsilon} }{ d(\delta_{k}^{-}\Theta_{k}^{(j)}, \delta_{k}^{-}\tilde{\theta}_{k}^{(j)}) }\right|^{2} \left|\delta_{k}^{-}e_{\theta,k}^{(j)} \right|^{2} \right)\Delta x. 
	\end{align*}
	From \eqref{diff_f} and \eqref{e_theta_nbc}, we observe that 
	\begin{align*}
	& -\frac{\kappa}{2}\sum_{k=0}^{K}{}^{\prime\prime}H_{k}^{(j+1)}\left( \frac{ d\gamma_{\varepsilon} }{ d(\delta_{k}^{+}\Theta_{k}^{(j)}, \delta_{k}^{+}\tilde{\theta}_{k}^{(j)}) }\delta_{k}^{+}e_{\theta,k}^{(j)} + \frac{ d\gamma_{\varepsilon} }{ d(\delta_{k}^{-}\Theta_{k}^{(j)}, \delta_{k}^{-}\tilde{\theta}_{k}^{(j)}) }\delta_{k}^{-}e_{\theta,k}^{(j)} \right) e_{\eta,k}^{(j+1)}\Delta x \\
		\leq & \frac{1}{2}\left( \frac{\kappa C_{1}}{\nu}\right)^{2}\left\| \bm{e}_{\eta}^{(j+1)} \right\|_{L_{\rm d}^{2}}^{2}
		+ \frac{\nu^{2}}{2}\sum_{k=0}^{K}{}^{\prime\prime}\frac{ \left|\delta_{k}^{+}e_{\theta,k}^{(j)}\right|^{2} + \left|\delta_{k}^{-}e_{\theta,k}^{(j)} \right|^{2} }{2}\Delta x \\
		= & \frac{1}{2}\left( \frac{\kappa C_{1}}{\nu}\right)^{2}\left\| \bm{e}_{\eta}^{(j+1)} \right\|_{L_{\rm d}^{2}}^{2}
		+ \frac{\nu^{2}}{2}\left\| D\bm{e}_{\theta}^{(j)} \right\|^{2}. 
	\end{align*}
	Again from the Young inequality, we get 
	\begin{equation*}
	\sum_{k=0}^{K}{}^{\prime\prime}\xi_{1-3, k}^{(j+1)} e_{\eta,k}^{(j+1)}\Delta x
		\leq \frac{1}{2c}\left\| \bm{\xi}_{1-3}^{(j+1)} \right\|_{ L_{\rm d}^{2} }^{2} + \frac{c}{2}\left\| \bm{e}_{\eta}^{(j+1)} \right\|_{L_{\rm d}^{2}}^{2}.
	\end{equation*}
	From the above, we obtain 
	\begin{align*}
	& \frac{1}{2\Delta t}\left( \left\| \bm{e}_{\eta}^{(j+1)}\right\|_{ L_{\rm d}^{2} }^{2} - \left\| \bm{e}_{\eta}^{(j)}\right\|_{ L_{\rm d}^{2} }^{2} \right) \\
		\leq & - \kappa_{0}^{2}\left\| D\bm{e}_{\eta}^{(j+1)} \right\|_{L_{\rm d}^{2}}^{2} - c \left\| \bm{e}_{\eta}^{(j+1)} \right\|_{L_{\rm d}^{2}}^{2}
		- \kappa\varepsilon\left\| \bm{e}_{\eta}^{(j+1)} \right\|_{L_{\rm d}^{2}}^{2} 
		+ \frac{1}{2} \left( \frac{\kappa C_{1}}{\nu} \right)^{2}\left\| \bm{e}_{\eta}^{(j+1)} \right\|_{L_{\rm d}^{2}}^{2} \\
		& + \frac{\nu^{2}}{2}\left\| D\bm{e}_{\theta}^{(j)} \right\|^{2} 
		+\frac{1}{2c}\left\| \bm{\xi}_{1-3}^{(j+1)} \right\|_{ L_{\rm d}^{2} }^{2} + \frac{c}{2}\left\| \bm{e}_{\eta}^{(j+1)} \right\|_{L_{\rm d}^{2}}^{2} \\
		\leq & \left\{ \frac{1}{2}\left(\frac{\kappa C_{1} }{\nu}\right)^{2} - \left(\kappa\varepsilon + \frac{c}{2}\right)\right\} \left\| \bm{e}_{\eta}^{(j+1)}\right\|_{ L_{\rm d}^{2} }^{2}  
		+ \frac{\nu^{2} }{2}\left\| D\bm{e}_{\theta}^{(j)} \right\|^{2} + \frac{1}{2c}\left\| \bm{\xi}_{1-3}^{(j+1)} \right\|_{ L_{\rm d}^{2} }^{2}. 
	\end{align*}
\end{proof}

\begin{lem}\label{e_theta_est}
    Let us define $\tilde{L}_{\alpha_{0}} := L_{\alpha_{0}} + 1 > 0$ with the Lipschitz constant $ L_{\alpha_0} $ of $ \alpha_0 $. 
Then, the following inequality holds:	
\begin{align}
	& \frac{1}{2\Delta t}\left( \left\| \sqrt{\bm{\alpha_{0}}^{(j+1)}}\bm{e}_{\theta}^{(j+1)}\right\|_{ L_{\rm d}^{2} }^{2} - \left\| \sqrt{\bm{\alpha_{0}}^{(j)}}\bm{e}_{\theta}^{(j)}\right\|_{ L_{\rm d}^{2} }^{2} \right) + \frac{\nu^{2}}{2}\left\| D\bm{e}_{\theta}^{(j+1)}\right\|^{2} \notag\\
		\leq &\frac{\tilde{L}_{\alpha_{0}}}{2\delta_{0} } \left\| \sqrt{\bm{\alpha_{0}}^{(j+1)}}\bm{e}_{\theta}^{(j+1)}\right\|_{ L_{\rm d}^{2} }^{2} 
		\! + \! \frac{\tilde{L}_{\alpha_{0}}}{2\delta_{0} } \left\| \sqrt{\bm{\alpha_{0}}^{(j)}}\bm{e}_{\theta}^{(j)}\right\|_{ L_{\rm d}^{2} }^{2} 
		\! + \! \frac{1}{2}\left( \frac{\kappa C_{1} }{\nu} \right)^{2}\left\| \bm{e}_{\eta}^{(j+1)}\right\|_{ L_{\rm d}^{2} }^{2} 
		\! + \! \frac{1}{2\tilde{L}_{\alpha_{0}}}\left\| \bm{\xi}_{4-6}^{(j+1)} \right\|_{L_{\rm d}^{2}}^{2}  \label{err_theta_est2}
\end{align}
	for $j = 0,\ldots, N-1$, where we use the following vector notation: 
	\begin{equation}
		\sqrt{\bm{\alpha_{0}}^{(j+1)}}\bm{e}_{\theta}^{(j+1)} := \left\{ \sqrt{\alpha_{0,k}^{(j+1)}}e_{\theta,k}^{(j+1)} \right\}_{k = 0}^{K} \quad \in \mathbb{R}^{K+1}. \label{vec_a_th}
	\end{equation}
\end{lem}

\begin{proof}
	Let us fix $j=0,1,\ldots, N-1$. 
	Also, let us multiply both sides of \eqref{e_theta_eq} by $e_{\theta,k}^{(j+1)}$ and sum each of the resulting equality over $k = 0,1, \ldots, K$. 
	Then, by using \eqref{tool_ineq}, we can estimate the left-hand side of the resulting equality as follows:
	\begin{align*}
	& \sum_{k=0}^{K}{}^{\prime\prime}\alpha_{0,k}^{(j+1)}\frac{e_{\theta,k}^{(j+1)} - e_{\theta,k}^{(j)} }{\Delta t} e_{\theta,k}^{(j+1)}\Delta x \\
		\geq & \frac{1}{2\Delta t}\sum_{k=0}^{K}{}^{\prime\prime}\alpha_{0,k}^{(j+1)}\left( \left| e_{\theta,k}^{(j+1)} \right|^{2} - \left| e_{\theta,k}^{(j)} \right|^{2} \right) \Delta x \\
		= & \frac{1}{2\Delta t}\sum_{k=0}^{K}{}^{\prime\prime}\left( 
			\alpha_{0,k}^{(j+1)} \left| e_{\theta,k}^{(j+1)} \right|^{2} 
			- \alpha_{0,k}^{(j+1)} \left| e_{\theta,k}^{(j)} \right|^{2} 
			+ \alpha_{0,k}^{(j)} \left| e_{\theta,k}^{(j)} \right|^{2}  
			- \alpha_{0,k}^{(j)} \left| e_{\theta,k}^{(j)} \right|^{2} 
		\right) \Delta x \\
		\geq & \frac{1}{2\Delta t}\!\left( \left\| \sqrt{\bm{\alpha_{0}}^{(j+1)}}\bm{e}_{\theta}^{(j+1)}\right\|_{ L_{\rm d}^{2} }^{2} \! - \! \left\| \sqrt{\bm{\alpha_{0}}^{(j)}}\bm{e}_{\theta}^{(j)}\right\|_{ L_{\rm d}^{2} }^{2} \right) 
		-\frac{\tilde{L}_{\alpha_{0}}}{2}\left\| \bm{e}_{\theta}^{(j)}\right\|_{ L_{\rm d}^{2} }^{2}. 
	\end{align*} 
	Next, we estimate the right-hand side of the resulting equality. 
	It follows from \eqref{e_theta_nbc} and the summation by parts formula (Lemma \ref{sbp5}) that 
	\begin{equation*}
	\nu^{2}\sum_{k=0}^{K}{}^{\prime\prime}\!\left( \delta_{k}^{\langle 2 \rangle}e_{\theta,k}^{(j+1)} \right)\! e_{\theta,k}^{(j+1)} \Delta x
		= - \nu^{2}\left\| D\bm{e}_{\theta}^{(j+1)} \right\|^{2}. 
	\end{equation*}
	Furthermore, we see from the Young inequality that 
	\begin{equation*}
	\sum_{k=0}^{K}{}^{\prime\prime}\xi_{4-6, k}^{(j+1)} e_{\theta,k}^{(j+1)}\Delta x
		\leq \frac{1}{2\tilde{L}_{\alpha_{0}}}\left\| \bm{\xi}_{4-6}^{(j+1)} \right\|_{ L_{\rm d}^{2} }^{2} + \frac{\tilde{L}_{\alpha_{0}}}{2}\left\| \bm{e}_{\theta}^{(j+1)} \right\|_{L_{\rm d}^{2}}^{2}. 
	\end{equation*}
	For reasons of space limitation, the superscript $(j+1)$ is omitted hereafter. 
	Here, it holds from the summation by parts formula (Lemma \ref{sbp2}) that 
	\begin{align*}
	& \frac{\kappa}{2}\sum_{k=0}^{K}{}^{\prime\prime}\left[\delta_{k}^{+}
	\left\{
		\frac{H_{k} + \eta_{k}}{2} \gamma_{\varepsilon}' (\delta_{k}^{-}\Theta_{k})e_{\eta,k} 
		+ \alpha\left( \eta_{k} \right)  \frac{ d\gamma_{\varepsilon}' }{ d(\delta_{k}^{-}\Theta_{k}, \delta_{k}^{-}\tilde{\theta}_{k}) }\delta_{k}^{-}e_{\theta,k}
	\right\}
	\right]e_{\theta,k}\Delta x \\
		= & - \frac{\kappa}{2}\sum_{k=1}^{K}
		\left(
			\frac{H_{k} + \eta_{k}}{2} \gamma_{\varepsilon}' (\delta_{k}^{-}\Theta_{k})e_{\eta,k} 
			+ \alpha\left( \eta_{k} \right) \frac{ d\gamma_{\varepsilon}' }{ d(\delta_{k}^{-}\Theta_{k}, \delta_{k}^{-}\tilde{\theta}_{k}) }\delta_{k}^{-}e_{\theta,k}
		\right)
		\delta_{k}^{-}e_{\theta,k}\Delta x \\
		& + \underset{\rm =: (d.B.T.1)}{\underline{ \left[\left\{\mu_{k}^{+}
		\left(
			\frac{H_{k} + \eta_{k}}{2} \gamma_{\varepsilon}' (\delta_{k}^{-}\Theta_{k})e_{\eta,k} 
			+ \alpha\left( \eta_{k} \right) \frac{ d\gamma_{\varepsilon}' }{ d(\delta_{k}^{-}\Theta_{k}, \delta_{k}^{-}\tilde{\theta}_{k}) }\delta_{k}^{-}e_{\theta,k}
		\right)
		\right\}
		e_{\theta,k} \right]_{0}^{K}
		}}.
	\end{align*}
	In the same manner as above, we observe that 
	\begin{align*}
	& \frac{\kappa}{2}\sum_{k=0}^{K}{}^{\prime\prime}\left[\delta_{k}^{-}\!
	\left\{
		\frac{H_{k} + \eta_{k}}{2} \gamma_{\varepsilon}' (\delta_{k}^{+}\Theta_{k})e_{\eta,k} 
		+ \alpha\left( \eta_{k} \right) \frac{ d\gamma_{\varepsilon}' }{ d(\delta_{k}^{+}\Theta_{k}, \delta_{k}^{+}\tilde{\theta}_{k}) }\delta_{k}^{+}e_{\theta,k}
	\right\}
	\right]e_{\theta,k}\Delta x \\
		= & - \frac{\kappa}{2}\sum_{k=0}^{K-1}
		\left(
			\frac{H_{k} + \eta_{k}}{2} \gamma_{\varepsilon}' (\delta_{k}^{+}\Theta_{k})e_{\eta,k} 
			+ \alpha\left( \eta_{k} \right) \frac{ d\gamma_{\varepsilon}' }{ d(\delta_{k}^{+}\Theta_{k}, \delta_{k}^{+}\tilde{\theta}_{k}) }\delta_{k}^{+}e_{\theta,k}
		\right)
		\delta_{k}^{+}e_{\theta,k}\Delta x \\
		& + \underset{\rm =: (d.B.T.2)}{\underline{ \left[\left\{\mu_{k}^{-}
		\left(
			\frac{H_{k} + \eta_{k}}{2} \gamma_{\varepsilon}' (\delta_{k}^{+}\Theta_{k})e_{\eta,k} 
			+ \alpha\left( \eta_{k} \right) \frac{ d\gamma_{\varepsilon}' }{ d(\delta_{k}^{+}\Theta_{k}, \delta_{k}^{+}\tilde{\theta}_{k}) }\delta_{k}^{+}e_{\theta,k}
		\right)
		\right\}
		e_{\theta,k} \right]_{0}^{K}
		}}.
	\end{align*}
	It is seen from the direct calculation that 
	\begin{align*}
	& \sum_{k=1}^{K}
	\left\{
		\frac{H_{k} + \eta_{k}}{2} \gamma_{\varepsilon}' (\delta_{k}^{-}\Theta_{k})e_{\eta,k} 
		+ \alpha\left( \eta_{k} \right) \frac{ d\gamma_{\varepsilon}' }{ d(\delta_{k}^{-}\Theta_{k}, \delta_{k}^{-}\tilde{\theta}_{k}) }\delta_{k}^{-}e_{\theta,k}
	\right\}
	\delta_{k}^{-}e_{\theta,k}\Delta x \\
		= &\sum_{k=1}^{K}
		\left\{
			\frac{H_{k} + \eta_{k}}{2} \gamma_{\varepsilon}' (\delta_{k}^{+}\Theta_{k-1})e_{\eta,k} 
			+ \alpha\left( \eta_{k} \right) \frac{ d\gamma_{\varepsilon}' }{ d(\delta_{k}^{+}\Theta_{k-1}, \delta_{k}^{+}\tilde{\theta}_{k-1}) }\delta_{k}^{+}e_{\theta,k-1}
		\right\}
		\delta_{k}^{+}e_{\theta,k-1}\Delta x \\
		= & \sum_{k=0}^{K-1}
		\left\{
			\frac{H_{k+1} + \eta_{k+1}}{2} \gamma_{\varepsilon}' (\delta_{k}^{+}\Theta_{k})e_{\eta,k+1} 
			+ \alpha\left( \eta_{k+1} \right) \frac{ d\gamma_{\varepsilon}' }{ d(\delta_{k}^{+}\Theta_{k}, \delta_{k}^{+}\tilde{\theta}_{k}) }\delta_{k}^{+}e_{\theta,k}
		\right\}
		\delta_{k}^{+}e_{\theta,k}\Delta x. 
	\end{align*}
	Additionally, it follows from \eqref{bc1}--\eqref{bc2}, \eqref{e_eta_nbc}--\eqref{e_theta_nbc}, and the definitions of $\tilde{\eta}$ and $\tilde{\theta}$ that $\mbox{(d.B.C.1)} + \mbox{(d.B.C.2)} = 0$. 
	Thus, we obtain from the following inequalities: $d\gamma_{\varepsilon}'/d(u,v) \geq 0$ for all $u,v \in \mathbb{R}$, $|\gamma_{\varepsilon}'(u)| \leq 1$ for all $u\in \mathbb{R}$, and the H\"{o}lder inequality that  
	\begin{align*}
	& \frac{\kappa}{2}\sum_{k=0}^{K}{}^{\prime\prime}\left[\delta_{k}^{+}
	\left\{
		\frac{H_{k} + \eta_{k}}{2} \gamma_{\varepsilon}' (\delta_{k}^{-}\Theta_{k})e_{\eta,k} 
		+ \alpha\left( \eta_{k} \right) \frac{ d\gamma_{\varepsilon}' }{ d(\delta_{k}^{-}\Theta_{k}, \delta_{k}^{-}\tilde{\theta}_{k}) }\delta_{k}^{-}e_{\theta,k}
	\right\}
	\right]e_{\theta,k}\Delta x \\[-1pt]
	& + \frac{\kappa}{2}\sum_{k=0}^{K}{}^{\prime\prime}\left[\delta_{k}^{-}
	\left\{
		\frac{H_{k} + \eta_{k}}{2} \gamma_{\varepsilon}' (\delta_{k}^{+}\Theta_{k})e_{\eta,k} 
		+ \alpha\left( \eta_{k} \right) \frac{ d\gamma_{\varepsilon}' }{ d(\delta_{k}^{+}\Theta_{k}, \delta_{k}^{+}\tilde{\theta}_{k}) }\delta_{k}^{+}e_{\theta,k}
	\right\}
	\right]e_{\theta,k}\Delta x \\[-1pt]
		= & - \frac{\kappa}{2}\sum_{k=0}^{K-1}
			\left(
				\frac{H_{k+1} + \eta_{k+1}}{2}e_{\eta,k+1} 
				+ \frac{H_{k} + \eta_{k}}{2}e_{\eta,k} 
			\right) \gamma_{\varepsilon}' (\delta_{k}^{+}\Theta_{k})
			\delta_{k}^{+}e_{\theta,k}\Delta x \\[-1pt]
		& - \frac{\kappa}{2}\sum_{k=0}^{K-1}
		\left(
			\alpha\!\left( \eta_{k+1} \right) + \alpha\!\left( \eta_{k} \right)
		\right)\frac{ d\gamma_{\varepsilon}' }{ d(\delta_{k}^{+}\Theta_{k}, \delta_{k}^{+}\tilde{\theta}_{k}) }
		\left| \delta_{k}^{+}e_{\theta,k} \right|^{2} \Delta x \\[-1pt]
		\leq & - \frac{\kappa}{2}\sum_{k=0}^{K-1}
			\left(
				\frac{H_{k+1} + \eta_{k+1}}{2}e_{\eta,k+1} 
				+ \frac{H_{k} + \eta_{k}}{2}e_{\eta,k} 
			\right) \gamma_{\varepsilon}' (\delta_{k}^{+}\Theta_{k})
			\delta_{k}^{+}e_{\theta,k}\Delta x \\[-1pt]
		\leq & \frac{\kappa}{2}\sum_{k=0}^{K-1}
		\left|
			\frac{H_{k+1} + \eta_{k+1}}{2}e_{\eta,k+1} 
			+ \frac{H_{k} + \eta_{k}}{2}e_{\eta,k} 
		\right| \left| \gamma_{\varepsilon}' (\delta_{k}^{+}\Theta_{k}) \right|
		\left| \delta_{k}^{+}e_{\theta,k} \right| \Delta x \\[-1pt]
		\leq & \frac{\kappa}{2}\left\| D\bm{e}_{\theta} \right\|
		\left\{ 
		\sum_{k=0}^{K-1}
		\left|
			\frac{H_{k+1} + \eta_{k+1}}{2}e_{\eta,k+1} 
			+ \frac{H_{k} + \eta_{k}}{2}e_{\eta,k} 
		\right|^{2}\Delta x 
		\right\}^{\frac{1}{2}}. 
	\end{align*}
	Furthermore, by using \eqref{eta_bound}, we obtain 
	\begin{align*}
	& \left\{ 
		\sum_{k=0}^{K-1}
		\left|
			\frac{H_{k+1} + \eta_{k+1}}{2}e_{\eta,k+1} 
			+ \frac{H_{k} + \eta_{k}}{2}e_{\eta,k} 
		\right|^{2}\Delta x 
		\right\}^{\frac{1}{2}} \\[-1pt]
		\leq & \left\{ 
			\frac{1}{2}\sum_{k=0}^{K-1}
			\left(
			\left|
				H_{k+1} + \eta_{k+1}
			\right|^{2}
			\left|
				e_{\eta,k+1} 
			\right|^{2}
			+ \left|
				H_{k} + \eta_{k}
			\right|^{2}
			\left|
				e_{\eta,k} 
			\right|^{2}
			\right)\Delta x 
			\right\}^{\frac{1}{2}} \\[-1pt]
		\leq & \left[ 
			\sum_{k=0}^{K-1}
			\left\{
			\left(
			\left|
				H_{k+1} 
			\right|^{2}
			+ \left|
				\eta_{k+1}
			\right|^{2}
			\right)
			\left|
				e_{\eta,k+1} 
			\right|^{2}
			+ \left(
			\left|
				H_{k} 
			\right|^{2}
			+ \left|
				\eta_{k}
			\right|^{2}
			\right)
			\left|
				e_{\eta,k} 
			\right|^{2}
			\right\} \Delta x 
			\right]^{\frac{1}{2}} \\[-1pt]
		\leq & \left\{ 
			2C_{1}^{2}\sum_{k=0}^{K-1}
			\left(
			\left|
				e_{\eta,k+1}
			\right|^{2}
			+ \left|
				e_{\eta,k} 
			\right|^{2}
			\right) \Delta x 
			\right\}^{\frac{1}{2}} 
		= 2C_{1}\left\| \bm{e}_{\eta}\right\|_{ L_{\rm d}^{2} }. 
	\end{align*}
	Hence, by using the Young inequality, we get 
	\begin{align*}
		& \frac{\kappa}{2}\sum_{k=0}^{K}{}^{\prime\prime}\left[\delta_{k}^{+}\!
		\left\{
			\frac{H_{k} + \eta_{k}}{2} \gamma_{\varepsilon}' (\delta_{k}^{-}\Theta_{k})e_{\eta,k} 
			+ \alpha\left( \eta_{k} \right) \frac{ d\gamma_{\varepsilon}' }{ d(\delta_{k}^{-}\Theta_{k}, \delta_{k}^{-}\tilde{\theta}_{k}) }\delta_{k}^{-}e_{\theta,k}
		\right\}
		\right]e_{\theta,k}\Delta x \\[-1pt]
		& + \frac{\kappa}{2}\sum_{k=0}^{K}{}^{\prime\prime}\left[\delta_{k}^{-}\!
		\left\{
			\frac{H_{k} + \eta_{k}}{2} \gamma_{\varepsilon}' (\delta_{k}^{+}\Theta_{k})e_{\eta,k} 
			+ \alpha\left( \eta_{k} \right) \frac{ d\gamma_{\varepsilon}' }{ d(\delta_{k}^{+}\Theta_{k}, \delta_{k}^{+}\tilde{\theta}_{k}) }\delta_{k}^{+}e_{\theta,k}
		\right\}
		\right]e_{\theta,k}\Delta x \\[-1pt]
		\leq & \kappa C_{1}\left\| \bm{e}_{\eta}\right\|_{ L_{\rm d}^{2} } 
		\left\| D\bm{e}_{\theta} \right\| \\[-1pt]
		\leq & \kappa C_{1} \! \left( \frac{\kappa C_{1}}{2\nu^{2}}\left\| \bm{e}_{\eta}\right\|_{ L_{\rm d}^{2} }^{2} 
		+ \frac{\nu^{2}}{2\kappa C_{1}}\left\| D\bm{e}_{\theta} \right\| \right)
		= \frac{1}{2}\!\left(\frac{\kappa C_{1}}{\nu}\right)^{2}\!\left\| \bm{e}_{\eta}\right\|_{ L_{\rm d}^{2} }^{2} 
		+ \frac{\nu^{2}}{2}\left\| D\bm{e}_{\theta} \right\|.
	\end{align*}
	As a result, we obtain 
	\begin{align*}
	& \frac{1}{2\Delta t}\!\left( \left\| \sqrt{\bm{\alpha_{0}}^{(j+1)}}\bm{e}_{\theta}^{(j+1)}\right\|_{ L_{\rm d}^{2} }^{2} \! - \! \left\| \sqrt{\bm{\alpha_{0}}^{(j)}}\bm{e}_{\theta}^{(j)}\right\|_{ L_{\rm d}^{2} }^{2} \right) \\
		\leq & - \nu^{2}\left\| D\bm{e}_{\theta}^{(j+1)} \right\|^{2} 
		+ \frac{1}{2}\left(\frac{\kappa C_{1}}{\nu}\right)^{2}\left\| \bm{e}_{\eta}^{(j+1)}\right\|_{ L_{\rm d}^{2} }^{2} 
		+ \frac{\nu^{2}}{2}\left\| D\bm{e}_{\theta}^{(j+1)} \right\| \\
		& + \frac{1}{2\tilde{L}_{\alpha_{0}}}\left\| \bm{\xi}_{4-6}^{(j+1)} \right\|_{ L_{\rm d}^{2} }^{2} 
		+ \frac{\tilde{L}_{\alpha_{0}}}{2}\left\| \bm{e}_{\theta}^{(j+1)} \right\|_{L_{\rm d}^{2}}^{2}
		+ \frac{\tilde{L}_{\alpha_{0}}}{2}\left\| \bm{e}_{\theta}^{(j)} \right\|_{L_{\rm d}^{2}}^{2} \\
		\leq & \frac{1}{2}\!\left(\! \frac{\kappa C_{1} }{\nu} \!\right)^{2}\!\left\| \bm{e}_{\eta}^{(j+1)}\right\|_{ L_{\rm d}^{2} }^{2}
		\! + \frac{\tilde{L}_{\alpha_{0}}}{2}\!\left\| \bm{e}_{\theta}^{(j+1)} \right\|_{L_{\rm d}^{2}}^{2}
		\! + \frac{\tilde{L}_{\alpha_{0}}}{2}\!\left\| \bm{e}_{\theta}^{(j)} \right\|_{L_{\rm d}^{2}}^{2}
		\! + \frac{1}{2\tilde{L}_{\alpha_{0}}}\!\left\| \bm{\xi}_{4-6}^{(j+1)} \right\|_{L_{\rm d}^{2}}^{2}
		\! - \frac{\nu^{2}}{2}\!\left\| D\bm{e}_{\theta}^{(j+1)}\right\|^{2}. 
	\end{align*}
	Here, it holds from $\alpha_{0,k}^{(j)} \geq \delta_{0} \ (k=0,\ldots,K,\ j=0,\ldots,N)$ that $1 \leq \alpha_{0,k}^{(j)}/\delta_{0} \ (k=0,\ldots,K,\ j=0,\ldots,N)$. 
	Thus, we observe that 
	\begin{equation*}
	\left\| \bm{e}_{\theta}^{(j)} \right\|_{L_{\rm d}^{2}}^{2} 
		= \sum_{k=0}^{K}{}^{\prime\prime}\left| e_{\theta,k}^{(j)} \right|^{2} \Delta x 
		\leq \sum_{k=0}^{K}{}^{\prime\prime}\frac{\alpha_{0,k}^{(j)}}{\delta_{0}}\left| e_{\theta,k}^{(j)} \right|^{2} \Delta x 
		= \frac{1}{\delta_{0}}\left\| \sqrt{\bm{\alpha_{0}}^{(j)}}\bm{e}_{\theta}^{(j)}\right\|_{ L_{\rm d}^{2} }^{2} \quad (j=0,\ldots, N). 
	\end{equation*}
	Thus, we obtain \eqref{err_theta_est2}. 
\end{proof}

\begin{lem}\label{e_eta_theta_est}
Let $a := \max\{2(\kappa C_{1} /\nu )^{2} - (2\kappa\varepsilon + c), \tilde{L}_{\alpha_{0}}/\delta_{0}\}$.  
If $\Delta t$ satisfies \eqref{dt_ass}, 
then, we have 
\begin{align}
& \left\| \bm{e}_{\eta}^{(j)}\right\|_{ L_{\rm d}^{2} }^{2} 
+ \left\| \sqrt{\bm{\alpha_{0}}^{(j)}}\bm{e}_{\theta}^{(j)}\right\|_{ L_{\rm d}^{2} }^{2} 
+ \frac{\Delta t}{1 - a\Delta t}\nu^{2}\left\| D\bm{e}_{\theta}^{(j)}\right\|^{2} \notag\\
	\leq & \frac{3}{2}e^{3aT}\sum_{i=0}^{j-1}\left( \frac{1}{c}\left\| \bm{\xi}_{1-3}^{(j-i)} \right\|_{L_{\rm d}^{2}}^{2} + \frac{1}{\tilde{L}_{\alpha_{0}}}\left\| \bm{\xi}_{4-6}^{(j-i)} \right\|_{L_{\rm d}^{2}}^{2}\right)\Delta t \quad (j=1,\ldots,N). \label{e_eta_theta_est2}
\end{align} 
\end{lem}

\begin{proof}
	Let us fix $j=0,1,\ldots, N-1$. 
	Summing up the sides of \eqref{err_eta_est2} (Lemma \ref{e_eta_est}) and \eqref{err_theta_est2} (Lemma \ref{e_theta_est}), it yields from the definition of $a$ that 
	\begin{align}
	& (1 - a\Delta t)\left( \left\| \bm{e}_{\eta}^{(j+1)}\right\|_{ L_{\rm d}^{2} }^{2}
	+ \left\| \sqrt{\bm{\alpha_{0}}^{(j+1)}}\bm{e}_{\theta}^{(j+1)}\right\|_{ L_{\rm d}^{2} }^{2}\right) 
	+ \Delta t\nu^{2}\left\| D\bm{e}_{\theta}^{(j+1)}\right\|^{2} \nonumber\\
		\leq & \left\| \bm{e}_{\eta}^{(j)}\right\|_{ L_{\rm d}^{2} }^{2}
		\! + \! (1 \! + \! a\Delta t )\!\left\| \sqrt{\bm{\alpha_{0}}^{(j)}}\bm{e}_{\theta}^{(j)}\right\|_{ L_{\rm d}^{2} }^{2}
		\! + \! \Delta t\nu^{2}\!\left\| D\bm{e}_{\theta}^{(j)} \right\|^{2} 
		\! + \! \Delta t\!\left(\! \frac{1}{c}\!\left\| \bm{\xi}_{1-3}^{(j+1)} \right\|_{ L_{\rm d}^{2} }^{2} 
		\! + \! \frac{1}{\tilde{L}_{\alpha_{0}}}\!\left\| \bm{\xi}_{4-6}^{(j+1)} \right\|_{L_{\rm d}^{2}}^{2} \right). \label{total_est}
	\end{align}
	We see from the assumption $\Delta t < 1/(3a)$ that $1 -3a\Delta t > 0$. 
	Furthermore, we observe that 
	\begin{equation*}
	\frac{1}{1 - a\Delta t}
		< \frac{1 + a\Delta t}{1 - a\Delta t}
		< 1 + 3a\Delta t, \quad 
	\frac{1}{1 - a\Delta t} < \frac{3}{2}. 
	\end{equation*}
	Thus, dividing both sides of \eqref{total_est} by $1 - a\Delta t$, we have 
	\begin{align*}
	&  \left\| \bm{e}_{\eta}^{(j+1)}\right\|_{ L_{\rm d}^{2} }^{2}
	+ \left\| \sqrt{\bm{\alpha_{0}}^{(j+1)}}\bm{e}_{\theta}^{(j+1)}\right\|_{ L_{\rm d}^{2} }^{2}
	+ \frac{\nu^{2}\Delta t}{1 - a\Delta t}\left\| D\bm{e}_{\theta}^{(j+1)}\right\|^{2} \\
		\leq & (1 \! + \! 3a\Delta t )\!\left( \left\| \bm{e}_{\eta}^{(j)}\right\|_{ L_{\rm d}^{2} }^{2}
		\! + \! \left\| \sqrt{\bm{\alpha_{0}}^{(j)}}\bm{e}_{\theta}^{(j)}\right\|_{ L_{\rm d}^{2} }^{2} \right)
		\! + \! \frac{\nu^{2}\Delta t}{1 - a\Delta t}\left\| D\bm{e}_{\theta}^{(j)} \right\|^{2} \\
		& + \! \frac{3}{2}\Delta t\left(\frac{1}{c}\left\| \bm{\xi}_{1-3}^{(j+1)} \right\|_{ L_{\rm d}^{2} }^{2} 
		\! + \! \frac{1}{\tilde{L}_{\alpha_{0}}}\left\| \bm{\xi}_{4-6}^{(j+1)} \right\|_{L_{\rm d}^{2}}^{2}\right) \\ 
		\leq & (1 \! + \! 3a\Delta t )\!\left( \left\| \bm{e}_{\eta}^{(j)}\right\|_{ L_{\rm d}^{2} }^{2}
		\! + \! \left\| \sqrt{\bm{\alpha_{0}}^{(j)}}\bm{e}_{\theta}^{(j)}\right\|_{ L_{\rm d}^{2} }^{2} 
		\! + \! \frac{\nu^{2}\Delta t}{1 - a\Delta t}\left\| D\bm{e}_{\theta}^{(j)} \right\|^{2} \right) \\
		& + \! \frac{3}{2}\Delta t\left(\frac{1}{c}\left\| \bm{\xi}_{1-3}^{(j+1)} \right\|_{ L_{\rm d}^{2} }^{2} 
		\! + \! \frac{1}{\tilde{L}_{\alpha_{0}}}\left\| \bm{\xi}_{4-6}^{(j+1)} \right\|_{L_{\rm d}^{2}}^{2}\right). 
	\end{align*}
	As a result, we obtain 
	\begin{align*}
	&  \left\| \bm{e}_{\eta}^{(j)}\right\|_{ L_{\rm d}^{2} }^{2}
	+ \left\| \sqrt{\bm{\alpha_{0}}^{(j)}}\bm{e}_{\theta}^{(j)}\right\|_{ L_{\rm d}^{2} }^{2}
	+ \frac{\nu^{2}\Delta t}{1 - a\Delta t}\left\| D\bm{e}_{\theta}^{(j)}\right\|^{2} \\
		\leq & (1 \! + \! 3a\Delta t )\!\left( \left\| \bm{e}_{\eta}^{(j-1)}\right\|_{ L_{\rm d}^{2} }^{2}
		\! + \! \left\| \sqrt{\bm{\alpha_{0}}^{(j-1)}}\bm{e}_{\theta}^{(j-1)}\right\|_{ L_{\rm d}^{2} }^{2} 
		\! + \! \frac{\nu^{2}\Delta t}{1 - a\Delta t}\left\| D\bm{e}_{\theta}^{(j-1)} \right\|^{2} \right) \\
		& + \frac{3}{2}\Delta t\left(\frac{1}{c}\left\| \bm{\xi}_{1-3}^{(j)} \right\|_{ L_{\rm d}^{2} }^{2} 
		+ \frac{1}{\tilde{L}_{\alpha_{0}}}\left\| \bm{\xi}_{4-6}^{(j)} \right\|_{L_{\rm d}^{2}}^{2}\right) \\ 
		\leq & (1 \! + \! 3a\Delta t )^{2}\!\left( \left\| \bm{e}_{\eta}^{(j-2)}\right\|_{ L_{\rm d}^{2} }^{2}
		\! + \! \left\| \sqrt{\bm{\alpha_{0}}^{(j-2)}}\bm{e}_{\theta}^{(j-2)}\right\|_{ L_{\rm d}^{2} }^{2} 
		\! + \! \frac{\nu^{2}\Delta t}{1 - a\Delta t}\left\| D\bm{e}_{\theta}^{(j-2)} \right\|^{2} \right) \\
		& + \frac{3}{2}(1 + 3a\Delta t )\Delta t\!\left(\frac{1}{c}\!\left\| \bm{\xi}_{1-3}^{(j-1)} \right\|_{ L_{\rm d}^{2} }^{2} 
		+ \frac{1}{\tilde{L}_{\alpha_{0}}}\!\left\| \bm{\xi}_{4-6}^{(j-1)} \right\|_{L_{\rm d}^{2}}^{2}\right)
		\! + \frac{3}{2}\Delta t \! \left(\frac{1}{c}\!\left\| \bm{\xi}_{1-3}^{(j)} \right\|_{ L_{\rm d}^{2} }^{2} 
		+ \frac{1}{\tilde{L}_{\alpha_{0}}}\!\left\| \bm{\xi}_{4-6}^{(j)} \right\|_{L_{\rm d}^{2}}^{2}\right) \\ 
		\leq & \cdots \quad \text{(by repeating the above argument inductively)} \\
		\leq & (1 \! + \! 3a\Delta t )^{j}\!\left( \left\| \bm{e}_{\eta}^{(0)}\right\|_{ L_{\rm d}^{2} }^{2}
		\! + \! \left\| \sqrt{\bm{\alpha_{0}}^{(0)}}\bm{e}_{\theta}^{(0)}\right\|_{ L_{\rm d}^{2} }^{2} 
		\! + \! \frac{\nu^{2}\Delta t}{1 - a\Delta t}\left\| D\bm{e}_{\theta}^{(0)} \right\|^{2} \right) \\
		& + \frac{3}{2}\sum_{i=0}^{j-1}(1 + 3a\Delta t )^{i}\left(\frac{1}{c}\left\| \bm{\xi}_{1-3}^{(j-i)} \right\|_{ L_{\rm d}^{2} }^{2} 
		+ \frac{1}{\tilde{L}_{\alpha_{0}}}\left\| \bm{\xi}_{4-6}^{(j-i)} \right\|_{L_{\rm d}^{2}}^{2}\right)\Delta t \\
		= & \frac{3}{2}\sum_{i=0}^{j-1}(1 + 3a\Delta t )^{i}\left(\frac{1}{c}\left\| \bm{\xi}_{1-3}^{(j-i)} \right\|_{ L_{\rm d}^{2} }^{2} 
		+ \frac{1}{\tilde{L}_{\alpha_{0}}}\left\| \bm{\xi}_{4-6}^{(j-i)} \right\|_{L_{\rm d}^{2}}^{2}\right)\Delta t \quad (j=1,\ldots, N). 
	\end{align*}
	Here, it holds that $(1 + 3a\Delta t )^{i} \leq (1 + 3a\Delta t )^{N} \leq \exp(N \cdot 3a\Delta t) = e^{3aT} \ (i=0,\ldots, N)$. 
	Therefore, we obtain \eqref{e_eta_theta_est2}. 
\end{proof}

\begin{lem}\label{exist1}
	Let us assume (A1). 
	Then, for any given $(\bm{H}^{(j)},\bm{\Theta}^{(j)})$, there exists a unique soluiton $\bm{H}^{(j + 1)}$ satisfying \eqref{sc1} with \eqref{bc1}--\eqref{bc2}.
\end{lem}

\begin{proof}
	From now on, we give the matrix expression of \eqref{sc1}. 
	For the purpose, let us define the $(K+1) \times (K+1)$ matrices $D_{2}$ and $V$ by
	\begin{gather*}
		D_{2} := \frac{1}{(\Delta x)^{2}}\left(
		\begin{array}{ccccc} 
		-2 & 2 & & &  \\
		1 & -2 & 1 & &  \\
		& \ddots & \ddots & \ddots & \\
		& & 1 & -2 & 1 \\
		& & & 2 & -2 \\
		\end{array}
		\right) \in \mathbb{R}^{(K+1) \times (K+1)}, \\
		V \! := \! \diag\!\left(\! 
			\frac{ \gamma_{\varepsilon}\bigl( \delta_{k}^{+}\Theta_{0}^{(j)}\bigr) \! + \! \gamma_{\varepsilon}\bigl( \delta_{k}^{-}\Theta_{0}^{(j)}\bigr)  }{2}, 
			\frac{ \gamma_{\varepsilon}\bigl( \delta_{k}^{+}\Theta_{1}^{(j)}\bigr) \! + \! \gamma_{\varepsilon}\bigl( \delta_{k}^{-}\Theta_{1}^{(j)}\bigr)  }{2}, 
			\ldots ,
			\frac{ \gamma_{\varepsilon}\bigl( \delta_{k}^{+}\Theta_{K}^{(j)}\bigr) \! + \! \gamma_{\varepsilon}\bigl( \delta_{k}^{-}\Theta_{K}^{(j)}\bigr)  }{2} 
		\!\right) . 
	\end{gather*}
	Then, we obtain the matrix expression of \eqref{sc1} as follows: 
	\begin{equation*}
		\left\{I + \Delta t\left(- \kappa_{0}^{2}D_{2} + cI + \kappa V \right)\right\}\bm{H}^{(j + 1)} = \bm{H}^{(j)} + c\Delta t\bm{1}. 
	\end{equation*}
	where $I$ is the $(K+1)$-dimensional identity matrix and $\bm{1} := (1,1,\ldots,1)^{\top} \in \mathbb{R}^{K+1}$. 
	For simplicity, let us denote the coefficient matrix of $\bm{H}^{(j + 1)}$ by $A \! := \! I + \Delta t\left(- \kappa_{0}^{2}D_{2} + cI + \kappa V \right)$. 
	By direct calculation, we see that $A$ is a strictly diagonally dominant matrix. 
	% Since a strictly diagonally dominant matrix is non-singular, 
	Namely, $A$ is nonsingular. 
\end{proof}

\begin{lem}\label{exist2}
	Let us assume (A1). 
	Then, for any given $(\bm{H}^{(j+1)},\bm{\Theta}^{(j)})$, if $\Delta t$ satisfies 
	\begin{equation}
	    \Delta t < \frac{\nu^{2}\varepsilon^{2}}{4\kappa^{2}\delta_{0}\left(\delta_{0} + \dfrac{1}{2}\right)^{2}}(\Delta x)^{2}, \label{exist_cond2}
	\end{equation}
	then there exists a unique soluiton $\bm{\Theta}^{(j + 1)}$ satisfying \eqref{sc2} with \eqref{bc1}--\eqref{bc2}. 
\end{lem}

\begin{proof}
	For any given $\bm{\Theta}^{(j)} = \{\Theta_{k}^{(j)}\}_{k=-1}^{K+1} \in \mathbb{R}^{K+3}$, we define the mapping $\Psi$: $\{\Theta_{k}\}_{k=-1}^{K+1} \mapsto \{\tilde{\Theta}_{k}\}_{k=-1}^{K+1}$ by 
	\begin{gather}
		\alpha_{0, k}^{(j+1)} \frac{ \tilde{\Theta}_{k} - \Theta_{k}^{(j)} }{\Delta t} 
		= \frac{\kappa}{2}\left\{ \delta_{k}^{+}\left( \alpha\!\left( H_{k}^{(j+1)} \right)\gamma_{\varepsilon}' (\delta_{k}^{-}\Theta_{k}) \right) + \delta_{k}^{-}\left( \alpha\!\left( H_{k}^{(j+1)} \right) \gamma_{\varepsilon}' (\delta_{k}^{+}\Theta_{k}) \right) \right\} + \nu^{2}\delta_{k}^{\langle 2 \rangle}\tilde{\Theta}_{k} \quad (k=0, \ldots, K), \label{map}\\
		\delta_{k}^{\langle 1 \rangle}\tilde{\Theta}_{k} = 0 \quad (k=0, K), \label{map_bc} 
	\end{gather}
	where 
	\begin{equation}
		\Theta_{-1} = \Theta_{1}, \quad \Theta_{K + 1} = \Theta_{K - 1}. \label{map_bc0} 
	\end{equation}
	Firstly, we show that the mapping $\Psi$ is well-defined. 
	From \eqref{map_bc}, $\tilde{\Theta}_{-1}$ and $\tilde{\Theta}_{K+1}$ can be explicitly written as 
	\begin{equation}
		\tilde{\Theta}_{-1} = \tilde{\Theta}_{1}, \quad \tilde{\Theta}_{K + 1} = \tilde{\Theta}_{K - 1}. \label{map_bc2}
	\end{equation}
	Then, the coefficient matrix of $\bm{\tilde{\Theta}}$ by $B := W - \Delta t\nu^{2}D_{2}$, where 
	\begin{equation*}
		W := \diag(\alpha_{0, 0}^{(j+1)}, \alpha_{0, 1}^{(j+1)}, \dots, \alpha_{0, K}^{(j+1)}). 
	\end{equation*}
	By direct calculation, we see that $B$ is a strictly diagonally dominant matrix. 
	Namely, $B$ is nonsingular.
	Hence, the mapping $\Psi$ is well-defined. 
	
	Next, we prove the existence and uniqueness of the solution of the proposed scheme by the fixed-point theorem for a contraction mapping. 
	From \eqref{map_bc0} and \eqref{map_bc2}, it is sufficient to show the existence of a $K+1$-vector $\bm{\Theta}$ that satisfies $\tilde{\Theta}_{k} = \Theta_{k} \ (k=0, \ldots, K)$. 
	Here, 
	% we define the mapping $\Phi$: $\mathbb{R}^{K+1} \to \mathbb{R}^{K+1}$ by 
	% \begin{equation*}
	% 	\Phi( \bm{\Theta} ) := \{ \tilde{\Theta}_{k} \}_{k=0}^{K} = \{ \Psi_{k}( \bm{\Theta} ) \}_{k=0}^{K} \quad \mbox{for\ all\ } \bm{\Theta} \in \mathbb{R}^{K+1}, 
	% \end{equation*}
	% where $\Psi_{k}( \bm{\Theta})$ is the $k$-th element of $\Psi( \bm{\Theta} )$. 
	% Also, 
	let 
	\begin{equation*}
		X := \{\bm{f} \in \mathbb{R}^{K+3}; \|\bm{f} \|_{L_{\rm d}^{2}} \leq \sqrt{2} M \}, 
	\end{equation*}
	where $M := \|\bm{\Theta}^{(j)}\|_{L_{\rm d}^{2}}$. 
	We show that the mapping $\Psi$ is a contraction mapping on $X$. 
	If $\Psi$ is a contraction mapping, $\Psi$ has a unique fixed-point $\bm{\Theta}^{\ast}$ in the closed ball $X$ from the fixed-point theorem for a contraction mapping.
	This $\bm{\Theta}^{\ast}$ is the solution $\bm{\Theta}^{(j+1)}$ to the scheme \eqref{sc2} with \eqref{bc1}--\eqref{bc2}. 
	Firstly, we show $\Psi(X) \subset X$. 
	For any fixed $\bm{\Theta} \in X$, we have 
	\begin{align}
		& \sum _ {k=0}^{K}{}^{\prime \prime} \alpha_{0, k}^{(j+1)} \frac{ \tilde{\Theta}_{k} - \Theta_{k}^{(j)} }{\Delta t}\tilde{\Theta}_ {k}\Delta x \notag\\
		= & \frac{\kappa}{2}\sum _ {k=0}^{K}{}^{\prime \prime}\!\left\{ 
			\delta_{k}^{+}\left( \alpha\!\left( H_{k}^{(j+1)} \right)\gamma_{\varepsilon}' (\delta_{k}^{-}\Theta_{k}) \right) 
			+ \delta_{k}^{-}\left( \alpha\!\left( H_{k}^{(j+1)} \right) \gamma_{\varepsilon}' (\delta_{k}^{+}\Theta_{k}) \right) 
		\right\} \tilde{\Theta}_ {k}\Delta x 
		+ \nu^{2}\sum _ {k=0}^{K}{}^{\prime \prime}\left(\delta_{k}^{\langle 2 \rangle}\tilde{\Theta}_{k}\right) \tilde{\Theta}_ {k}\Delta x. \label{tool1}
	\end{align}
	Firstly, by using \eqref{tool_ineq}, we evaluate the left-hand side of \eqref{tool1} as follows: 
	\begin{equation*}
		\frac{1}{\Delta t}\sum _ {k=0}^{K}{}^{\prime \prime}\alpha_{0, k}^{(j+1)}\!\left(\tilde{\Theta}_ {k} \! - \! \Theta_{k}^{(j)}\right)\!\tilde{\Theta}_ {k}\Delta x
		\geq \frac{1}{2\Delta t}\sum _ {k=0}^{K}{}^{\prime \prime}\alpha_{0, k}^{(j+1)}\!\left\{\tilde{\Theta}_{k}^{2} \! - \! \left(\Theta_{k}^{(j)}\right)^{\! 2} \right\}\Delta x 
		\geq \frac{1}{2\delta_{0}\Delta t}\left( \left\|\bm{\tilde{\Theta}}\right\|_{L_{\rm d}^2}^{2} - \left\|\bm{\Theta}^{(j)}\right\|_{L_{\rm d}^2}^{2} \right).
	\end{equation*}
	Next, we consider the right-hand side of \eqref{tool1}. 
	It holds from the summation by parts formula (Lemma \ref{sbp5}) and \eqref{map_bc2} that 
	\begin{equation*}
		\nu^{2}\sum _ {k=0}^{K}{}^{\prime \prime}\left(\delta_{k}^{\langle 2 \rangle}\tilde{\Theta}_{k}\right) \tilde{\Theta}_ {k}\Delta x
		= -\nu^{2}\sum _ {k=0}^{K-1}\left|\delta_{k}^{+}\tilde{\Theta}_{k}\right|^{2} \Delta x 
		+ \nu^{2}\left[\left(\delta_{k}^{\langle 1 \rangle}\tilde{\Theta}_{k}\right) \tilde{\Theta}_ {k}\right]_{0}^{K} 
		= -\nu^{2}\left\| D\bm{\tilde{\Theta}} \right\|^{2}.
	\end{equation*}
	Similarly, we see from the summation by parts formula (Lemma \ref{sbp2}), \eqref{bc1}, and \eqref{map_bc0} that 
	\begin{align*}
		& \frac{\kappa}{2}\sum _ {k=0}^{K}{}^{\prime \prime}\left\{ \delta_{k}^{+}\left( \alpha\!\left( H_{k}^{(j+1)} \right)\gamma_{\varepsilon}' (\delta_{k}^{-}\Theta_{k}) \right) + \delta_{k}^{-}\left( \alpha\!\left( H_{k}^{(j+1)} \right) \gamma_{\varepsilon}' (\delta_{k}^{+}\Theta_{k}) \right) \right\} \tilde{\Theta}_ {k}\Delta x \\
		= & - \frac{\kappa}{2}\sum _ {k=1}^{K}\alpha\left( H_{k}^{(j+1)} \right)\gamma_{\varepsilon}' (\delta_{k}^{-}\Theta_{k}) \delta_{k}^{-}\tilde{\Theta}_ {k}\Delta x 
		- \frac{\kappa}{2}\sum _ {k=0}^{K-1}\alpha\left( H_{k}^{(j+1)} \right)\gamma_{\varepsilon}' (\delta_{k}^{+}\Theta_{k}) \delta_{k}^{+}\tilde{\Theta}_ {k}\Delta x \\
		& + \frac{\kappa}{2}\left[ 
			\left\{ 
				\mu_{k}^{+}\left( \alpha\left( H_{k}^{(j+1)} \right)\gamma_{\varepsilon}' (\delta_{k}^{-}\Theta_{k}) \right) 
				+ \mu_{k}^{-}\left( \alpha\left( H_{k}^{(j+1)} \right)\gamma_{\varepsilon}' (\delta_{k}^{+}\Theta_{k}) \right) 
			\right\} \tilde{\Theta}_ {k} 
		\right]_{0}^{K} \\
		= & - \frac{\kappa}{2}\sum _ {k=0}^{K-1}\alpha\left( H_{k+1}^{(j+1)} \right)\gamma_{\varepsilon}' (\delta_{k}^{-}\Theta_{k+1}) \delta_{k}^{-}\tilde{\Theta}_ {k+1}\Delta x 
		- \frac{\kappa}{2}\sum _ {k=0}^{K-1}\alpha\left( H_{k}^{(j+1)} \right)\gamma_{\varepsilon}' (\delta_{k}^{+}\Theta_{k}) \delta_{k}^{+}\tilde{\Theta}_ {k}\Delta x \\
		= & - \kappa\sum _ {k=0}^{K-1}\frac{\alpha\left( H_{k+1}^{(j+1)} \right) + \alpha\left( H_{k}^{(j+1)} \right)}{2}\gamma_{\varepsilon}' (\delta_{k}^{+}\Theta_{k}) \delta_{k}^{+}\tilde{\Theta}_ {k}\Delta x. 
	\end{align*}
	As a result, we observe from the Young inequality that 
	\begin{align*}
		\frac{1}{2\delta_{0}\Delta t}\left( \left\|\bm{\tilde{\Theta}}\right\|_{L_{\rm d}^2}^{2} - \left\|\bm{\Theta}^{(j)}\right\|_{L_{\rm d}^2}^{2} \right) 
		\leq & -\nu^{2}\left\| D\bm{\tilde{\Theta}} \right\|^{2} 
		- \kappa\sum _ {k=0}^{K-1}\frac{\alpha\left( H_{k+1}^{(j+1)} \right) + \alpha\left( H_{k}^{(j+1)} \right)}{2}\gamma_{\varepsilon}' (\delta_{k}^{+}\Theta_{k}) \delta_{k}^{+}\tilde{\Theta}_ {k}\Delta x \\[-1pt]
		\leq & \frac{\kappa^{2}}{4\nu^{2}}\sum _ {k=0}^{K-1}\left|\frac{\alpha\left( H_{k+1}^{(j+1)} \right) + \alpha\left( H_{k}^{(j+1)} \right)}{2}\right|^{2}\left|\gamma_{\varepsilon}' (\delta_{k}^{+}\Theta_{k})\right|^{2} \Delta x \\[-1pt]
		= & \frac{\kappa^{2}}{4\nu^{2}}\sum _ {k=0}^{K-1}\left|\frac{\alpha\left( H_{k+1}^{(j+1)} \right) + \alpha\left( H_{k}^{(j+1)} \right)}{2}\right|^{2}\frac{\left|\delta_{k}^{+}\Theta_{k}\right|^{2}}{\left|\gamma_{\varepsilon} (\delta_{k}^{+}\Theta_{k})\right|^{2}} \Delta x.
	\end{align*}
	Here, we see from Main Theorem 1 (O)(a) that 
	\begin{equation}
		\left|\frac{\alpha\!\left( H_{k+1}^{(j+1)} \right) \! + \! \alpha\!\left( H_{k}^{(j+1)} \right)}{2}\right|^{2}
		\! \leq \! \frac{\left|\alpha\!\left( H_{k+1}^{(j+1)} \right)\right|^{2} \! + \! \left|\alpha\!\left( H_{k}^{(j+1)} \right)\right|^{2} }{2} 
		\! \leq \! \left(\! \delta_{0} + \frac{1}{2} \!\right)^{2} \quad (k = 0, \dots, K). \label{al_bd}
	\end{equation}
	Using \eqref{al_bd} and the following inequality: 
	\begin{equation*}
		0 \leq \frac{1}{\gamma_{\varepsilon}(u)} = \frac{1}{\sqrt{\varepsilon^{2} + u^{2}}} \leq \frac{1}{\varepsilon}, 
	\end{equation*}
	we have
	\begin{align*}
		\frac{1}{2\delta_{0}\Delta t}\left( \left\|\bm{\tilde{\Theta}}\right\|_{L_{\rm d}^2}^{2} - \left\|\bm{\Theta}^{(j)}\right\|_{L_{\rm d}^2}^{2} \right) 
		\leq \frac{\kappa^{2}}{4\nu^{2}\varepsilon^{2}}\left(\delta_{0} + \frac{1}{2}\right)^{2}\sum _ {k=0}^{K-1}\left|\delta_{k}^{+}\Theta_{k}\right|^{2}\Delta x 
		= & \frac{\kappa^{2}}{4\nu^{2}\varepsilon^{2}}\left(\delta_{0} + \frac{1}{2}\right)^{2}\sum _ {k=0}^{K-1}\left|\frac{\Theta_{k + 1} - \Theta_{k}}{\Delta x}\right|^{2}\Delta x \\[-1pt]
		\leq & \frac{\kappa^{2}}{2\nu^{2}\varepsilon^{2}}\left(\delta_{0} + \frac{1}{2}\right)^{2}\sum _ {k=0}^{K-1}\frac{\left|\Theta_{k + 1}\right|^{2} + \left|\Theta_{k}\right|^{2}}{(\Delta x)^{2}}\Delta x \\[-1pt]
		= & \frac{1}{(\Delta x)^{2}}\frac{\kappa^{2}}{\nu^{2}\varepsilon^{2}}\left(\delta_{0} + \frac{1}{2}\right)^{2}\sum _ {k=0}^{K}{}^{\prime \prime}\left|\Theta_{k}\right|^{2}\Delta x \\[-1pt]
		= & \frac{1}{(\Delta x)^{2}}\frac{\kappa^{2}}{\nu^{2}\varepsilon^{2}}\left(\delta_{0} + \frac{1}{2}\right)^{2}\left\|\bm{\Theta}\right\|_{L_{\rm d}^2}^{2}.
	\end{align*}
	Therefore, from the assumption \eqref{exist_cond}, we conclude that   
	\begin{equation*}
		\left\|\bm{\tilde{\Theta}}\right\|_{L_{\rm d}^2}^{2} 
		\leq \left\|\bm{\Theta}^{(j)}\right\|_{L_{\rm d}^2}^{2} 
		+ \frac{2\Delta t}{(\Delta x)^{2}}\frac{\kappa^{2}\delta_{0}}{\nu^{2}\varepsilon^{2}}\left(\delta_{0} + \frac{1}{2}\right)^{2}\left\|\bm{\Theta}\right\|_{L_{\rm d}^2}^{2} 
		\leq \left\{
			1 + \frac{4\Delta t}{(\Delta x)^{2}}\frac{\kappa^{2}\delta_{0}}{\nu^{2}\varepsilon^{2}}\left(\delta_{0} + \frac{1}{2}\right)^{2}
		\right\}M^{2} 
		\leq 2M^{2}. 
	\end{equation*} 
	Namely, $\Psi(\bm{\Theta}) = \bm{\tilde{\Theta}} \in X$. 
	
	Next, we show that $\Psi$ is a contraction mapping. 
	For any fixed $\bm{\Theta}_{1}, \bm{\Theta}_{2} \in X$ which satisfy \eqref{map_bc0}, 
	the vector $\{\tilde{\Theta}_{i,k}\}_{k=-1}^{K+1} = \{\Psi_{k}(\bm{\Theta}_{i})\}_{k=-1}^{K+1}$ satisfies \eqref{map}--\eqref{map_bc} ($i=1,2$) from the definition of $\Psi$, 
	where $\Psi_{k}( \bm{\Theta})$ is the $k$-th element of $\Psi( \bm{\Theta} )$. 
	Subtracting these relations, we obtain 
	\begin{gather}
		\begin{split}
			\alpha_{0, k}^{(j+1)} \frac{ \tilde{\Theta}_{1, k} - \tilde{\Theta}_{2, k} }{\Delta t} 
			= & \frac{\kappa}{2} \delta_{k}^{+}\left\{
					\alpha\!\left( H_{k}^{(j+1)} \right)\left( \gamma_{\varepsilon}' (\delta_{k}^{-}\Theta_{1,k}) - \gamma_{\varepsilon}' (\delta_{k}^{-}\Theta_{2,k}) \right) 
			\right\} \\
			& + \frac{\kappa}{2}\delta_{k}^{-}\left\{
					\alpha\!\left( H_{k}^{(j+1)} \right)\left( \gamma_{\varepsilon}' (\delta_{k}^{+}\Theta_{1,k}) - \gamma_{\varepsilon}' (\delta_{k}^{+}\Theta_{2,k}) \right) 
			\right\} + \nu^{2}\delta_{k}^{\langle 2 \rangle}\left(\tilde{\Theta}_{1, k} - \tilde{\Theta}_{2, k}\right) \quad (k=0, \ldots, K), 
		\end{split} \label{tool2}\\
		\delta_{k}^{\langle 1 \rangle}\left(\tilde{\Theta}_{1, k} - \tilde{\Theta}_{2, k}\right) = 0 \quad (k=0, K). \label{tool3}
	\end{gather}
	Firstly, we see from \eqref{tool2} that 
	\begin{align*}
		\frac{1}{\delta_{0}\Delta t}\left\|\bm{\tilde{\Theta}}_{1} - \bm{\tilde{\Theta}}_{2}\right\|_{L_{\rm d}^{2}}^{2}
		\leq & \frac{1}{\Delta t}\sum _ {k=0}^{K}{}^{\prime \prime}\alpha_{0, k}^{(j+1)} \left| \tilde{\Theta}_{1, k} - \tilde{\Theta}_{2, k} \right|^{2} \Delta x \\
		= & \frac{\kappa}{2}\sum _ {k=0}^{K}{}^{\prime \prime} \delta_{k}^{+}\!\left\{\!
			\alpha\!\left( H_{k}^{(j+1)} \right)\!\left( \gamma_{\varepsilon}' (\delta_{k}^{-}\Theta_{1,k}) - \gamma_{\varepsilon}' (\delta_{k}^{-}\Theta_{2,k}) \right) 
		\!\right\}\!\left(\! \tilde{\Theta}_{1, k} \! - \! \tilde{\Theta}_{2, k} \!\right)\!\Delta x \\
		& + \frac{\kappa}{2}\sum _ {k=0}^{K}{}^{\prime \prime} \delta_{k}^{-}\!\left\{\!
				\alpha\!\left( H_{k}^{(j+1)} \right)\!\left( \gamma_{\varepsilon}' (\delta_{k}^{+}\Theta_{1,k}) - \gamma_{\varepsilon}' (\delta_{k}^{+}\Theta_{2,k}) \right) 
		\!\right\}\!\left(\! \tilde{\Theta}_{1, k} \! - \! \tilde{\Theta}_{2, k} \!\right)\!\Delta x \\
		& + \nu^{2}\sum _ {k=0}^{K}{}^{\prime \prime}\left\{\delta_{k}^{\langle 2 \rangle}\left(\tilde{\Theta}_{1, k} - \tilde{\Theta}_{2, k}\right)\right\}\left( \tilde{\Theta}_{1, k} - \tilde{\Theta}_{2, k} \right)\Delta x.
	\end{align*}
	Next, by the summation by parts formula (Lemma \ref{sbp5}) and \eqref{tool3}, we have  
	\begin{equation*}
		\nu^{2}\sum _ {k=0}^{K}{}^{\prime \prime}\left\{\!\delta_{k}^{\langle 2 \rangle}\!\!\left(\tilde{\Theta}_{1,k} - \tilde{\Theta}_ {2,k}\right)\!\right\}\!\left(\tilde{\Theta}_ {1,k} - \tilde{\Theta}_ {2,k}\right)\Delta x 
		= -\nu^{2}\left\| D\left(\bm{\tilde{\Theta}}_{1} - \bm{\tilde{\Theta}}_{2}\right)\right\|^{2}.
	\end{equation*}
	% \begin{align*}
	% 	\nu^{2}\sum _ {k=0}^{K}{}^{\prime \prime}\left\{\!\delta_{k}^{\langle 2 \rangle}\!\!\left(\tilde{\Theta}_{1,k} - \tilde{\Theta}_ {2,k}\right)\!\right\}\!\left(\tilde{\Theta}_ {1,k} - \tilde{\Theta}_ {2,k}\right)\Delta x 
	% 		= & -\nu^{2}\sum _ {k=0}^{K-1}\left|\delta_{k}^{+}\left(\tilde{\Theta}_{1,k} - \tilde{\Theta}_ {2,k}\right)\right|^{2}\Delta x \\ 
	% 		& + \nu^{2}\left[\left\{\!\delta_{k}^{\langle 1 \rangle}\!\!\left(\tilde{\Theta}_{1,k} - \tilde{\Theta}_ {2,k}\right)\!\right\}\!\left(\tilde{\Theta}_ {1,k} - \tilde{\Theta}_ {2,k}\right)\right]_{0}^{K} \\
	% 		= & -\nu^{2}\left\| D\left(\bm{\tilde{\Theta}}_{1} - \bm{\tilde{\Theta}}_{2}\right)\right\|^{2}.
	% \end{align*}
	Similarly, we observe from the summation by parts formul (Lemma \ref{sbp2}), \eqref{bc1}, and \eqref{bc2} that 
	\begin{align*}
		& \frac{\kappa}{2}\sum _ {k=0}^{K}{}^{\prime \prime} \delta_{k}^{+}\left\{
			\alpha\!\left( H_{k}^{(j+1)} \right)\left( \gamma_{\varepsilon}' (\delta_{k}^{-}\Theta_{1,k}) - \gamma_{\varepsilon}' (\delta_{k}^{-}\Theta_{2,k}) \right) 
		\right\}\left( \tilde{\Theta}_{1, k} - \tilde{\Theta}_{2, k} \right)\Delta x \\[-1pt]
		& + \frac{\kappa}{2}\sum _ {k=0}^{K}{}^{\prime \prime} \delta_{k}^{-}\left\{
				\alpha\!\left( H_{k}^{(j+1)} \right)\left( \gamma_{\varepsilon}' (\delta_{k}^{+}\Theta_{1,k}) - \gamma_{\varepsilon}' (\delta_{k}^{+}\Theta_{2,k}) \right) 
		\right\}\left( \tilde{\Theta}_{1, k} - \tilde{\Theta}_{2, k} \right)\Delta x \\[-1pt]
		= & - \frac{\kappa}{2}\sum _ {k=1}^{K}\alpha\left( H_{k}^{(j+1)} \right)
			\left( \gamma_{\varepsilon}' (\delta_{k}^{-}\Theta_{1,k}) - \gamma_{\varepsilon}' (\delta_{k}^{-}\Theta_{2,k}) \right)  
			\delta_{k}^{-}\left( \tilde{\Theta}_{1, k} - \tilde{\Theta}_{2, k} \right)\Delta x \\[-1pt]
		& - \frac{\kappa}{2}\sum _ {k=0}^{K-1}\alpha\left( H_{k}^{(j+1)} \right)
			\left( \gamma_{\varepsilon}' (\delta_{k}^{+}\Theta_{1,k}) - \gamma_{\varepsilon}' (\delta_{k}^{+}\Theta_{2,k}) \right)  
			\delta_{k}^{+}\left( \tilde{\Theta}_{1, k} - \tilde{\Theta}_{2, k} \right)\Delta x \\[-1pt]
		& + \frac{\kappa}{2}\left[ 
			\left\{ 
				\mu_{k}^{+}\left(
					\alpha\!\left( H_{k}^{(j+1)} \right)\left( \gamma_{\varepsilon}' (\delta_{k}^{-}\Theta_{1,k}) - \gamma_{\varepsilon}' (\delta_{k}^{-}\Theta_{2,k}) \right) 
				\right)
		\right.\right. \\[-1pt] 
		& \qquad \left.\left. 
			+ \mu_{k}^{-}\left(
					\alpha\!\left( H_{k}^{(j+1)} \right)\left( \gamma_{\varepsilon}' (\delta_{k}^{+}\Theta_{1,k}) - \gamma_{\varepsilon}' (\delta_{k}^{+}\Theta_{2,k}) \right) 
				\right)
			\right\} 
			\left( \tilde{\Theta}_{1, k} - \tilde{\Theta}_{2, k} \right) 
		\right]_{0}^{K} \\[-1pt]
		= & - \frac{\kappa}{2}\sum _ {k=0}^{K - 1}\alpha\left( H_{k + 1}^{(j+1)} \right)
			\left( \gamma_{\varepsilon}' (\delta_{k}^{-}\Theta_{1,k + 1}) - \gamma_{\varepsilon}' (\delta_{k}^{-}\Theta_{2,k + 1}) \right)  
			\delta_{k}^{-}\left( \tilde{\Theta}_{1, k + 1} - \tilde{\Theta}_{2, k + 1} \right)\Delta x \\[-1pt]
		& - \frac{\kappa}{2}\sum _ {k=0}^{K-1}\alpha\left( H_{k}^{(j+1)} \right)
			\left( \gamma_{\varepsilon}' (\delta_{k}^{+}\Theta_{1,k}) - \gamma_{\varepsilon}' (\delta_{k}^{+}\Theta_{2,k}) \right)  
			\delta_{k}^{+}\left( \tilde{\Theta}_{1, k} - \tilde{\Theta}_{2, k} \right)\Delta x \\[-1pt]
		= & - \kappa\sum _ {k=0}^{K-1}\frac{\alpha\left( H_{k+1}^{(j+1)} \right) + \alpha\left( H_{k}^{(j+1)} \right)}{2}
		\left( \gamma_{\varepsilon}' (\delta_{k}^{+}\Theta_{1,k}) - \gamma_{\varepsilon}' (\delta_{k}^{+}\Theta_{2,k}) \right)  
			\delta_{k}^{+}\left( \tilde{\Theta}_{1, k} - \tilde{\Theta}_{2, k} \right)
		\Delta x. 
	\end{align*}
	Thus, using the Young inequality and Lemma \ref{lem:4.2}, we obtain 
	\begin{align*}
		& \frac{1}{\delta_{0}\Delta t}\left\|\bm{\tilde{\Theta}}_{1} - \bm{\tilde{\Theta}}_{2}\right\|_{L_{\rm d}^{2}}^{2} \\
		\leq & -\nu^{2}\left\| D\!\left(\! \bm{\tilde{\Theta}}_{1} - \bm{\tilde{\Theta}}_{2} \!\right)\right\|^{2} 
		- \kappa\sum _ {k=0}^{K-1}\frac{\alpha\!\left(\! H_{k+1}^{(j+1)} \!\right) \! + \! \alpha\!\left(\! H_{k}^{(j+1)} \!\right)}{2}
			\left( \gamma_{\varepsilon}' (\delta_{k}^{+}\Theta_{1,k}) - \gamma_{\varepsilon}' (\delta_{k}^{+}\Theta_{2,k}) \right)  
			\delta_{k}^{+}\!\left( \tilde{\Theta}_{1, k} - \tilde{\Theta}_{2, k} \right)
		\Delta x \\
		\leq & -\nu^{2}\left\| D\left(\bm{\tilde{\Theta}}_{1} - \bm{\tilde{\Theta}}_{2}\right)\right\|^{2} \\
		& + \kappa\sum _ {k=0}^{K-1}\left\{
			\frac{\kappa}{4\nu^{2}}
			\frac{ \left|\alpha\!\left( H_{k+1}^{(j+1)} \right) + \alpha\!\left( H_{k}^{(j+1)} \right) \right|^{2}}{4}
			\left| \gamma_{\varepsilon}' (\delta_{k}^{+}\Theta_{1,k}) - \gamma_{\varepsilon}' (\delta_{k}^{+}\Theta_{2,k}) \right|^{2} + \frac{\nu^{2}}{\kappa}
			\left| \delta_{k}^{+}\left( \tilde{\Theta}_{1, k} - \tilde{\Theta}_{2, k} \right) \right|^{2}
		\right\}\Delta x \\
		\leq & \frac{\kappa^{2}}{4\nu^{2}}\sum _ {k=0}^{K-1}
			\frac{ \left|\alpha\!\left( H_{k+1}^{(j+1)} \right) \right|^{2} + \left|\alpha\!\left( H_{k}^{(j+1)} \right) \right|^{2}}{2}
			\left| \frac{d\gamma_{\varepsilon}''}{d( \delta_{k}^{+}\Theta_{1,k} , \delta_{k}^{+}\Theta_{2,k} )} \right|^{2} 
			\left| \delta_{k}^{+}\left( \Theta_{1,k} - \Theta_{2,k} \right) \right|^{2}
		\Delta x. 
	\end{align*}
	Meanwhile, we see from $|uv| \leq \sqrt{\varepsilon^{2} + u^{2}}\sqrt{\varepsilon^{2} + v^{2}}$ that 
	\begin{align*}
		0 \leq \frac{d\gamma_{\varepsilon}''}{d(u,v)} 
		= & \frac{\varepsilon^{2} + \sqrt{\varepsilon^{2} + u^{2}}\sqrt{\varepsilon^{2} + v^{2}} - uv }{\sqrt{\varepsilon^{2} + u^{2}}\sqrt{\varepsilon^{2} + v^{2}}\left(\sqrt{\varepsilon^{2} + u^{2}} + \sqrt{\varepsilon^{2} + v^{2}}\right)} \\ 
		\leq & \frac{\varepsilon^{2} + 2\sqrt{\varepsilon^{2} + u^{2}}\sqrt{\varepsilon^{2} + v^{2}} }{\sqrt{\varepsilon^{2} + u^{2}}\sqrt{\varepsilon^{2} + v^{2}}\left(\sqrt{\varepsilon^{2} + u^{2}} + \sqrt{\varepsilon^{2} + v^{2}}\right)} \\
		\leq & \frac{\varepsilon^{2} }{\sqrt{\varepsilon^{2} + u^{2}}\sqrt{\varepsilon^{2} + v^{2}}\left(\sqrt{\varepsilon^{2} + u^{2}} + \sqrt{\varepsilon^{2} + v^{2}}\right)} 
		+ \frac{2}{\sqrt{\varepsilon^{2} + u^{2}} + \sqrt{\varepsilon^{2} + v^{2}}} \\
		\leq & \frac{\varepsilon^{2}}{\varepsilon \cdot \varepsilon (\varepsilon + \varepsilon)} + \frac{2}{\varepsilon + \varepsilon} = \frac{3}{2\varepsilon},
	\end{align*}
	Furthermore, by \eqref{al_bd} and the above inequality, we conclude that 
	\begin{align*}
		\frac{1}{\delta_{0}\Delta t}\left\|\bm{\tilde{\Theta}}_{1} - \bm{\tilde{\Theta}}_{2}\right\|_{L_{\rm d}^{2}}^{2} 
		\leq & \frac{9\kappa^{2}}{16\nu^{2}\varepsilon^{2}}\left(\delta_{0} + \frac{1}{2}\right)^{2}
		\sum _ {k=0}^{K-1}
			\left| \delta_{k}^{+}\left( \Theta_{1,k} - \Theta_{2,k} \right) \right|^{2}
		\Delta x \\
		= & \frac{9\kappa^{2}}{16\nu^{2}\varepsilon^{2}}\left(\delta_{0} + \frac{1}{2}\right)^{2}
		\sum _ {k=0}^{K-1}
			\frac{\left| \left( \Theta_{1,k+1} - \Theta_{2,k+1} \right) - \left( \Theta_{1,k} - \Theta_{2,k} \right) \right|^{2}}{(\Delta x)^{2}}
		\Delta x \\
		\leq & \frac{9}{8(\Delta x)^{2}}\frac{\kappa^{2}}{\nu^{2}\varepsilon^{2}}\left(\delta_{0} + \frac{1}{2}\right)^{2}
		\sum _ {k=0}^{K-1}
			\left( \left|\Theta_{1,k+1} - \Theta_{2,k+1}\right|^{2} + \left|\Theta_{1,k} - \Theta_{2,k}\right|^{2} \right) 
		\Delta x \\
		= & \frac{9}{4(\Delta x)^{2}}\frac{\kappa^{2}}{\nu^{2}\varepsilon^{2}}\left(\delta_{0} + \frac{1}{2}\right)^{2}
		\left\|\bm{\tilde{\Theta}}_{1} - \bm{\tilde{\Theta}}_{2}\right\|_{L_{\rm d}^{2}}^{2}. 
	\end{align*}
	Namely, we obtain 
	\begin{equation*}
		\left\|\bm{\tilde{\Theta}}_{1} - \bm{\tilde{\Theta}}_{2}\right\|_{L_{\rm d}^{2}}^{2} 
		\leq \frac{9\Delta t}{4(\Delta x)^{2}}\frac{\kappa^{2}\delta_{0}}{\nu^{2}\varepsilon^{2}}\left(\delta_{0} + \frac{1}{2}\right)^{2}
		\left\|\bm{\tilde{\Theta}}_{1} - \bm{\tilde{\Theta}}_{2}\right\|_{L_{\rm d}^{2}}^{2}. 
	\end{equation*}
	Since 
	\begin{equation*}
		0 \leq \frac{9\Delta t}{4(\Delta x)^{2}}\frac{\kappa^{2}\delta_{0}}{\nu^{2}\varepsilon^{2}}\left(\delta_{0} + \frac{1}{2}\right)^{2} < \frac{4\Delta t}{(\Delta x)^{2}}\frac{\kappa^{2}\delta_{0}}{\nu^{2}\varepsilon^{2}}\left(\delta_{0} + \frac{1}{2}\right)^{2} <1
	\end{equation*}
	from the assumption \eqref{exist_cond2}, the mapping $\Phi$ is contraction into $X$. 
	This completes the proof.  
\end{proof}

%%%%%%%%%%%%%%%%%%%%%%%%%%%%%%%%%%%%%%%%%%%%%%%%%%%%%%%%%%%%%%%%%%%%%%%%%%%%%%%%%%%%%%%%%%%%%%%%%%%%
%\bibliography{sinst41}

\begin{thebibliography}{10}

\bibitem{MR4218112}
Antil, H.; Kubota, S.; Shirakawa, K.; Yamazaki, N.
\newblock Optimal control problems governed by 1-{D}
  {K}obayashi--{W}arren--{C}arter type systems.
\newblock {\em Math. Control Relat. Fields}, {\bfseries 11}(2): 253--289, 2021.

\bibitem{MR4395725}
Antil, H.; Kubota, S.; Shirakawa, K.; Yamazaki, N.
\newblock Constrained optimization problems governed by {PDE} models of grain
  boundary motions.
\newblock {\em Adv. Nonlinear Anal.}, {\bfseries 11}(1): 1249--1286, 2022.

\bibitem{MR1752970}
Kobayashi, R.; Warren, J.~A.; Carter, W.~C.
\newblock A continuum model of grain boundaries.
\newblock {\em Phys. D}, {\bfseries 140}(1-2): 141--150, 2000.

\bibitem{MR1794359}
Kobayashi, R.; Warren, J.~A.; Carter, W.~C.
\newblock Grain boundary model and singular diffusivity.
\newblock In {\em Free boundary problems: theory and applications, {II}
  ({C}hiba, 1999)}, Vol.~14 of {\em GAKUTO Internat. Ser. Math. Sci. Appl.},
  pp.  283--294. Gakk\=otosho, Tokyo, 2000.

\bibitem{MR3268865}
Moll, S.; Shirakawa, K.
\newblock Existence of solutions to the {K}obayashi--{W}arren--{C}arter system.
\newblock {\em Calc. Var. Partial Differential Equations}, {\bfseries 51}(3-4):
  621--656, 2014.

\bibitem{MR3670006}
Moll, S.; Shirakawa, K.; Watanabe, H.
\newblock Energy dissipative solutions to the {K}obayashi--{W}arren--{C}arter
  system.
\newblock {\em Nonlinearity}, {\bfseries 30}(7): 2752--2784, 2017.

\bibitem{MO2018JIAM}
Okumura, M.
\newblock A stable and structure-preserving scheme for a non-local {A}llen--{C}ahn
  equation.
\newblock {\em Japan Journal of Industrial and Applied Mathematics}, {\bfseries
  35}: 1245--1281, 2018.

\bibitem{MO2020DCDS}
Okumura, M.; Furihata, D.
\newblock A structure-preserving scheme for the {A}llen--{C}ahn equation with a
  dynamic boundary condition.
\newblock {\em Discrete and Continuous Dynamical Systems}, {\bfseries 40}(8):
  4927--4960, 2020.

\bibitem{MR3082861}
Shirakawa, K.; Watanabe, H.
\newblock Energy-dissipative solution to a one-dimensional phase field model of
  grain boundary motion.
\newblock {\em Discrete Contin. Dyn. Syst. Ser. S}, {\bfseries 7}(1): 139--159,
  2014.

\bibitem{MR3203495}
Watanabe, H.; Shirakawa, K.
\newblock Qualitative properties of a one-dimensional phase-field system
  associated with grain boundary.
\newblock In {\em Nonlinear analysis in interdisciplinary
  sciences---modellings, theory and simulations}, Vol.~36 of {\em GAKUTO
  Internat. Ser. Math. Sci. Appl.}, pp.  301--328. Gakk\=otosho, Tokyo, 2013.

\bibitem{MR4177183}
Watanabe, H.; Shirakawa, K.
\newblock Energy-dissipation in a coupled system of {A}llen--{C}ahn-type
  equation and {K}obayashi--{W}arren--{C}arter-type model of grain boundary
  motion.
\newblock {\em Math. Methods Appl. Sci.}, {\bfseries 43}(17): 10138--10167,
  2020.

\bibitem{YOSHIKAWA2017394}
Yoshikawa, S.
\newblock Energy method for structure-preserving finite difference schemes and
  some properties of difference quotient.
\newblock {\em Journal of Computational and Applied Mathematics}, {\bfseries
  311}: 394--413, 2017.

\end{thebibliography}
%%%%%%%%%%%%%%%%%%%%%%%%%%%%%%%%%%%%%%%%%%%%%%%%%%%%%%%%%%%%%%%%%%%%%%%%%%%%%%%%%%%%%%%%%%%%%%%%%%%%

\end{document}